\documentclass[11pt]{amsart}

\title[Random regular graphs and the systole of a random surface]{Random regular graphs\\and the systole of a random surface}
\author{Bram Petri}
\address{Max Planck Institute for Mathematics, Bonn, Germany}
\email{brampetri@mpim-bonn.mpg.de}
\date{\today}
\thanks{Research supported by Swiss National Science Foundation grant number PP00P2\textunderscore 128557}
\keywords{Random Riemann surfaces, hyperbolic surfaces, Bely\v{\i} surfaces, Riemannian surfaces, random graphs, systole}
\subjclass[2010]{Primary: 57M50. Secondary: 53C22, 05C80}

\newtheorem{thm}{Theorem}[section]
\newtheorem{prp}[thm]{Proposition}

\newtheorem{lem}[thm]{Lemma}

\newtheorem*{thmA}{Theorem A}
\newtheorem*{thmB}{Theorem B}
\newtheorem*{thmC}{Theorem C}
\newtheorem*{thmD}{Theorem D}
\newtheorem*{thmA1}{Theorem A.1}
\newtheorem*{thmA2}{Theorem A.2}

\theoremstyle{definition}
\newtheorem{dff}{Definition}[section]

\newtheorem*{obs}{Observation}

\usepackage{amsmath,amsfonts,amssymb,float,graphicx,enumitem}
\usepackage[margin=1in]{geometry}

\newcommand{\Pro}[2]{\mathbb{P}_{#1}\left[#2\right]}
\newcommand{\ExV}[2]{\mathbb{E}_{#1}\left[#2\right]}
\newcommand{\verz}[2]{\left\{ #1 ; #2\right\}}
\newcommand{\rij}[3]{\left\{#1\right\}_{#2}^{#3}}

\newcommand{\abs}[1]{\left| #1 \right|}
\newcommand{\aant}[1]{\left| #1 \right|}
\newcommand{\tr}[1]{\mathrm{tr}\left(#1\right)}
\newcommand{\floor}[1]{\left\lfloor #1 \right\rfloor}
\newcommand{\ceil}[1]{\left\lceil #1 \right\rceil }
\newcommand{\sys}{\mathrm{sys}}
\newcommand{\inp}[2]{\left< #1, #2\right>}

\begin{document}

\begin{abstract}
We study the systole of a random surface, where by a random surface we mean a surface constructed by randomly gluing together an even number of triangles. We study two types of metrics on these surfaces, the first one coming from using ideal hyperbolic triangles and the second one using triangles that carry a given Riemannian metric. 

In the hyperbolic case we compute the limit of the expected value of the systole when the number of triangles tends to infinity (approximately $2.484$). We also determine the asymptotic probability distribution of the number of curves of any finite length. This turns out to be a Poisson distribution.

In the Riemannian case we give an upper bound to the limit supremum and a lower bound to the limit infimum of the expected value of the systole depending only on the metric on the triangle. We also show that this upper bound is sharp in the sense that there is a sequence of metrics for which the limit infimum comes arbitrarily close to the upper bound.

The main tool we use is random regular graphs. One of the difficulties in the proof of the limits is controlling the probability that short closed curves are separating. To do this, we first prove that the probability that a random cubic graph has a short separating circuit tends to $0$ as the number of vertices tends to infinity and show that this holds for circuits of a length up to $\log_2$ of the number of vertices.
\end{abstract}

\maketitle

\section{Introduction}

Given a surface (i.e. a $2$-dimensional manifold), one of the basic questions one can ask is, what kind of metrics can be defined on this surface? If we restrict the question to Riemannian metrics of constant curvature $-1$, this question leads to the study of moduli space and Teichm\"uller space. These spaces carry various metric structures themselves, one of those being the Weil-Petersson metric. In this metric, moduli space has finite volume. That makes it possible to speak of the probability that a Riemann surface has certain properties. In \cite{Mir}, Mirzakhani proved various asymptotic results about these probabilities. For example, she proved that the probability that a Riemann surface has a non-contractible closed curve of length less than $\varepsilon$ decays like $\varepsilon^2$ for $\varepsilon\rightarrow 0$, independantly of the topological type of the surface.

Another approach to random surfaces that one can take is a more combinatorial one: we start with $2N$ triangles that carry a certain metric and glue them together along their sides in a random fashion (the precise model will be explained in Section \ref{sec_model}). This gives a very different probability space. For one thing, this space is finite and also the topological type of the resulting surface is no longer fixed. 

So, the first question to consider is the resulting topology. Because we are gluing together an even number of triangles, a random surface has no boundary components. Because these surfaces are asymptotically almost surely connected, the genus completely determines its topology (sometimes we will choose to turn the corners of the triangles into cusps, but then through the Euler charateristic the genus tells us the number of cusps as well). It turns out that the expected value of the genus in this model grows like $N/2$. Various versions of this result were proved by Gamburd and Makover in \cite{GM}, Brooks and Makover in \cite{BM}, Pippenger and Schleich in \cite{PS}, Gamburd in \cite{Gam} and Dunfield and Thurston in \cite{DT}. In \cite{DT}, similar questions are also considered for random constructions of 3-manifolds.

In \cite{BM}, Brooks and Makover identify the triangles in this combinatorial model with ideal hyperbolic triangles that are glued together in a natural way. This means that the resulting surface will have cusps. It can however be compactified by adding points in the cusps. A theorem of Bely\v{\i} \cite{Bel} then implies that these compactified surfaces form a dense set in any moduli space of compact surfaces.

One of the nice properties of a surface constructed by gluing ideal triangles is that the lengths of curves on the can easily be computed from the combinatorial data of the gluing. While the metric on the compact version of the random surface is different from the one on the non-compact surface, a theorem of Brooks \cite{Bro} in combination with results from Brooks and Makover \cite{BM}, implies that these metrics are comparable outside disks around cusps and their images with probability tending to $1$ for $N\rightarrow\infty$. Using this, Brooks and Makover for example study the asymptotics of the probability distribution of the Cheeger constant, the shortest non-contractible geodesic (the \emph{systole}) and the diameter of the compact random surface. 

Instead of using ideal hyperbolic triangles one can also use equilateral Euclidean triangles. In \cite{GPY}, Guth, Parlier and Young prove that the probability that a surface constructed using equilateral Euclidean triangles can be decomposed into pairs of pants (three-holed spheres) with total boundary length less than $N^{7/6-\varepsilon}$ tends to $0$ for $N\rightarrow \infty$ for any $\varepsilon >0$. Because of the behaviour of the genus in the combinatorial model it is natural to compare this result with the behaviour of the \emph{total pants length} in the Weil-Petersson metric for $g\rightarrow\infty$. It turns out that this is very similar, they prove that the probability that the total pants length of a hyperbolic surface is less than $g^{7/6-\varepsilon}$ tends to $0$ for $g\rightarrow\infty$ for any $\varepsilon>0$.

In this article we will study the systole of a random surface. We will take the combinatorial approach and we will consider both the metric coming from identifying the triangles with ideal hyperbolic triangles and the metric coming from a Riemannian metric with suitable boundary conditions on the triangles. The main goal of this text is the study of the asymptotics of the expected value of the length of the systole in these two settings. 

For the hyperbolic setting we are able to compute the exact limit of this expected value, both in the non-compact and compact case. $\Omega_N$ will denote the probability space of random surfaces constructed with $2N$ triangles and $\sys_N:\Omega_N\rightarrow\mathbb{R}$ the random variable that measures the length of the systole of a random surface. To state the result we need to define the following set for every $k\in\mathbb{N}$:
\begin{equation*}
A_k = \left.\verz{w\text{ word in }\left(\begin{array}{cc} 1 & 1 \\ 0 & 1 \end{array}\right)\text{ and }\left(\begin{array}{cc} 1 & 0 \\ 1 & 1 \end{array}\right)}{\tr{w}=k}\middle/ \sim\right.
\end{equation*}
where $w\sim w'$ if either $w$ is a cyclic permutation of $w'$ as a word or $w=(w')^t$ as a matrix. We will prove the following theorem:
\begin{thmA}
Both in the compact and non-compact hyperbolic setting we have:
$$
\lim_{N\rightarrow\infty}\ExV{}{\sys_N} = \sum\limits_{k=3}^\infty 2\left(\prod\limits_{[w]\in\bigcup\limits_{i=3}^{k-1}A_i} \exp\left(-\frac{\aant{[w]}}{2\abs{w}}\right)\right) \left(1-\prod\limits_{[w]\in A_k}\exp\left(-\frac{\aant{[w]}}{2\abs{w}}\right)\right)\cosh^{-1}\left(\frac{k}{2}\right)
$$
\end{thmA}
\noindent To prove that the result in the compact setting is the same as in the non-compact setting we will use the fact that the probability that a random hyperbolic surface has large cusps tends to $1$. This fact was proved in \cite{BM}. However, we also need to know how fast this probability tends to $1$, so we will prove a quantitative version of their theorem, using similar ideas.

\noindent Even though the expression on the right hand side looks rather abstract, it is possible to approximate it. We can determine the sets $A_k$ for $k$ up to some prescribed trace. Together with an approximation for the remainder term this gives the following approximation for the limit (see (\ref{eq_numHyp}) in section \ref{sec_numval}):
\begin{equation*}
2.48432 \leq \lim_{N\rightarrow\infty}\ExV{}{\sys_N} \leq 2.48434
\end{equation*}

As a byproduct of the proof of Theorem A, we also obtain the asymptotic probability distribution of the number of closed curves of any finite length. Before we can make this statement precise, we need to introduce some notation. First of all, by tracing a curve on the surface and recording whether the curve turns right or left on every triangle it passes, we get a word (technically only defined up to the equivalence of words mentioned above) in the letters $L$ and $R$. We associate the matrices 
\begin{equation*}
\left(\begin{array}{cc} 1 & 1 \\ 0 & 1 \end{array}\right)\text{ and }\left(\begin{array}{cc} 1 & 0 \\ 1 & 1 \end{array}\right)
\end{equation*}
to the letters $L$ and $R$ respectively. Given an equivalence class of words $[w]$ corresponding to a homotopy class of curves $\gamma$, we say that $\gamma$ carries $[w]$.

Now for every equivalence class $[w]$ of words in $L$ and $R$ we define the random variables $Z_{N,[w]}:\Omega_N\rightarrow\mathbb{N}$ by:
\begin{equation*}
Z_{N,[w]}(\omega) = \aant{\verz{\text{homotopy classes of curves }\gamma\text{ on the surface corresponding to }\omega}{\gamma\text{ carries }[w]}}
\end{equation*}
The relation between these random variables and the length spectrum is that if a curve $\gamma$ carries a word $[w]$ in $L$ and $R$, then the unique geodesic in its homotopy class has length:
\begin{equation*}
\ell(\gamma) = 2\cosh^{-1}\left( \frac{\tr{w}}{2}\right)
\end{equation*}
where the matrix $w$ is obtained by multiplying the matrices corresponding to $L$ and $R$ as dictated by $[w]$.

The theorem is following:
\begin{thmB}
Let $W$ be a finite set of equivalence classes of words. Then we have: 
$$
Z_{N,[w]}\rightarrow Z_{[w]} \text{ in distribution as }N\rightarrow\infty
$$
for all $[w]\in W$, where:
\begin{itemize}[leftmargin=0.2in]
\item $Z_{[w]}:\mathbb{N}\rightarrow\mathbb{N}$ is a Poisson distributed random variable with mean $\lambda_{[w]}=\frac{\aant{[w]}}{2\abs{w}}$ for all $w\in W$.
\item The random variables $Z_{[w]}$ and $Z_{[w']}$ are independent for all $[w],[w']\in W$ with $[w]\neq [w']$.
\end{itemize}
\end{thmB}
This should be interpreted as the limit of any finite part of the length spectrum. Given any fixed set of fixed lengths $\ell_1,\ell_2,\ldots,\ell_k \in\mathbb{R}$ and any $n_1,\ldots,n_k\in\mathbb{N}$, Theorem B gives the asymptotic probability that a random surface has $n_i$ closed curves of length $\ell_i$ for all $i=1,\ldots,k$. Again using the results of Brooks \cite{Bro} and Brooks and Makover \cite{BM}, this reasoning also applies to the compactified case.

For the Riemannian model we are only able to give a lower bound on the $\liminf$ and an upper bound for the $\limsup$ of the expected value of the length of the systole. This has to do with the fact that in this model there is not such a simple way of computing the exact length of a curve in terms of the combinatorial data of the surface.

For this model we need to fix a Riemannian metric on our triangle. Let us first fix the standard $2$-simplex:
\begin{equation*}
\Delta = \verz{t_1e_1+t_2e_2+t_3e_3}{(t_1,t_2,t_3)\in [0,1]^3,\; t_1+t_2+t_3=0}
\end{equation*}
where $\{e_1,e_2,e_3\}$ is the standard basis of $\mathbb{R}^3$. The class of metrics on $\Delta$ we will study is the class of Riemannian metrics that are what we will call symmetric in the sides. Symmetric in the sides will be a set of conditions that make sure that the metric on the resulting surface is sufficiently regular (see section \ref{sec_GeomRiem} for the definition). 

Given such a metric $d:\Delta\times\Delta\rightarrow [0,\infty)$ we define:
\begin{equation*}
m_1(d) = \min\verz{d(s,s')}{s,s' \text{ opposite sides of a gluing of two copies of }(\Delta,d)\text{ along one side}}
\end{equation*}
and:
\begin{equation*}
m_2(d) = \max\verz{d\left(\frac{e_i+e_j}{2},\frac{e_k+e_l}{2}\right)}{i,j,k,l\in\{1,2,3\},\;i\neq j,\;k\neq l}
\end{equation*}
We are able to prove the following:
\begin{thmC}
Any Riemannian model satisfies:
$$
m_1(d)\leq \liminf_{N\rightarrow\infty}\ExV{}{\sys_N}
$$
and
$$
\limsup_{N\rightarrow\infty}\ExV{}{\sys_N} \leq m_2(d)\sum_{k=2}^\infty k\left(e^{-\sum\limits_{j=1}^{k-1}\frac{2^{j-1}-1}{j}} -e^{-\sum\limits_{j=1}^k\frac{2^{j-1}-1}{j}}\right)
$$
\end{thmC}
\noindent We can also compute the sum expression in the upper bound. We get (see (\ref{eq_numRiem}) in section \ref{sec_RiemSys}):
\begin{equation*}
\limsup_{N\rightarrow\infty}\ExV{}{\sys} \leq 2.87038 \cdot m_2(d)
\end{equation*}

\noindent In the equilateral Euclidean case we have $m_1(d)=1$ and $m_2(d)=\frac{1}{2}$, so we get (see also (\ref{eq_numEucl1}) and (\ref{eq_numEucl2}) in section \ref{sec_RiemSys}):
\begin{equation*}
1\leq \liminf_{N\rightarrow\infty}\ExV{}{\sys_N} 
\end{equation*}
and:
\begin{equation*}
\limsup_{N\rightarrow\infty}\ExV{}{\sys_N} \leq 1.43519
\end{equation*}
In the Euclidean case there are steps in the approximation where we make a clear over-estimation of certain lengths, hence we do not expect this bound to be optimal.

However, in general the upper bound of this theorem is sharp in the sense that we can define a sequence of metrics with prescribed $m_2$ that comes arbitrarily close to the upper bound. The idea behind these metrics is to construct large bumps on the triangle that force short curves to trace a fixed path on every triangle (for more explanation see section \ref{sec_Sharpness}).

Furthermore, given a triangulated surface, we can construct a graph dual to the triangulation. The vertices of this graph are the triangles of the triangulation and two vertices are connected by $0$, $1$, $2$ or $3$ edges if their corresponding triangles share $0$, $1$, $2$ or $3$ sides respectively. The orientation on the surface also gives rise to a cyclic orientation of the edges emanating from the vertices. This oriented cubic graph (sometimes called a ribbon graph) combined with the metric on the triangles completely determine the surface. So all the information we want on the surface is fully encoded in the graph. This allows us to apply results from the theory of random regular graphs. In particular, we apply a theorem from Bollob\'as \cite{Bol1} that states the asymptotics of the probability distribution of the number of circuits of a fixed length on a random regular graph.

In some cases we will need to prove certain bounds and asymptotics about probabilities related to random regular graphs ourselves. A result that might be of independent interest in this context is the following:
\begin{thmD}
Let $C\in (0,1)$. We have:
$$
\lim\limits_{N\rightarrow\infty} \Pro{}{\substack{\text{A random cubic graph on }2N \text{ vertices contains}\\ \text{a separating circuit of }\leq C\log_2(N)\text{ edges }}} = 0
$$
\end{thmD}

Finally, we would like to remark a few things:
\begin{itemize}[leftmargin=0.2in]
\item[-] From Theorems A and B it follows that in all settings the expected value of the systole is a bounded sequence. By Markov's inequality this implies that there exists a constant $R\in (0,\infty)$ depending only on the metric on the triangle such that:
$$
\Pro{}{\sys_N > x} \leq \frac{R}{x}
$$
for all $x\in (0,\infty)$ and $N\in\mathbb{N}$. So this result can be seen as a result on the decay of the tail opposite to the one that Mirzakhani exhibited in the Weil-Petersson setting.
\item[-] Using exactly the same arguments as in this article the limits of the expected values the first shortest $k$ curves can be computed and estimated up to any finite $k$.
\item[-] Some of the results in this article contradict the results of \cite{MM}. It is our belief that there are some gaps in the latter.
\end{itemize}

The organization of this article is as follows:
\begin{itemize}[leftmargin=0.2in]
\item[-] The first section explains the two models and the connection between random surfaces and cubic ribbon graphs,
\item[-] In the second section we study separating circuits and show that a random cubic ribbon graph does not have short separating circuits (Theorem D).
\item[-] In section three we prove a set of upper bounds for probabilities related to circuits on cubic graphs.
\item[-] In the last section we prove the main results (Theorems A, B and C).
\end{itemize} 

\subsection*{Acknowledgement}
The author would like to thank his doctoral advisor Hugo Parlier for all his support and for reading multiple drafts of this text. The author is also grateful to the organizers of the `Special Program on Teichm\"uller Theory' at the Erwin Schr\"odinger Institute in Vienna, where he stayed during the month of March 2013 and proved parts of the results presented here.

\section{The models}\label{sec_model}

\subsection{Random surfaces and partitions}
As mentioned in the introduction, topologically the two models for random surfaces we consider are the same. We will explain the topological model in the first two sections and in Sections \ref{sec_GeomHyp} and \ref{sec_GeomRiem} we will explain the geometry.

The construction starts by taking an even number of triangles (so: topological $2$-simplices) with cyclicly ordered sides\footnote{We will use a lot of graphs, so to avoid confusion we will speak about sides and corners instead of edges and vertices when speaking about triangles.}. Then the sides of these triangles are randomly partitioned into pairs. After that, we glue these triangles together along the sides that are paired. We do this in such a way that corners are glued to corners. We also require that the surface is oriented in such a way that the orientation corresponds to the ordering of the sides (via the right hand rule). These conditions ensure that the partition uniquely defines the surface. Because we use an even number of triangles the surface has no boundary components.

To describe the probability space, we label the triangles with the numbers $1,2,\ldots,2N$ and their sides with the numbers $1,2,\ldots,6N$, where the sides $1,2,3$ belong to triangle $1$, sides $4,5,6$ belong to triangle $2$ et cetera. We do this in such a way that the cyclic ordering of the sides of triangle $i$ corresponds to $(3i-2,3i-1,3i)$. So a random surface in our model corresponds to a partition of the set $\{1,2,\ldots 6N\}$ into pairs. In some computations we will want to make a distinction between two different orders of picking the pairs, so we will be interested in the following two sets:
\begin{equation*}
\Omega_N=\left\{\text{Partitions of }\{1,2,\ldots 6N\}\text{ into pairs} \right\}
\end{equation*}
and
\begin{equation*}
\Omega_N^o=\left\{\text{Ordered partitions of }\{1,2,\ldots 6N\}\text{ into pairs} \right\}
\end{equation*}
We have:
\begin{equation*}
\aant{\Omega_N^o} = \frac{(6N)!}{2^{3N}}
\end{equation*}
and
\begin{equation*}
\aant{\Omega_N} = \frac{\aant{\Omega_N^o}}{(3N)!} = (6N-1)(6N-3)\cdots 3\cdot 1
\end{equation*}

We can turn $\Omega_N$ and $\Omega_N^o$ into probability spaces by giving every element equal probability. So, if $A_N\subseteq\Omega_N$ and $A_N^o\subseteq\Omega^o_N$ then:
\begin{equation*}
\Pro{}{A_N}=\frac{\aant{A_N}}{\aant{\Omega_N}}
\end{equation*}
and
\begin{equation*}
\Pro{}{A_N^o}=\frac{\aant{A_N^o}}{\aant{\Omega_N^o}}
\end{equation*}
Note that in principle the two equations above define two different infinite sequences of probability measures. However, to prevent double labelling we will not label the probability measures.

Also, the order in which we pick the pairs of the partition has no influence on the surface. This means that in $\Omega_N^o$ we just get $(3N)!$ copies of every surface in $\Omega_N$. So the probability of events only depending on properties of the surface does not depend on whether we use $\Omega_N^o$ or $\Omega_N$ in the computation.

\subsection{Cubic ribbon graphs}
These random surfaces can also be described with so-called cubic ribbon graphs. `Cubic' means that every vertex in the graph has degree $3$, sometimes these graphs are also called trivalent. We will not make a distinction between graphs and multigraphs, i.e. a graph can have loops and multiple edges. `Ribbon graph' means that the graph is equipped with an orientation, where an orientation is a cyclic order at every vertex of the edges emanating from this vertex. We will sometimes depict the orientation at a vertex with an arrow as in the picture below.
\begin{figure}[H] 
\begin{center} 
\includegraphics[scale=1]{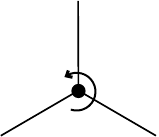} 
\caption{A vertex of a cubic ribbon graph.}
\label{pic4}
\end{center}
\end{figure}
\noindent This orientation gives a notion of turning left or right when traversing a vertex on the graph.

The number of vertices of a cubic ribbon graph $\Gamma = (V,E)$, or in fact of any cubic graph, must be even. This follows from the fact that $3\aant{V} = 2\aant{E}$ for cubic graphs. Hence we can assume that our graph has $2N$ vertices for some $N\in\mathbb{N}$.

Cubic ribbon graphs with $2N$ vertices can also be described by partitions  of the set $\{1,2,\ldots,6N\}$. We label the vertices with the numbers $1,2,\ldots 2N$ and the half-edges emanating from the vertices $1,2,\ldots,6N$, in such a way that the cyclic order of the half-edges at the vertex $i$ corresponds to $(3i-2,3i-1,3i)$ for all $1\leq i\leq 2N$, as in figure \ref{fig_labvert} below. 
\begin{figure}[H] \label{fig_labvert}
\begin{center} 
\includegraphics[scale=1]{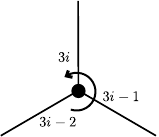} 
\caption{Vertex $i$ of a labeled cubic ribbon graph.}
\label{pic3}
\end{center}
\end{figure}

A partition $\omega$ then corresponds to a cubic ribbon graph $\Gamma(\omega)$ by connecting two half-edges if they form a pair in the partition.

It is not difficult to see that we obtain all possible cubic ribbon graphs like this. In fact, because of the labeling, we even obtain many isomorphic copies of every graph. This means that we can interpret $\Omega_N^o$ and $\Omega_N$ with the probability measures induced by counting measures as probability spaces for random cubic graphs as well. We note that also the properties of the graph do not depend on the ordering of the pairs. So events depending only on graph theoretic properties, like those depending only on the properties of the resulting surface, have the same probability in $\Omega_N^o$ and $\Omega_N$. The model for random cubic graphs that we obtain is a well known model for random cubic graphs (see for example \cite{Bol2} or \cite{Wor}) that is sometimes called the \emph{Configuration model}.

The cubic ribbon graph corresponding to a partition is dual to the triangulation of the surface corresponding to this same partition. That is to say, we can obtain the graph by adding a vertex to every triangle of the triangulation and then adding an edge between two vertices if the corresponding triangles share a side. If we label the graph in the same way as the triangles on the surface, we obtain the cubic ribbon graph corresponding to the partition.

The construction above describes an embedding of the random cubic ribbon graph into the random surface corresponding to the same partition. Figure \ref{pic1} shows what the graph looks like on the surface:
\begin{figure}[H] 
\begin{center} 
\includegraphics[scale=1]{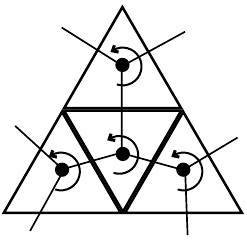} 
\caption{A part of a triangulation and its dual graph.}
\label{pic1}
\end{center}
\end{figure}

\subsubsection{Recovering the topology from the graph}

In genreal we will actually reason the other way around: we will look at random cubic ribbon graphs and try to derive properties of the corresponding surfaces from the properties of the graph. The first natural problem is recovering the topology of the surface from properties of the graph. This can be done using the Euler characteristic. So, a random graph corresponds to a triangulated surface and we want to recover the number of triangles, sides and corners from the graph. The number of triangles is the number of vertices of the graph, which is given and equal to $2N$. Likewise the number of sides can be easily recovered, this is equal to the number of edges and from the earlier mentioned fact that $3\aant{V} = 2\aant{E}$ we get that this is equal to $3N$. The difficult thing to recover is the number of corners. Around a corner the surface looks like the picture below:

\begin{figure}[H] 
\begin{center} 
\includegraphics[scale=1]{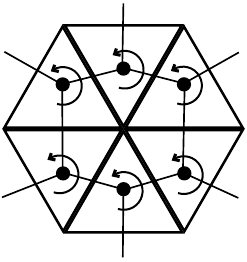} 
\caption{A part of a triangulation around a corner.}
\label{pic2}
\end{center}
\end{figure}

The sides of all the triangles around the corner have to be ordered consistently, because these orders have to correspond to the orientation on the surface. This means that if we walk along the circuit on the graph around this corner on the surface, we turn in the same direction at every vertex of the graph with respect to the orientation on the graph. Whether this direction is a constant `left' or `right' depends on the direction in which we traverse the circuit. So we can conclude that corners correspond to left hand turn circuits on the graph.

Using this fact, one can estimate the expected value of the genus. This was done by Gamburd and Makover in \cite{GM}, Brooks and Makover in \cite{BM}, Pippenger and Schleich in \cite{PS} and Gamburd in \cite{Gam}. We state the version of Brooks and Makover here. For $\omega\in\Omega_N$ we will denote the genus of the surface corresponding to $\omega$ by $g_N(\omega)$.
\begin{thm}\label{thm_genus} \cite{BM} There exist constants $C_1,C_2\in (0,\infty)$ such that the expected value of the genus satisfies:
$$
1+\frac{N}{2}-\left(C_1+\frac{3}{4}\log(N)\right) \leq \ExV{}{g_N} \leq 1+\frac{N}{2}-\left(C_2+\frac{1}{2}\log(N)\right) 
$$
\end{thm}

\subsubsection{Curves and circuits}

An important fact is that curves on the surface are homotopic to curves on the graph, where with curves on the graph we mean curves on the graph embbeded into the surface in the way described above. From hereon we will often think of the graph as embedded into the surface in this way without mentioning it. This homotopy can be realized as follows: 
we divide the curve in pieces, such that every piece corresponds to the curve entering and leaving one specific triangle exactly once. We then homotope these entry and exit points to the midpoints of the corresponding sides triangle (where the edge of the graph cuts the side). After that we homotope the curve onto the two half-edges connecting the sides to the vertex corresponding to the triangle in every piece of the curve. Figure \ref{pic24} shows an example:
\begin{figure}[H] 
\begin{center} 
\includegraphics[scale=1]{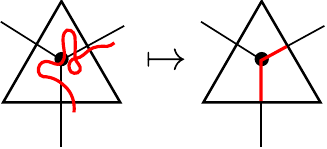} 
\caption{Homotopy.}
\label{pic24}
\end{center}
\end{figure}
 This means that we can use results on curves on random graphs for the study of curves on random surfaces, which is the main reason we make the connection with graphs.

There is one problem: when we homotope a curve on the surface to a curve on the graph, it does not necessarily maintain its properties. For instance, a simple curve on the surface does not always homotope to a simple curve on the graph. However, the following proposition tells us that the homotopic image of a non null homotopic curve does always contain a circuit.

\begin{prp}\label{prp_circuits} Let $\gamma$ be a non null homotopic curve on the surface corresponding to a partition $\omega\in\Omega_N$ and $\gamma'$ a homotopic image of $\gamma$ on $\Gamma(\omega)$ then $\gamma'$ contains a non null homotopic circuit.
\end{prp}

\begin{proof} We argue by contradiction. Suppose that $\gamma'$ contains no homotopocially non-trivial circuits. That means that we can contract all circuits and $\gamma'$ is homotopic to a subtree of $\Gamma(\omega)$. Trees are homotopically trivial, which concludes the proof. \end{proof}

\subsection{Geometry of the hyperbolic model}\label{sec_GeomHyp}

\subsubsection{Ideal triangulations}

The first way of putting a metric on the surface we consider uses ideal hyperbolic triangles as in \cite{BM}. Let $\mathbb{H}^2=\verz{z\in\mathbb{C}}{\mathrm{Im}(z)>0}$ be the upper half plane model of the hyperbolic plane. We will use $2N$ isometric copies of the triangle $T\subset\mathbb{H}^2$, given by the vertices $0$,$1$ and $\infty$, shown in the picture below:
\begin{figure}[H] 
\begin{center} 
\includegraphics[scale=1]{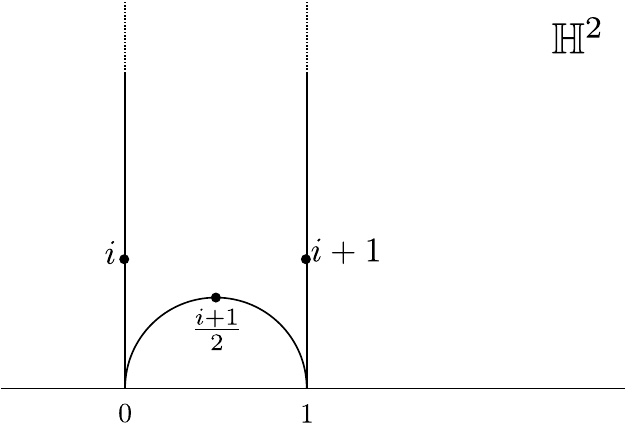} 
\caption{The triangle $T$}
\label{pic6}
\end{center}
\end{figure}

Note that in this model the corners of the triangles turn into punctures. To define the gluing we assign a midpoint to every edge of $T$: $i$, $i+1$ and $\frac{i+1}{2}$ and add the condition that midpoints are glued to midpoints (this comes down to $0$ shearing).

If $\omega\in\Omega_N$, the surface obtained by gluing together isometric copies of $T$ will be denoted $S_O(\omega)$.

We will be interested in curves on the surface and in particular in their length. Given a homotopically non-trivial closed curve $\gamma$, it follows from the fact that the metric on our random surface is hyperbolic that there exists a unique geodesic homotopic to $\gamma$ (see for instance Theorem 1.6.6 in \cite{Bus}). We will sometimes call this geodesic the \emph{geodesic representative} of $\gamma$. In the case of random surface constructed with ideal triangles the length of this geodesic can be computed from the combinatorial data of the gluing. In what follows, we will explain how this computation can be done.

We will write $S_O(\omega)=\mathbb{H}^2/G$, where $G$ is a discrete subgroup of $\mathrm{PSL}_2(\mathbb{R})$. A closed curve $\gamma$ corresponds to a conjugacy class of $G$ (the set of elements in $\pi_1(S_O,x)\simeq G$ freely homotopic to $\gamma$ forms a conjugacy class). Given a M\"obius transformation $g$ in this class, the geodesic in $\mathbb{H}^2$ on which the translation distance of $g$ is realized is called the axis $\alpha_g\subset\mathbb{H}^2$  of $g$. $\alpha_g$ projects to the unique geodesic homotopic to $\gamma$. The length of this geodesic is exactly the distance between $p$ and $g(p)$ for any point $p\in\alpha_g$. This in turn is equal to the translation length of $g$.

So, given a closed curve $\gamma$ on $S_O(\omega)=\mathbb{H}^2/G$, we need to find a corresponding group element $g\in G$. We pick a starting point on the curve $\gamma$, this point we identify with a point $p$ in the ideal triangle $T\subset\mathbb{H}^2$. After this starting point $\gamma$ travels over a sequence of triangles and eventually enters the triangle containing the starting point again and returns to the starting point. We copy this sequence of triangles into the hyperbolic plane, where we keep track of whether $\gamma$ turns right or left on the triangle. So every time $\gamma$ turns right between two triangles, we glue the new triangle on the right hand side (with respect to the direction of travel of $\gamma$) of the triangle before and when $\gamma$ turns left we glue it on the left hand side. It could also happen that $\gamma$ visits the same triangle twice, but also in that case we will glue a new triangle to the domain. We do this until $\gamma$ returns to the initial triangle for the last time. 

Figure \ref{pic25} shows an example: we have (topologically) drawn a piece of a surface $S_O(\omega)$ and a curve $\gamma$ on it. The small lines running through some of the sides are meant to indicate that some sides are identified on $S_O(\omega)$ (those with equal numbers of lines). The curve $\gamma$ starts at the triangle $A$, then goes through the triangles $B$, $C$, $B$ again and then $C$ again before returning to triangle $A$. If we apply the procedure described above we get a sequence of triangles $A,B,C,B',C',A'$ in the hyperbolic plane. $\gamma$ lifts to a segment running between a lift of $p$ on the triangle $A$ and a point on $A'$, which will be the image of $p$ under the M\"obius transformation $g$ we are looking for.

\begin{figure}[H] 
\begin{center} 
\includegraphics[scale=0.9]{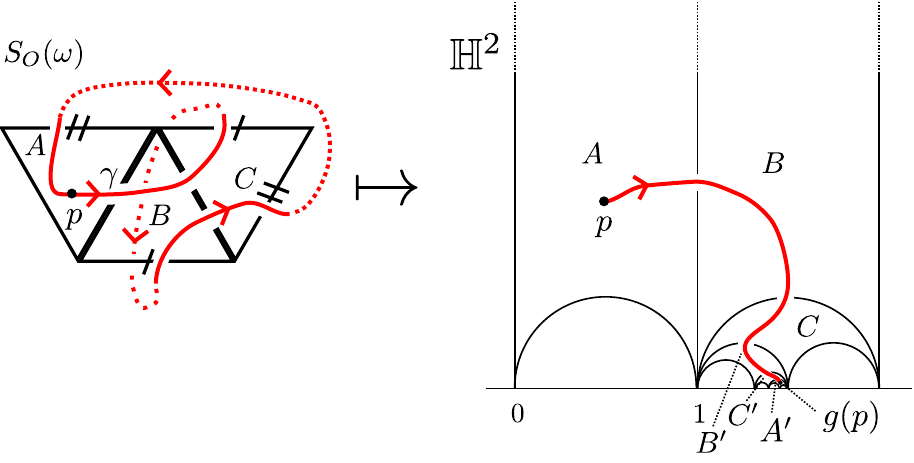} 
\caption{Turning a curve into a group element.}
\label{pic25}
\end{center}
\end{figure}

So now we need to find the M\"obius transformation $g$. This is the M\"obius transformation that (in our example) identifies the triangle $A$ with the triangle $A'$. We can construct this from the two M\"obius transformations $g_L:\mathbb{H}^2\rightarrow\mathbb{H}^2$ and ${g_R:\mathbb{H}^2\rightarrow\mathbb{H}^2}$, determined by the matrices:
\begin{equation*}
L=\left(
\begin{array}{cc}
1 & 1 \\
0 & 1
\end{array}
\right)
\text{ and }
R=\left(
\begin{array}{cc}
1 & 0 \\
1 & 1
\end{array}
\right)
\end{equation*}
respectively. So we have:
\begin{equation*}
g_L(z)=z+1\text{ and }g_R(z)=\frac{z}{z+1}
\end{equation*}
for all $z\in\mathbb{H}^2$. This means that:
\begin{equation*}
g_L([0,\infty])=[1,\infty],\; g_R([0,\infty])=[0,1],\; {g_L(i)=i+1} \text{ and } g_R(i)=\frac{i+1}{2}
\end{equation*}
where $[a,b]$ for $a,b\in\mathbb{H}^2\cup\partial\mathbb{H}^2$ denotes the geodesic (segment) between $a$ and $b$. 

When we go back to our example we see that $g_L(A)=B$ and $g_R(B)=C$ and so forth. This means that the group element we are looking for in the example is: $g=g_Lg_Rg_Rg_L$.

So, in general traveling along a curve $\gamma$ on the surface corresponds to a transformation composed of $g_L$s and $g_R$s, where every time we take a `left turn' we compose with a $g_L$ and every time we take a `right turn' we compose with a $g_R$. Furthermore, the length of the unique geodesic homotopic to $\gamma$ is the translation length of this composition of $g_L$s and $g_R$.

Suppose the word $W$ in $g_L$ and $g_R$ corresponding to a curve $\gamma$ is given by ${W=W_1W_2\cdots W_k}$ with $W_i\in\{g_L,g_R\}$ and $k\in\mathbb{N}$. Then standard hyperbolic geometry tells us that the translation length of $W$ is (see for instance \S\ 7.34 in \cite{Bea}):
\begin{equation*}
\ell_{\mathbb{H}}(\gamma)= 2\cosh^{-1}\left(\frac{\tr{w_1w_2\cdots w_k}}{2}\right)
\end{equation*}
where:
\begin{equation*}
w_i=\left\{
\begin{array}{ll}
L & \text{if }W_i=g_L \\
R & \text{if } W_i=g_R
\end{array}
\right.
\end{equation*}
for all $1\leq i\leq k$.

We note again that all this data can also be found in the graph. On a curve the orientation at a vertex tells us whether we are turning right or left. So we get the word corresponding to a curve by picking up a letter $L$ or $R$ at every vertex.

\begin{prp}\label{prp_circuits2}
Let $\omega\in\Omega_N$. Suppose $\gamma$ is a non null homotopic curve on $S_O(\omega)$ and $\gamma'$ the curve on $\Gamma(\omega)$ corresponding to the geodesic representative of $\gamma$. Then $\gamma'$ contains a homotopically non-trivial circuit $\gamma''$ with $\ell(\gamma'')\leq\ell(\gamma')\leq\ell(\gamma)$.
\end{prp}

\begin{proof}
Proposition \ref{prp_circuits} tells us that $\gamma'$ contains a homotopically non-trivial circuit $\gamma''$. We will prove that $\ell_{\mathbb{H}}(\gamma'')\leq\ell_{\mathbb{H}}(\gamma')$.

Because $\gamma''\subset\gamma'$, the word in $L$ and $R$ on $\gamma'$ can be obtained by inserting letters into the word on $\gamma''$. So, suppose the word in $L$ and $R$ on $\gamma''$ is $w=w_1w_2\cdots w_k$ and the word on $\gamma'$ is $v=v_1\cdots v_l$. Then  we have $k\leq l$ and there are $1\leq i_1<i_2<\ldots <i_k\leq l$ such that $w_j=v_{i_j}$ for all $1\leq j\leq k$. 

To prove that $\ell_{\mathbb{H}}(\gamma'')\leq \ell_{\mathbb{H}}(\gamma')$ we will prove that $\tr{w}\leq\tr{v}$ and to prove this we will prove that if we add letters to a word, the trace of the corresponding matrix increases (the rest will then follow by induction on the number of letters in the word).

So, let $w=w_1w_2\cdots w_n$ with $w_i\in\{L,R\}$ for $1\leq i\leq n$ and $w'=w_1\cdots w_i x w_{i+1}\cdots w_n$ with $x\in\{L,R\}$ and $1\leq i\leq n$. We have:
\begin{align*}
\tr{w}	& = \tr{w_1\cdots w_iw_{i+1}\cdots w_n} \\
		& = \tr{w_{i+1}\cdots w_nw_1\cdots w_i} 
\end{align*}
Likewise, we have:
\begin{equation*}
\tr{w'} = \tr{w_{i+1}\cdots w_nw_1\cdots w_ix} 
\end{equation*}
Write:
\begin{equation*}
w_{i+1}\cdots w_nw_1\cdots w_i=
\left(
\begin{array}{cc}
a_{11} & a_{12} \\
a_{21} & a_{22}
\end{array}
\right)
\end{equation*}
Because $w_i\in\{L,R\}$ for $1\leq i\leq n$ we have $a_{11},a_{12},a_{21},a_{22}\geq 0$. So:
\begin{equation*}
\tr{w} = a_{11}+a_{22}
\end{equation*}
and:
\begin{align*}
\tr{w'} & = 
\left\{
\begin{array}{ll}
a_{11}+a_{22}+a_{21} & \text{if }x=L \\
a_{11}+a_{22}+a_{12} & \text{if }x=R \\
\end{array}
\right. \\
& \geq \tr{w}
\end{align*}
\end{proof}

\subsubsection{Compactifcation}

We can also conformally compactify the surfaces we obtain by adding points in the cusps. That is, if $\omega\in\Omega_N$  then there is a unique closed Riemann surface $S_C(\omega)$ with a set of points ${\{p_1,\ldots p_n\}\subset S_C(\omega)}$ such that:
\begin{equation*}
S_O(\omega) \simeq S_C(\omega) \backslash \{p_1,\ldots p_n\}
\end{equation*}
conformally.

It follows from a theorem from Bely\v{\i} \cite{Bel} that the set $\verz{S_C(\omega)}{\omega\in \bigcup\limits_{N=1}^\infty\Omega_N}$ is dense in the set of all Riemann Surfaces.

The problem is that the hyperbolic geometry of $S_C(\omega)$ is difficult to deduce from the partition $\omega$ (as opposed to the geometry of $S_O(\omega)$). However, a theorem from Brooks \cite{Bro} tells us that the geometries of $S_C(\omega)$ and $S_O(\omega)$ are close if the cusps of $S_C(\omega)$ are `big enough'. This idea is formalised by the following definition:
\begin{dff}
Let $\omega\in\Omega_N$, $n\in\mathbb{N}$ such that $S_O(\omega)$ has cusps $\rij{C_i}{i=1}{n}$ and $L\in (0,\infty)$. Then $S_O(\omega)$ is said to have cusp length $\geq L$ if there exists a set of horocycles $\rij{h_i}{i=1}{n}\subset S_O(\omega)$ such that:
\begin{itemize}[leftmargin=0.2in]
\item $h_i$ is a horocycle around $C_i$ for $1\leq i\leq n$.
\item $\ell_{\mathbb{H}} (h_i) \geq L$ for $1\leq i\leq n$.
\item $h_i\cap h_j=\emptyset$ for $1\leq i\neq j \leq n$.
\end{itemize}
\end{dff}

We also need some notation for disks. If $\omega\in\Omega_N$, $r\in (0,\infty)$ and $p\in S_C(\omega)$ then ${B_r(p)\subset S_C(\omega)}$ denotes the hyperbolic disk of radius $r$ around $p$. If $C$ is one of the cusps of $S_O(\omega)$ then ${B_r(C)\subset S_O(\omega)}$ denotes the neighborhood of $C$ bounded by the horocycle of length $r$ around $C$.

Now we can state the comparison theorem:

\begin{thm} \cite{Bro} For every $\varepsilon >0$ there exist $L\in (0,\infty)$ and $r\in (0,\infty)$ such that: If $\omega\in\Omega_N$ such that:
\begin{itemize}[leftmargin=0.2in]
\item $S_C(\omega)$ carries a hyperbolic metric $ds_{S_C(\omega)}^2$
\item $S_O(\omega)$ carries a hyperbolic metric $ds_{S_O(\omega)}^2$
\item $S_O(\omega)$ has cusps $\rij{C_i}{i=1}{n}$ and cusp length $\geq L$.
\end{itemize}
Then: outside $\bigcup\limits_{i=1}^n B_L(C_i)$ and $\bigcup\limits_{i=1}^n B_r(p_i)$ we have:
$$
\frac{1}{1+\varepsilon} ds_{S_O(\omega)}^2 \leq ds_{S_C(\omega)}^2 \leq (1+\varepsilon)ds_{S_O(\omega)}^2
$$
\end{thm}

Because we are mainly interested in the compactified case we shall slightly adapt the definition of a systole. We will only consider curves on $S_O(\omega)$ to be homotopically non trivial if they are non trivial on $S_C(\omega)$ as well. It will turn out later that this does not make a difference for the asymptotics because the probability that there are short separating closed curves on $S_O$ will tend to $0$ if we let the number of triangles grow.

Furthermore, it follows from propositions \ref{prp_circuits} and \ref{prp_circuits2} that the systole, both in the compact and non-compact hyperbolic model, corresponds to a circuit.

Finally we will use the following lemma by Brooks:
\begin{lem}\label{lem_brooks}\cite{Bro} For $L\in (0,\infty)$ sufficiently large, there is a constant $\delta(L)$ with the following property: Let $\omega\in\Omega_N$ such that $S_O(\omega)$ has cusp length $\geq L$. Then for every geodesic $\gamma$ in $S_C(\omega)$ there is a geodesic $\gamma'$ in $S_O(\omega)$ such that the image of $\gamma'$ is homotopic to $\gamma$, and:
$$
\ell(\gamma) \leq \ell(\gamma') \leq (1+\delta(L))\ell(\gamma)
$$
Furthermore, $\delta(L)\rightarrow 0$ as $L\rightarrow\infty$.
\end{lem}

\subsection{Geometry of the Riemannian model}\label{sec_GeomRiem}

The second model we will study is actually a collection of models. The idea is just to assume that we are given a fixed triangle with a metric on it. We will however make some assumptions on this metric. For the gluing we need the metric to be symmetric in a certain sense. Furthermore, because we need to apply Gromov's systolic inequality for surfaces at some point, we need the metric on the surface to be Riemannian up to a finite set of points, which means that we need to make some smoothness assumptions. One of these models is the model using equilateral Euclidean triangles that was studied in \cite{GPY}. The goal of this section is to explain the type of metrics we are talking about.

Since we will be gluing triangles, we can define all our metrics on the standard $2$-simplex given by:
\begin{equation*}
\Delta = \verz{t_1e_1+t_2e_2+t_3e_3}{(t_1,t_2,t_3)\in [0,1]^3,\; t_1+t_2+t_3=0}
\end{equation*}
where $\{e_1,e_2,e_3\}$ is the standard basis of $\mathbb{R}^3$. Note that this description of the triangle also gives us a natural midpoint of the triangle and natural midpoints of the sides. So, given a random surface made of these triangles, we get a natural embedding of the corresponding cubic ribbon graph.

We will assume that we have a metric $d:\Delta\times\Delta\rightarrow [0,\infty)$ that comes from a Riemannian metric $g$, given by four smooth functions on $\Delta$: $\rij{g_{ij}}{i,j=1}{2}$.

We will also assume some symmetry. Basically, we want that permuting the corners of the triangle is an isometry of the sides and that all the derivatives of the metric at the boundary of $\Delta$ in directions normal to the boundary vanish. We shall denote the symmetric group of $k!$ elements by $\mathcal{S}_k$. We have the following definition:
\begin{dff}
$(\Delta,g)$ will be called \emph{symmetric in the sides} if:
$$
g_{ij}(t e_i+(1-t)e_j) =g_{ij}(te_{\sigma(i)}+(1-t)e_{\sigma(j)})
$$
for all $\sigma\in \mathcal{S}_3$ (the symmetric group of order $6$), $t\in[0,1]$ and $i,j\in\{1,2\}$ and:
$$
\frac{\partial^k}{\partial n^k}\big|_x g_{ij} = 0
$$
for all $k\geq 1$, $i,j\in\{1,2\}$, $x\in\partial\Delta$ and $n$ normal to $\partial\Delta$ at $x$.
\end{dff}

These two symmetries are necessary to turn the `obvious' gluing into an isometry. It is clear that the sides have to be isometric. The fact that we want the metric to be symmetric under reflection in the midpoint of a side comes from the fact that when we glue two triangles along a side, we might glue them together with an opposite orientation on the two sides. The condition on the derivatives guarantees that the metric is not only continuous but also smooth. Finally, note that these conditions do not imply that $g$ has a central symmetry.

For estimates on the systole later on we want to define a rough minimal and maximal ratio between the length of a curve on the surface and the number of edges of its representive on the graph. To this end we have the following definition:

\begin{dff}
Let $d:\Delta\times\Delta\rightarrow [0,\infty)$ be a metric. We define:
$$
m_1(d) = \min\verz{d(s,s')}{s,s' \text{ opposite sides of a gluing of two copies of }(\Delta,d)\text{ along one side}}
$$
and:
$$
m_2(d) = \max\verz{d\left(\frac{e_i+e_j}{2},\frac{e_k+e_l}{2}\right)}{i,j,k,l\in\{1,2,3\},\; i\neq j,\; k\neq l}
$$
\end{dff}

As mentioned before, a special case of these is given by equilateral Euclidean triangles of side length $1$ as in \cite{GPY}. A simple Euclidean geometric argument shows that in this case we have $m_1(d)=1$ and $m_2(d)=\frac{1}{2}$.

\begin{obs} We also note that in the Riemannian setting the systole of a random surface does not necessarily correspond to a circuit on the graph. We will describe an example of this.  Let us consider the graph (with some loose half-edges attached) in the figure below:
\begin{figure}[H] 
\begin{center} 
\includegraphics[scale=1]{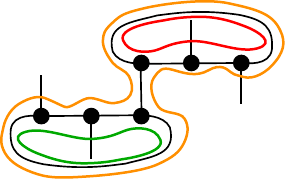} 
\caption{A graph with three cycles on it.}
\label{pic_exriemgraph1}
\end{center}
\end{figure}
The surface (with boundary corresponding to the the loose half-edges) coming from to this graph (taking the orientation from the orientation of the plane) is formed by two cylinders that share half of one of each of their boundary components, as in Figure \ref{pic_exriemcyl} below:
\begin{figure}[H] 
\begin{center} 
\includegraphics[scale=0.75]{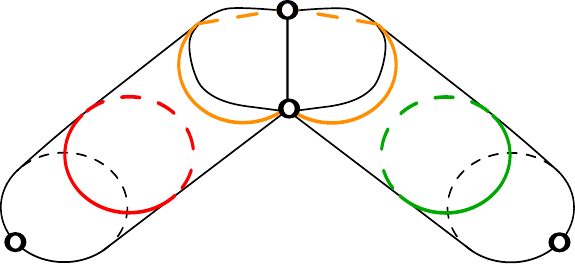} 
\caption{The surface corresponding to the graph from Figure \ref{pic_exriemgraph1}.}
\label{pic_exriemcyl}
\end{center}
\end{figure}

We will triangulate this surface in such a way that the orange cycle is shorter than the two (red and green) circuits. 
\begin{figure}[H] 
\begin{center} 
\includegraphics[scale=0.75]{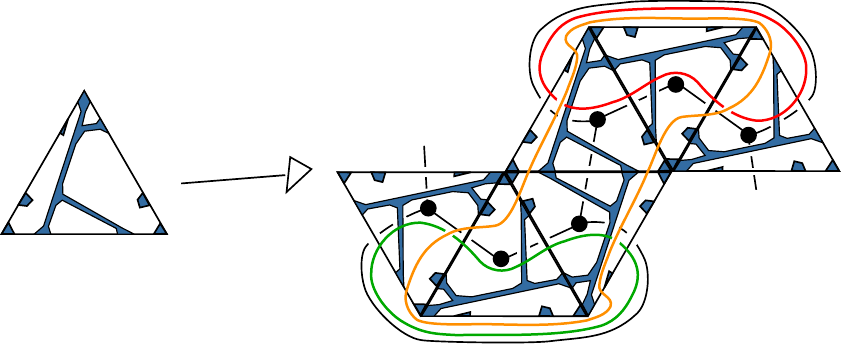} 
\caption{Triangulating the graph from Figure \ref{pic_exriemgraph1}.}
\label{pic_exriemtriang}
\end{center}
\end{figure}

The Riemannian metric we choose on the triangle is such that the white regions on the triangle are `cheap' and the dark blue regions `expensive'. This can be achieved by choosing a Euclidean metric on the whole triangle and multiplying it with a large factor in the blue parts (and then smoothing it). If this factor is large enough then the orange curve, which does not cross any of the blue parts, is shorter than the red curve and the green curve that both do cross the blue parts. Furthermore, the red and green curve cannot be homotoped to curves that do not cross the blue parts. Hence, on this surface with boundary, the shortest closed curve is not homotopic to a circuit. An example of a closed surface that has this property can be obtained by gluing two copies (circled) of the surface above along their boundary as follows:
\begin{figure}[H] 
\begin{center} 
\includegraphics[scale=1]{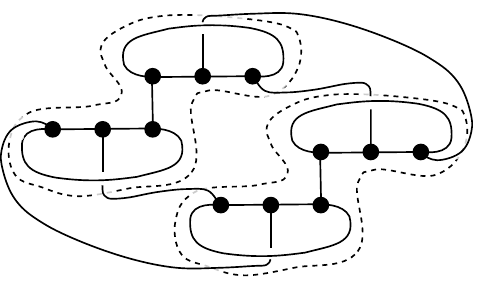} 
\caption{Two copies of the graph from Figure \ref{pic_exriemgraph1} glued together.}
\label{pic_exriemgraph2}
\end{center}
\end{figure}

This graph represents a surface of genus $3$. All the other circuit intersect the blue parts essentially, which means that the orange curve on the subsurface is still shorter than any circuit, even though it is not a circuit itself.

Another feature that this example illustrates is that in the Riemannian setting the central symmetry of the triangle is lost. That means that not only the cyclic orientation at the vertices of the corresponding graph matters, but actually which half-edge gets identified with which half-edge. Luckily, our probability space already encodes this information, so another way of looking at this is that there is less redundancy in the probability space in the Riemannian setting.

\end{obs}

\section{Short separating circuits}\label{sect_sepcurv}

We have already noted that simple closed curves on the surface do not always correspond to circuits on the graph. However, in the Riemannian setting the shortest closed curve homotopic to a circuit will at least be an upper bound for the systole and in the hyperbolic model the the systole will even be homotopic to a circuit. Basically, the end goal is to show that we can ignore separating circuits in the computation of the limit of the expected value in both models. So we need to look at circuits on the graph that are separating on the surface when we see the graph embedded on the surface. Suppose $\gamma$ is such a separating circuit, then there are the following two options:
\begin{itemize}[leftmargin=0.2in]
\item $\gamma$ is homotopic to a corner shared by some number of triangles and corresponds to a left hand turn circuit on the corresponding graph.
\item $\gamma$ is separating on the graph as well.
\end{itemize}
Because in both models the corners of the triangles correspond to singularities on the surface, left hand turn circuits are always considered homotopically trivial. This means that we want to focus on the second case, so we can try to count the number of graphs that have short separating circuits. 

In this section we will forget about surfaces for a moment and consider separating circuits on graphs. Note however that we have not proved that a circuit that is separating on a graph is separating on the corresponding surface as well. In fact, this is not the case and figure \ref{pic29} gives a counter example:

\begin{figure}[H] 
\begin{center} 
\includegraphics[scale=0.7]{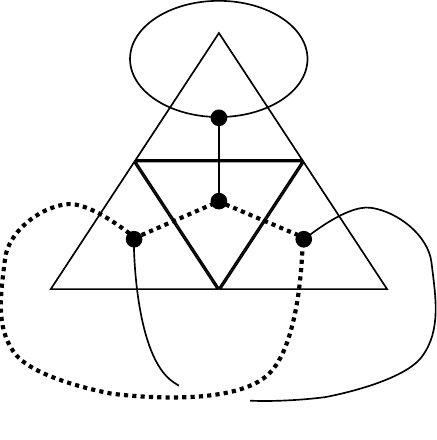} 
\caption{A circuit (dotted) that is separating on the graph but not on the corresponding surface.}
\label{pic29}
\end{center}
\end{figure}

This means that the probability that a random graph has short separating circuits is only an upper bound on the probability that a random graph has short circuits that are separating on the surface.

It also turns out that for the computations in this section it is easier to work with ordered partitions. In particular, this means that as a probability space we will use $\Omega^o_N$.

So in this section we are going to study the following set:
\begin{equation}
G_{N,k}^o=\verz{\omega\in\Omega^o_N}{\Gamma(\omega)\text{ has a separating circuit of length }k}
\end{equation}
where the length of a circuit in this case means the number of edges in that circuit. 

What we want is to count the number of elements in $G_{N,k}^o$ and study the asymptotics of this number. However, we will not count the cardinality of $G_{N,k}^o$ directly, this turns out to be too difficult. Instead we will count the number of graphs we obtain by cutting open the graphs in $G_{N,k}^o$ along a separating circuit of length $k$. This will give us an upper bound for the number of elements in this set.

The plan of the rest of this section is as follows: in section \ref{subs_1} we explain how this cutting along a separating circuit of length $k$ works and how this gives an upper bound for $\aant{G_{N,k}^o}$. After that we count the cardinality of the set of graphs that are cut open along a separating circuit of length $k$ in section \ref{subs_2}. Because the expression we obtain is rather difficult to handle, we turn it into an upper bound in section \ref{subs_3} and in section \ref{subs_4} we use that to prove that the probability of $G_{N,k}^o$ tends to $0$ for $k$ up to $C\log_2(N)$ for any $C\in (0,1)$.

\subsection{Cutting along a separating circuit}\label{subs_1}

If we cut a graph along a separating circuit of length $k$ (as shown in figure \ref{fig_cut} below), we obtain a graph with $k$ degree $1$ vertices and, because the circuit along which we cut is separating, more than one connected component. The degree $1$ vertices will also be spread over different connected components. 

\begin{figure}[H] 
\begin{center} 
\includegraphics[scale=0.5]{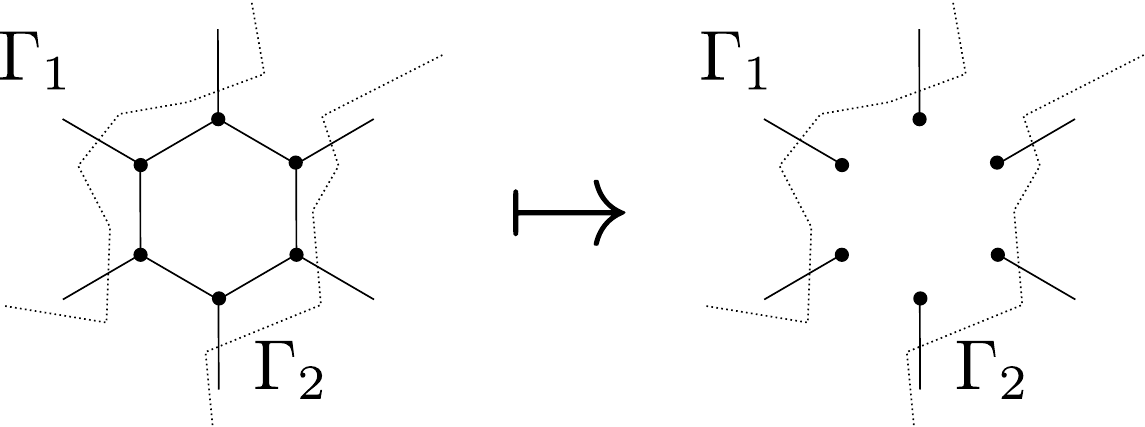} 
\caption{Cutting a graph along a separating circuit such that it splits into two connected components $\Gamma_1$ and $\Gamma_2$ each of which contain degree $1$ vertices.}
\label{fig_cut}
\end{center}
\end{figure}

We need some notation for the set of graphs, or rather to say partitions, we obtain by this procedure. First of all we need to look at more general sets of disjoint pairs out of $\{1,\ldots,6N\}$ than just partitions. We define:
\begin{equation*}
\Theta_N=\verz{\omega\subset\mathcal{P}(\{1,\ldots,6N\})}{\aant{p}=2\;\forall p\in\omega,\; p_1\cap p_2=\emptyset\; \forall p_1\neq p_2\in\omega}
\end{equation*}
where $\mathcal{P}(\{1,\ldots,6N\})$ denotes the power set of $\{1,\ldots,6N\}$. We will denote the corresponding set of ordered sequences of pairs with the same properties by $\Theta_N^o$. Furthermore, recall that for a partition $\omega\in\Omega_N^o$ the corresponding graph is denoted by $\Gamma(\omega)$. An element $\omega\in\Theta_N^o$ also naturally defines a graph which we will denote by $\Gamma(\omega)$ as well. Finally, we will denote the vertex set and edge set of a graph $\Gamma$ by $V(\Gamma)$ and $E(\Gamma)$ repsectively.

First we define the set of partitions that represent graphs with $2N$ vertices out of which $2N-k$ have degree $3$ and $k$ have degree $1$:
\begin{equation}
\Omega_{N,k}^o=\verz{\omega\in\Theta_N^o}{\begin{array}{l}\aant{\omega}=3N-k,\; \aant{V(\Gamma(\omega))}=2N,\\ \aant{\verz{v\in V(\Gamma(\omega))}{\deg(v)=3}}=2N-k,\\ \aant{\verz{v\in V(\Gamma(\omega))}{\deg(v)=1}}=k\end{array}}
\end{equation}

The reason we define this set only for an even number of vertices is that the set of partitions satisfying the properties above but corresponding to an odd number of vertices is empty. This follows from what is sometimes called the handshaking lemma: suppose $\Gamma$ is a graph with vertex set $V(\Gamma)$ consisting vertices of degree $1$ and $3$ only and edge set $E(\Gamma)$. Then the handshaking lemma tells us that:
\begin{align*}
 2\cdot\aant{E(\Gamma)} & = \aant{\verz{v\in V(\Gamma)}{\deg(v)=1}} + 3\cdot\aant{\verz{v\in V(\Gamma)}{\deg(v)=3}} \\
 & = \aant{V(\Gamma)}+2\cdot\aant{\verz{v\in V(\Gamma)}{\deg(v)=3}}
\end{align*}
which means that $\aant{V(\Gamma)}$ must be even.

The graph we obtain by forgetting the pairs that formed the separating circuit is an element the set:
\begin{equation}
D_{N,k}^o=\verz{\omega\in\Omega_{N,k}^o}{\begin{array}{l}\Gamma(\omega)\text{ not connected, }\Gamma(\omega)\text{ has more than one} \\ \text{component with degree }1\text{ vertices}\end{array}}
\end{equation}

Now we want an upper bound for the cardinality of $G_{N,k}^o$ in terms of that of $D_{N,k}^o$. This is given by the following lemma:

\begin{lem}\label{lem_GNkDNk}
$$
\aant{G_{N,k}^o}\leq 2^k (k-1)! \frac{(3N)!}{(3N-k)!} \aant{D_{N,k}^o}
$$
\end{lem}

\begin{proof} To get to an upper bound of this form we need a map $f:G_{N,k}^o\rightarrow D_{N,k}^o$. Because then:
\begin{equation*}
\aant{G_{N,k}^o}\leq \max\verz{\aant{f^{-1}(\omega)}}{\omega\in D_{N,k}^o}\aant{D_{N,k}^o}
\end{equation*}
`Cutting along a separating circuit of length $k$' is a good candidate for such a map. However, this is not a well-defined map: if a graph in $G_{N,k}^o$ has more than one such circuit we have to choose which one to cut. So we fix such a choice and call it $f$.

Now we need to know how many graphs in $G_{N,k}^o$ can land on the same graph in $D_{N,k}^o$ under $f$. 

First of all, given $\omega\in G_{N,k}^o$ note that the degree $1$ vertices of $f(\omega)$ must correspond to the separating circuit that was cut open in $\omega$ by $f$, also if the original has another separating circuit of length $k$. This means that we can give an upper bound for $\max\verz{\aant{f^{-1}(\omega)}}{\omega\in D_{N,k}^o}$ by looking at how many circuits we can construct from the degree $1$ vertices in $D_{N,k}^o$ and then multiplying by a factor that accounts for the order of picking the pairs in the partition.

Given a graph in $D_{N,k}^o$ we can make $(k-1)!$ different circuits out of the $k$ vertices with degree $1$ with $2^k$ different orientations on these vertices. Furthermore, an element in $D_{N,k}^o$ consists of only $3N-k$ pairs of half-edges, while an element of $G_{N,k}^o$ consists of $3N$ pairs of half-edges. This means that, because we are making a distinction between the different orders of picking the pairs of half-edges, we get a factor of $\frac{(3N)!}{(3N-k)!}$. \end{proof}

\subsection{The cardinality of $D_{N,k}^o$}\label{subs_2}
What remains is to compute the cardinality of $D_{N,k}^o$. Let us first compute the number of graphs with $2N$ vertices, where $2N-k$ have degree $3$ and $k$ degree $1$ (so we drop the assumption of the graph having more than one connected component). This means that we want to know the cardinality of $\Omega_{N,k}^o$. We have the following lemma:
\begin{lem}
$$
\aant{\Omega_{N,k}^o}=6^k\binom{2N}{k}\frac{(6N-2k)!}{2^{3N}}
$$
\end{lem}

\begin{proof} Basically, the formula consists of two factors. The first one counts the number of ordered partitions into pairs of a set of $6N-2k$ labeled half-edges. The second factor comes from the fact that we have to choose which vertices will be degree $1$ and which half-edge of these vertices we include, so in principle this factor comes from the labeling of the half-edges.

The first factor is $\frac{(6N-2k)!}{2^{3N-k}}$, there are $k$ vertices with degree $1$, so in a graph there are $3N-k$ pairs of formed out of $6N-2k$ half-edges. There are $\frac{(6N-2k)!}{2^{3N-k}}$ ordered ways to choose these pairs.

Now we have to count the number of labeled sets of half-edges we can choose. First of all, we have to choose $k$ degree $1$ vertices, which gives a factor of $\binom{2N}{k}$. After that, we have to include $1$ out of $3$ half-edges per degree $1$ vertex, this gives a factor of $3^k$. So we get:
\begin{align*}
\aant{\Omega_{N,k}^o} & =3^k\binom{2N}{k}\frac{(6N-2k)!}{2^{3N-k}} \\
                      & =6^k\binom{2N}{k}\frac{(6N-2k)!}{2^{3N}}
\end{align*} 
\end{proof}

With this number, we can compute the cardinality of $D_{N,k}$.

\begin{lem} \label{lem_cardDNk}
$$
\aant{D_{N,k}^o}= \frac{1}{2} \sum\limits_{L=1}^{N-1}\sum\limits_{l=1}^{k-1}\binom{2N}{k}\binom{k}{l}\binom{2N-k}{2L-l}\binom{3N-k}{3L-l}\frac{6^k}{2^{3N}}(6L-2l)!(6(N-L)-2(k-l))!
$$
\end{lem}

\begin{proof} The idea behind the proof is that if $\omega\in D_{N,k}^o$, then we can write:
\begin{equation*}
\Gamma(\omega) = \Gamma(\omega_1)\sqcup\Gamma(\omega_2)
\end{equation*}
with $\omega_1\in\Omega_{L,l}^o$ and $\omega_2\in\Omega_{N-L,k-l}^o$ for some appropriate $L\in\{1,\ldots,N-1\}$ and $l\in\{1,\ldots, k\}$, as depicted in figure \ref{fig_cut}. So we can count the cardinality of $D_{N,k}^o$ by counting all the possible combinations of smaller components.

If the first component has $2L$ vertices out of which $l$ have degree $1$ and the second component has $2(N-L)$ vertices out of which $k-l$ have degree $1$, we have $\aant{\Omega_{L,l}^o}\aant{\Omega_{N-L,k-l}^o}$ possibilities for the graph. There are $\binom{2N}{2L}$ ways of choosing $2L$ out of $2N$ triangles and $\binom{3N-k}{3L-l}$ ways to re-order the picking of the pairs. Also, we are counting everything double: we are making an artificial distinction between the `first' and `second' component. So we get:
\begin{align*}
\aant{D_{N,k}^o} & =\frac{1}{2} \sum\limits_{L=1}^{N-1}\sum\limits_{l=1}^{k-1}\binom{2N}{2L}\binom{3N-k}{3L-l}\aant{\Omega_{L,l}^o}\aant{\Omega_{N-L,k-l}^o} \\
                 & = \frac{6^k}{2^{3N+1}} \sum\limits_{L=1}^{N-1}\sum\limits_{l=1}^{k-1}\binom{2N}{k}\binom{k}{l}\binom{2N-k}{2L-l}\binom{3N-k}{3L-l}(6L-2l)!(6(N-L)-2(k-l))! 
\end{align*}
\end{proof}

Note that there is some redundancy in the expression for $\aant{D_{N,k}^o}$, that is to say, for certain combinations of $L$ and $l$ we have that $\binom{2N-k}{2L-l}=0$. 
There are two cases where this happens. The first one is when $2L-l<0$, so when $L< \ceil{\frac{l}{2}}$. The second one is when $2N-k<2L-l$ so when $L>N-\ceil{\frac{1}{2}(k-l)}$. So we can also write:
\begin{equation*}
\aant{D_{N,k}^o}= \frac{6^k}{2^{3N+1}} \sum\limits_{l=1}^{k-1}\sum\limits_{L=\ceil{\frac{l}{2}}}^{N-\ceil{\frac{1}{2}(k-l)}}\binom{2N}{k}\binom{k}{l}\binom{2N-k}{2L-l}\binom{3N-k}{3L-l}(6L-2l)!(6(N-L)-2(k-l))!
\end{equation*}

Sometimes we will write $\Pro{}{D_{N,k}^o}=\frac{\aant{D_{N,k}^o}}{\aant{\Omega_N^o}}$, which is a slight abuse of notation, because technically it is not defined ($D_{N,k}^o$ is not a subset of $\Omega_N^o$). So, we get:
\begin{equation}\label{eq_PDNk}
\Pro{}{D_{N,k}^o}  = \frac{6^k}{2} \frac{\sum\limits_{l=1}^{k-1}\sum\limits_{L=\ceil{\frac{l}{2}}}^{N-\ceil{\frac{1}{2}(k-l)}}\binom{2N}{k}\binom{k}{l}\binom{2N-k}{2L-l}\binom{3N-k}{3L-l}\binom{6N-2k}{6L-2l}^{-1}}{6N(6N-1)\cdots(6N-2k+1)} 
\end{equation}

Note that $\ceil{\frac{l}{2}} \leq L \leq N-\ceil{\frac{1}{2}(k-l)}$ implies that $\binom{6N-2k}{6L-2l}\neq 0$, so the right hand side of the equation above is well defined.

\subsection{An upper bound for $\Pro{}{D_{N,k}^o}$}\label{subs_3}

Even though we have a closed expression for $\Pro{}{D_{N,k}^o}$, it does not immediately give much insight into how `big' the probability is when we let $N$ grow. The goal of this section is to obtain an upper bound for this probability, using the given expression. 

Because of the binomial coefficients in the expression we will work with Stirling's approximation. In particular we will use a result from \cite{Rob} that also gives us estimates for the error in the approximation. 

\begin{thm}\label{thm_rob} \cite{Rob}
Let $n\in\mathbb{N}$ and $n\neq 0$. Let $\lambda_n\in\mathbb{R}$ such that:
$$
n! = \sqrt{2\pi n}\left(\frac{n}{e}\right)^n e^{\lambda_n}
$$
then:
$$
\frac{1}{12n+1} \leq \lambda_n \leq \frac{1}{12n}
$$
\end{thm}

This means that we can write:
\begin{equation*}
\binom{n}{k} = \sqrt{\frac{n}{2\pi (n-k)k}} \frac{n^n}{(n-k)^{n-k}k^k}e^{\lambda_n-\lambda_{n-k}-\lambda_k}
\end{equation*}
with the bounds on $\lambda_n$, $\lambda_{n-k}$ and $\lambda_k$ from the theorem above for $n,k\in\mathbb{N}$ such that $0<k<n$.

The `difficult part' of the expression for $\Pro{}{D_{N,k}^o}$ is the double sum over the product of binomial coefficients, so we will focus on finding an upper bound for that. For $2L-l>0$ and $2L-l<2N-k$ we will write:
\begin{equation*}
\frac{\binom{2N-k}{2L-l}\binom{3N-k}{3L-l}}{\binom{6N-2k}{6L-2l}} = F_1(N,k,L,l)F_2(N,k,L,l)F_3(N,k,L,l)
\end{equation*}
where:
\begin{equation*}
F_1(N,k,L,l) = \sqrt{\frac{2N-k}{\pi (2(N-L)-(k-l))(2L-l)}}
\end{equation*}

\begin{equation*}
F_2(N,k,L,l)  = \frac{(2N-k)^{2N-k}(6(N-L)-2(k-l))^{3(N-L)-(k-l)} (6L-2l)^{3L-l}}{(2(N-L)-(k-l))^{2(N-L)-(k-l)} (2L-l)^{2L-l}(6N-2k)^{3N-k}}
\end{equation*}
and:
\begin{equation*}
F_3(N,k,L,l) = \frac{\exp(\lambda_{2N-k}+\lambda_{2N-k}+\lambda_{6(N-L)-2(k-l)}+\lambda_{6L-2l})}{\exp(\lambda_{2(N-L)-(k-l)}+\lambda_{2L-l}+\lambda_{3(N-L)-(k-l)}+\lambda_{3L-l}+\lambda_{6N-2k})}
\end{equation*}

\noindent Note that we have to treat the cases where $2L-l=0$ and $2L-l=2N-k$ seperately, because there this way of writing the binomial coefficients does not work.

We want upper bounds for $F_1$, $F_2$ and $F_3$ for $N,k,L,l$ in the appropriate ranges. We will start with $F_1$ and $F_3$ because those are the two easiest expressions.

Because we're assuming that $2L-l>0$ and $2(N-L)-(k-l)>0$ and at least one of these two expressions must be greater than or equal to $\frac{1}{2}(2N-k)$ we get:
\begin{equation*}
F_1(N,k,L,l) \leq \frac{1}{\sqrt{2\pi}}
\end{equation*}

From the fact that $\frac{1}{12n+1} \leq \lambda_n \leq \frac{1}{12n}$ and the assumption that $2L-l>0$ and $2(N-L)-(k-l)>0$ and as a consequence also:
\begin{equation*}
3L-l>0,\;3(N-L)-(k-l)>0,\;6L-2l>0\text{ and }6(N-L)-2(k-l)>0
\end{equation*}
we get that:
\begin{equation*}
F_3(N,k,L,l) \leq e^{-\frac{2}{39}}
\end{equation*}

For both of these bounds we are not really interested in the exact constants, only the fact that such constants exist is important.

For $F_2$ we have the following lemma:
\begin{lem}\label{lem_convexity}
Let $N>0$, $k>0$ and $1\leq l\leq k-1$. Then the function
$$\Phi_{N,k,l}:\left[\ceil{\frac{l}{2}},N-\ceil{\frac{1}{2}(k-l)}\right]\rightarrow\mathbb{R}$$
defined by:
$$\Phi_{N,k,l}(x)=F_2(N,k,x,l)$$
for all $x\in \left[\ceil{\frac{l}{2}},N-\ceil{\frac{1}{2}(k-l)}\right]$ is a convex function.
\end{lem}

\begin{proof} To prove this we will look at the derivative of $F_2(N,k,x,l)$ with respect to $x$ and show that it is monotonely increasing. To shorten the expressions a bit, we will only look at the factor in $F_2(N,k,x,l)$ that actually depends on $x$. So we set:
\begin{equation*}
G_2(N,k,x,l) = \frac{(6(N-x)-2(k-l))^{3(N-x)-(k-l)} (6x-2l)^{3x-l}}{(2(N-x)-(k-l))^{2(N-x)-(k-l)} (2x-l)^{2x-l}}
\end{equation*}
We can write:
\begin{equation*}
G_2(N,k,x,l) = \frac{\exp\left[(3(N-x)-(k-l))\log(6(N-x)-2(k-l))+(3x-l)\log(6x-2l)\right]}{\exp\left[(2(N-x)-(k-l))\log(2(N-x)-(k-l))+(2x-l)\log(2x-l)\right]}
\end{equation*}
So we get:
\begin{align*}
\frac{\partial}{\partial x}G_2(N,k,x,l) & = G_2(N,k,x,l)\left(3\log\left(\frac{6x-2l}{6(N-x)-2(k-l)}\right)+2\log\left(\frac{2(N-x)-(k-l)}{2x-l} \right) \right)
\end{align*}
Because $G_2(N,k,x,l)\geq 0$ we see that the derivative turns from negative (when $x$ is small compared to $N$) to positive. This means that $G_2(N,k,x,l)$ turns from monotonely decreasing to monotonely increasing. The second factor in the derivative is monotonely increasing and the point where it turns from negative to positive is exactly the point where $G_2(N,k,x,l)$ turns from decreasing to increasing. This means that the product of the two (the derivative) is still monotonely increasing, which concludes the proof. \end{proof}

The reason that this is interesting is that this implies that the maxima of $F_2(N,k,\cdot,l)$ on \linebreak $\left[\ceil{\frac{l}{2}},N-\ceil{\frac{1}{2}(k-l)}\right]$ lie on the edges of this interval. The same is true if we look for the maxima of $F_2(N,k,\cdot,l)$ on the subinterval $\left[\ceil{\frac{l}{2}}+1,N-\ceil{\frac{1}{2}(k-l)}-1\right]\subset\left[\ceil{\frac{l}{2}},N-\ceil{\frac{1}{2}(k-l)}\right]$. The lemma allows us to prove the following:

\begin{prp}\label{prp_uboundD}
There exists a $R\in (0,\infty)$ such that for all $N,k\in\mathbb{N}$ with $k\geq 2$ and $N\geq k^2$ we have:
$$
\Pro{}{D_{N,k}^o} \leq \frac{1}{N}\frac{R k^3}{k!} \frac{(3N-k)!}{(3N)!}
$$
\end{prp}

\begin{proof} The proof consists of applying lemma \ref{lem_convexity} to the sum of binomial coefficients that appears in the expression we have for $\Pro{}{D_{N,k}^o}$. We have:
\begin{align}\label{eq_ubound}
\sum\limits_{L=\ceil{\frac{l}{2}}}^{N-\ceil{\frac{1}{2}(k-l)}}\frac{\binom{2N-k}{2L-l}\binom{3N-k}{3L-l}}{\binom{6N-2k}{6L-2l}} & \leq \frac{\binom{2N-k}{2\ceil{\frac{l}{2}}-l}\binom{3N-k}{3\ceil{\frac{l}{2}}-l}}{\binom{6N-2k}{6\ceil{\frac{l}{2}}-2l}}+\frac{\binom{2N-k}{2(N-\ceil{\frac{1}{2}(k-l)})-l}\binom{3N-k}{3(N-\ceil{\frac{1}{2}(k-l)})-l}}{\binom{6N-2k}{6(N-\ceil{\frac{1}{2}(k-l)})-2l}} \notag\\
 & \quad +N\frac{e^{-\frac{2}{39}}}{\sqrt{2\pi}}\max\left\{F_2(N,k,\ceil{\frac{l}{2}}+1,l),F_2(N,k,N-\ceil{\frac{1}{2}(k-l)}-1,l)\right\} 
\end{align}
So we need to study the four terms that appear on the right hand side above. We will only treat the two terms that correspond to $L=\ceil{\frac{l}{2}}$ and $L=\ceil{\frac{l}{2}}+1$. The analysis for the other two terms is analogous, the only difference is that $k-l$ takes over the role of $l$.

If $l$ is even we have:
\begin{align*}
\frac{\binom{2N-k}{2\ceil{\frac{l}{2}}-l}\binom{3N-k}{3\ceil{\frac{l}{2}}-l}}{\binom{6N-2k}{6\ceil{\frac{l}{2}}-2l}} & = \frac{\binom{3N-k}{\frac{l}{2}}}{\binom{6N-2k}{l}} \\
  & = \frac{l!}{\left(\frac{l}{2}\right)!(6N-2k-1)(6N-2k-3)\cdots (6N-2k-l+1)} \\
  & \leq \frac{l(l-1)\cdots (\frac{l}{2}+1)}{(6N-3k)^{\frac{l}{2}}} \\
  & \leq \frac{l(l-1)\cdots (\frac{l}{2}+1)}{(6N-3k) k^{l-1}} \\
  & \leq \frac{k\cdot l!}{(6N-3k) k(k-1)\cdots (k-l+1)} \\
  & = \frac{k}{(6N-3k) \binom{k}{l}}
\end{align*}
where we have used the assumption that $N\geq k^2$ and hence $6N-3k\geq k^2$.

Similarly, for odd $l$ we have:
\begin{equation*}
\frac{\binom{2N-k}{2\ceil{\frac{l}{2}}-l}\binom{3N-k}{3\ceil{\frac{l}{2}}-l}}{\binom{6N-2k}{6\ceil{\frac{l}{2}}-2l}} \leq \frac{(k+3)(k+2)}{(6N-3k)\binom{k}{l}}   
\end{equation*}

For the second term we have:
\begin{align*}
F_2(N,k,\ceil{\frac{l}{2}}+1,l) & = \frac{(2N-k)^{2N-k}}{(2(N-\ceil{\frac{l}{2}})-(k-l)-2)^{2(N-\ceil{\frac{l}{2}})-(k-l)-2}(2\ceil{\frac{l}{2}}-l+2)^{2\ceil{\frac{l}{2}}-l+2}} \\
   & \quad \cdot \frac{(6(N-\ceil{\frac{l}{2}})-2(k-l)-6)^{3(N-\ceil{\frac{l}{2}})-(k-l)-3}(6\ceil{\frac{l}{2}}-2l+6)^{3\ceil{\frac{l}{2}}-l+3}}{(6N-2k)^{3N-k}}
\end{align*}
So if $l$ is even we get:
\begin{equation*}
F_2(N,k,\ceil{\frac{l}{2}}+1,l)  = \frac{(2N-k)^2(2N-k)^{2N-k-2}(6N-2k-l-6)^{3N-k-\frac{1}{2}l-3}(l+6)^{\frac{l}{2}+3}}{(2N-k-2)^{2N-k-2}2^2(6N-2k)^{3N-k}}
\end{equation*}
For all $x\in\mathbb{R}$ we have $\lim\limits_{n\rightarrow\infty}\left(1+\frac{x}{n}\right)^n = e^x$. This means that for all $x\in\mathbb{R}$ there exists a constant $M_x\in(0,\infty)$ such that $\left(1+\frac{x}{n}\right)^n\leq M_x\;e^x$ for all $n\in\mathbb{N}$. So:
\begin{align*}
F_2(N,k,\ceil{\frac{l}{2}}+1,l) & \leq \frac{M_2\;e^2}{4} \frac{(2N-k)^2(6N-2k-l-6)^{3N-k-\frac{1}{2}l-3}(l+6)^{\frac{l}{2}+3}}{(6N-2k)^{3N-k}} \\
  & \leq \frac{M_2\;e^2}{4} \frac{(l+6)^{\frac{l}{2}+3}}{(6N-2k)^2 k^{l-2}} \\
\end{align*}
Using Theorem \ref{thm_rob} we see that there exists a constant $K\in (0,\infty)$ such that $(l+6)^{\frac{l}{2}+3}\leq K\cdot l!$. So there exists a constant $K'\in (0,\infty)$ such that for even $l$:
\begin{equation*}
F_2(N,k,\ceil{\frac{l}{2}}+1,l) \leq \frac{K'}{(6N-2k)^2\binom{k}{l}}
\end{equation*}
The proof for odd $l$ is analogous and as we have already mentioned, so are the corresponding proofs for $L=N-\ceil{\frac{1}{2}(k-l)}$. Hence there exists a constant $K''\in (0,\infty)$ such that:
\begin{equation*}
\max\left\{F_2(N,k,\ceil{\frac{l}{2}}+1,l),F_2(N,k,N-\ceil{\frac{1}{2}(k-l)}-1,l)\right\} \leq \frac{K''}{(6N-2k)^2\binom{k}{l}}
\end{equation*}
So if we fill in equation \ref{eq_ubound} we get:
\begin{equation*}
\sum\limits_{L=\ceil{\frac{l}{2}}}^{N-\ceil{\frac{1}{2}(k-l)}}\frac{\binom{2N-k}{2L-l}\binom{3N-k}{3L-l}}{\binom{6N-2k}{6L-2l}}  \leq \frac{k}{(6N-3k) \binom{k}{l}} + \frac{(k+3)(k+2)}{(6N-3k)\binom{k}{l}} +N\frac{e^{-\frac{2}{39}}}{\sqrt{2\pi}} \frac{K''}{(6N-2k)^2\binom{k}{l}}
\end{equation*}
So there exists a constant $B\in(0,\infty)$ such that:
\begin{equation*}
\sum\limits_{l=1}^{k-1}\binom{k}{l}\sum\limits_{L=\ceil{\frac{l}{2}}}^{N-\ceil{\frac{1}{2}(k-l)}}\frac{\binom{2N-k}{2L-l}\binom{3N-k}{3L-l}}{\binom{6N-2k}{6L-2l}}  \leq \frac{Bk^3}{N}
\end{equation*}
Now we can fill in this bound in the expression we have for $\Pro{}{D_{N,k}^o}$ in equation \ref{eq_PDNk} and we get:
\begin{align*}
\Pro{}{D_{N,k}^o} & \leq \frac{6^k}{2} \frac{\binom{2N}{k} Bk^3}{N\cdot 6N(6N-1)\cdots(6N-2k+1)}  \\ 
   & = \frac{ 3^k}{2\cdot k!} \frac{2N}{6N-2k+1} \left(\prod_{i=1}^k\frac{2N-i}{6N-2i+1}\right) \frac{Bk^3}{N\cdot 3N(3N-1)\cdots (3N-k+1)} \\
\end{align*}
We have $3^{k-1}\prod\limits_{i=1}^k\frac{2N-i}{6N-2i+1}=\prod\limits_{i=1}^k\frac{6N-3i}{6N-2i+1}\leq 1$, because every factor in the product is at most equal to $1$. Thus:
\begin{equation*}
\Pro{}{D_{N,k}^o} \leq \frac{ 3}{2\cdot k!} \frac{2N}{6N-2k+1} \frac{Bk^3}{N\cdot 3N(3N-1)\cdots (3N-k+1)} 
\end{equation*}
So there exists an $R\in (0,\infty)$ such that:
\begin{equation*}
\Pro{}{D_{N,k}^o} \leq \frac{1}{N}\frac{R k^3}{k! 3N(3N-1)\cdots (3N-k+1)}
\end{equation*}
which is what we wanted to prove.\end{proof} 

\subsection{The limit}\label{subs_4} Recall that $G_{N,k}^o\subset\Omega_{N,k}^o$ denotes the set of cubic graphs on $2N$ vertices that contain a separating circuit of $k$ edges. Using lemma \ref{lem_GNkDNk} and proposition \ref{prp_uboundD} we are now able to prove the following:

\begin{thmD}\label{thm_sepcurves}
Let $C\in (0,1)$. We have:
$$
\lim\limits_{N\rightarrow\infty} \Pro{}{\bigcup\limits_{2\leq k\leq C\log_2(N)}G_{N,k}^o} = 0
$$
\end{thmD}

\begin{proof} We have:
\begin{align*}
\Pro{}{\bigcup\limits_{2\leq k\leq C\log_2(N)}G_{N,k}^o} & \leq \sum\limits_{2\leq k\leq C\log_2(N)}\Pro{}{G_{N,k}^o} \\
   & \leq \sum\limits_{2\leq k\leq C\log_2(N)}2^k (k-1)! \frac{(3N)!}{(3N-k)!} \Pro{}{D_{N,k}^o}
\end{align*}
where we have used lemma \ref{lem_GNkDNk} in the last step. Proposition \ref{prp_uboundD} now tells us that there exists an $R\in (0,\infty)$ such that:
\begin{equation*}
\Pro{}{\bigcup\limits_{2\leq k\leq C\log_2(N)}G_{N,k}^o} \leq \sum\limits_{2\leq k\leq C\log_2(N)} 2^k \frac{Rk^2}{N}
\end{equation*}
The last term (corresponding to $k=\floor{C\log_2(N)}$) in the sum above is the biggest term, so we get:
\begin{align*}
\Pro{}{\bigcup\limits_{2\leq k\leq C\log_2(N)}G_{N,k}^o} & \leq \floor{C\log_2(N)}2^{\floor{C\log_2(N)}} \frac{R(\floor{C\log_2(N)})^2}{N} \\
   & \leq \frac{R(C\log_2(N))^3}{N^{1-C}} \\
\end{align*}
Because by assumption $C\in (0,1)$ we get:
\begin{equation*}
\lim\limits_{N\rightarrow\infty}\Pro{}{\bigcup\limits_{2\leq k\leq C\log_2(N)}G_{N,k}^o} = 0
\end{equation*}
which proves the statement. \end{proof}

\section{Circuits on graphs}

The computation of the limit of the expected value of the systole in both models will rely heavily on the distribution of the number of circuits of a fixed length in a random graph. The asymptotic distribution of this number is known and we will recall it in the first part of this section. In the second part of this section we will prove two uniform upper bounds on this distribution. These bounds will be used in dominated convergence arguments later on.

\subsection{Asymptotics}

The number of circuits of any fixed length converges to a Poisson distribution. Before we can state the theorem we need to specify the notion of convergence we are talking about. This is given by the following notion of convergence of $\mathbb{N}$-valued random variables:

\begin{dff}
Let $k\in\mathbb{N}$ and $X_{N,1},\ldots,X_{N,k}:\Omega_i\to \mathbb{N}$ be random variables for all $i\in\mathbb{N}$. We say that $X_{N,1},\ldots,X_{N,k}$ \emph{converge in distribution} to random variables $X_1,\ldots,X_k:\Omega\to\mathbb{N}$ as $N\to\infty$ if for all $n_1,\ldots,n_k \in\mathbb{N}$ we have:
$$
\lim_{N\to\infty}\Pro{}{X_{N,1}=n_1,\ldots, X_{N,k}=n_k} = \Pro{}{X_{1}=n_1,\ldots, X_{k}=n_k}
$$
\end{dff}

Because we will use it a lot, we also recall the definition of a Poisson distributed random variable:
\begin{dff}
Let $(\Omega,\Sigma,\mathbb{P})$ be a probability space. A random variable ${X:\Omega\rightarrow\mathbb{N}}$ is said to be \emph{Poisson distributed with mean $\lambda\in\mathbb{R}$} if:
$$
\Pro{}{X=k}=\frac{\lambda^{k}e^{-\lambda}}{k!}
$$
for all $k\in\mathbb{N}$.
\end{dff}

We also need to define a random variable that counts the number of $k$-circuits:
\begin{dff}
Define $X_{N,k}:\Omega_N\rightarrow\mathbb{N}$ by:
$$
X_{N,k}(\omega) = \aant{\verz{\gamma\text{ circuit on }\Gamma(\omega)}{\aant{\gamma}=k}}
$$
where $\aant{\gamma}$ denotes the number of edges in the circuit $\gamma$.
\end{dff}

We will need the following theorem by Bollob\'as:
\begin{thm}\label{thm_PoissonDist}
\cite{Bol1} As $N\to\infty$, the random variables $X_{N,1},\ldots,X_{N,k}$ converge in distribution to random variables $X_1,\ldots, X_k$, where:
\vspace{-0.2in}
\begin{itemize}[leftmargin=0.2in]
\item $X_i:\mathbb{N}\rightarrow\mathbb{N}$ is a Poisson distributed random variable with mean $\lambda_i=\frac{2^i}{2i}$ for $1\leq i\leq k$
\item The random variables $X_i$ and $X_j$ are independent for $1\leq i\neq j\leq k$
\end{itemize}
\end{thm}

\subsection{Bounds on the probability distribution of $X_{N,k}$}

In this section we will prove two upper bounds on the probability distribution of $X_{N,k}$. The first one will be an upper bound on $\Pro{}{X_{N,k}=0}$ for all $k\in\mathbb{N}$ and uniform in $N$, which will be given by proposition \ref{prp_ubPX1}. The second one will be an upper bound on $\Pro{}{X_{N,k}=i}$ for all $i\in\mathbb{N}$ for fixed $k$ and uniform in $N$ and will be given by lemma \ref{lem_ubPX2}.

The first of these two requires the most effort and will need some preparation. We start with the following lemma:

\begin{lem}\label{lem_kcirc}
Let $e$ be an edge in a cubic graph $\Gamma$. $e$ is part of at most $2^{\floor{\frac{k}{2}}}$ circuits of length $k$.
\end{lem}

\begin{proof} First we assume that $k$ is even. Suppose that $e=\{v_1,v_2\}$. We look at all the vertices at distance $\frac{k}{2}-1$ from $v_1$ and $v_2$ respectively (as in figure \ref{pic9} below).
\begin{figure}[H] 
\begin{center} 
\includegraphics[scale=1]{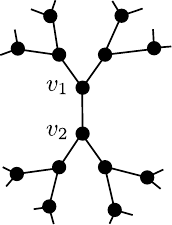} 
\caption{The vertices at distance $\leq 2$ from $v_1$ and $v_2$.}
\label{pic9}
\end{center}
\end{figure}
A circuit of length $k$ containing $e$ corresponds to an edge between a vertex at distance $\frac{k}{2}-1$ from $v_1$ and a vertex at distance $\frac{k}{2}-1$ from $v_2$. There are $2^{\frac{k}{2}-1}$ vertices at distance $\frac{k}{2}-1$ from $v_1$ and $2$ half-edges emanating from each of them. This means that in a given graph there can be at most $2\cdot2^{\frac{k}{2}-1}=2^\frac{k}{2}$ edges between these two sets of vertices and hence at most $2^\frac{k}{2}=2^{\floor{\frac{k}{2}}}$ length $k$ circuits passing through $e$.

If $k$ is odd a length $k$ circuit corresponds to an edge between a vertex at distance $\frac{k}{2}-\frac{1}{2}$ from $v_1$ and a vertex at distance $\frac{k}{2}-\frac{3}{2}$ from $v_2$. There can be at most $2^{\frac{k}{2}-\frac{1}{2}}=2^{\floor{\frac{k}{2}}}$ such edges in a given graph. \end{proof}

We will also need the following definition and theorem from \cite{Wor}:
\begin{dff}
Let $\omega,\omega'\in\Omega_N$. We say $\omega$ and $\omega'$ \emph{differ by a simple switching} when $\omega'$ can be obtained by taking two pairs $\{p_1,p_2\},\{p_3,p_4\}\in\omega$ and replacing them by $\{p_1,p_3\}$ and $\{p_2,p_4\}$.
\end{dff}
So on the level of graphs a simple switching looks like the picture below:
\begin{figure}[H] 
\begin{center} 
\includegraphics[scale=1]{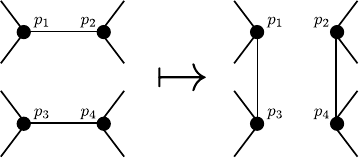} 
\caption{A simple switching.}
\label{pic7}
\end{center}
\end{figure}

The reason we are interested in simple switchings is the following theorem about the behavior of random variables that do not change too much after a simple switching:

\begin{thm}\label{thm_switch} \cite{Wor} Let $c\in (0,\infty)$. If $Z_N:\Omega_N\rightarrow \mathbb{R}$ is a random variable such that \linebreak $\abs{Z_N(\omega)-Z_N(\omega')}<c$ whenever $\omega$ and $\omega'$ differ by a simple switching then:
$$
\Pro{}{\abs{Z_N-\ExV{}{Z_N}}\geq t} \leq 2\exp\left(-\frac{t^2}{6Nc^2}\right)
$$
\end{thm}

In the proof of the upper bound we need to control the number of possible ways two circuits of fixed length can interesect. We will be interested in pairs of $k$-circuits that intersect in $i$ vertices such that the intersection has $j$ connected components. To avoid having to keep repeating this phrase, we have the following definition:
\begin{dff}
An \emph{$(i,j,k)$ double circuit} is a graph consisting of two $k$-circuits that intersect in $i$ vertices such that the intersection has $j$ connected components and such that every vertex of the graph has degree at most $3$.
\end{dff}
Figure \ref{pic30} gives an example:
\begin{figure}[H] 
\begin{center} 
\includegraphics[scale=1]{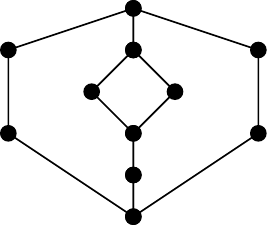} 
\caption{A $(5,2,8)$ double circuit.}
\label{pic30}
\end{center}
\end{figure}

Note that because the degree of the graph is not allowed to exceed $3$, a connected component of the intersection of an $(i,j,k)$ double circuit contains at least $2$ vertices and $1$ edge. This means that we can assume that $j\leq\ceil{\frac{i}{2}}$.

We have the following lemma about this type of graphs:

\begin{lem}\label{lem_doublecirc}
Let $i,j,k\in\mathbb{N}$ such that $1\leq i\leq k$ and $j\leq \floor{\frac{i}{2}}$ then there are at most 
$$2^{j-1}(j-1)!\binom{i-j-1}{j-1}\binom{k-i+j-1}{j-1}^2 $$
isomorphism classes of $(i,j,k)$ double circuits.
\end{lem}

\begin{proof} To prove this we will deconstruct the graphs we are considering into building blocks. We will count how many possible building blocks there are and in how many ways we can put these blocks together.

Every $(i,j,k)$ double circuit can be constructed from the following building blocks:
\begin{itemize}[leftmargin=0.2in]
\item $j$ connected components of the intersection of lengths $l_1,l_2,\ldots,l_j$. Where the $n^{\text{th}}$ connected component consists of a line of $l_n$ vertices with two edges emanating from each end of the line.
\item $j$ segments of the first circuit of lengths $l_{j+1},l_{j+2},\ldots,l_{2j}$. 
\item $j$ segments of the second circuit of lengths $l_{2j+1},l_{2j+2},\ldots,l_{3j}$. 
\end{itemize}
such that:
\begin{equation}\label{eq_cond1}
\sum\limits_{n=1}^j l_n = i
\end{equation}
\begin{equation}\label{eq_cond2}
\sum\limits_{n=k+11}^{2j} l_n = k-i
\end{equation}
\begin{equation}\label{eq_cond3}
\sum\limits_{n=2j+1}^{3j} l_n = k-i
\end{equation}
and:
\begin{equation}\label{eq_cond4}
l_i \geq \left\{
\begin{array}{ll}
2 & \text{if } 1\leq i\leq j \\
0 & \text{otherwise}
\end{array}
\right.
\end{equation}
Figure \ref{pic13} depicts these building blocks:
\begin{figure}[H] 
\begin{center} 
\includegraphics[scale=1]{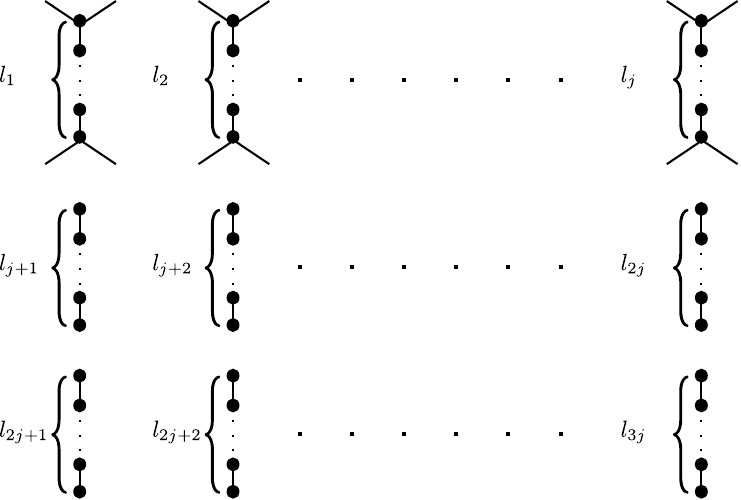} 
\caption{Building blocks for an $(i,j,k)$ double circuit.}
\label{pic13}
\end{center}
\end{figure}
\noindent We can construct an $(i,j,k)$ double circuit out of these blocks in the following way:
\begin{itemize}[leftmargin=0.2in]
\item we start by connecting the $j$ `intersection' blocks by joining them with the first $j$ `line blocks' (we use one loose edge on either side of each intersection block). So, between every pair of intersection blocks there needs to be a line block. From this construction we obtain one circuit with loose edges. The conditions on the lengths $l_i$ above guarantee that we get a $k$-circuit.
\item after this we make the second circuit with the remaining line blocks and `open edges' of the intersection blocks, again such that the blocks alternate.
\end{itemize}
Every $(i,j,k)$ double circuit can be constructed like this. This means that the number of ways this construction can be carried out times the number of sequences $(l_1,l_2,\ldots,l_{3j})$ satisfying conditions (\ref{eq_cond1}), (\ref{eq_cond2}), (\ref{eq_cond3}) and (\ref{eq_cond4}) gives an upper bound for the number of isomorphism classes of $(i,j,k)$ double circuits. 

We start with the factor accounting for the number of sequences $(l_1,l_2,\ldots,l_{3j})$. This factor is:
\begin{equation}\label{eq_binoms}
\binom{i-j-1}{j-1}\binom{k-i+j-1}{j-1}^2
\end{equation}
To prove this, we use the fact that the number of positive integer solutions to the equation:
\begin{equation*}
a_1+a_2+\ldots +a_m=n
\end{equation*}
is equal to:
\begin{equation*}
\binom{n-1}{m-1}
\end{equation*}
(see for instance \cite{Sta}). To get the first binomial coefficient we note that $a_1+a_2+\ldots +a_m=n-m$ if and only if:
\begin{equation*}
(a_1+1)+(a_2+1)+\ldots +(a_m+1)=n
\end{equation*}
So the number of integer solutions to:
\begin{equation*}
b_1+b_2+\ldots +b_m=n
\end{equation*}
with $b_i\geq 2$ for $1\leq i\leq m$ is equal to the number of positive integer solutions to:
\begin{equation*}
a_1+a_2+\ldots +a_m=n-m
\end{equation*}
which is:
\begin{equation*}
\binom{n-m-1}{m-1}.
\end{equation*}
This gives us the first binomial coefficient in (\ref{eq_binoms}). The same trick works for the second binomial coefficient: because:
\begin{equation*}
a_1+a_2+\ldots +a_m=n+m
\end{equation*}
if and only if:
\begin{equation*}
(a_1-1)+(a_2-1)+\ldots +(a_m-1)=n
\end{equation*}
the number of integer solutions to :
\begin{equation*}
c_1+c_2+\ldots +c_m=n
\end{equation*}
with $c_i\geq 0$ for $1\leq i\leq m$ is:
\begin{equation*}
\binom{n+m-1}{m-1},
\end{equation*}
which explains the quadratic part in (\ref{eq_binoms}).

The factor accounting for the number of gluings is equal to:
\begin{equation*}
2^{j-1}(j-1)!
\end{equation*}

We get to this number by using the fact the gluing is determined by two things: we can choose the order of the blocks in each circuit (as long as the alternation of intersections and lines is maintained) and we can choose the orientation of the intersection blocks in each circuit (i.e. to which of the two ends of the block we glue the line). 

We start with first circuit that we glue. When we reorder the blocks in this circuit, the only thing we do is interchanging the lengths $l_1,l_2,\ldots l_{2n}$. This has already been accounted for in the factor for the lengths. Also the choice in orientation of the intersection blocks does not matter for the isomorphism class of the circuit with loose edges we obtain at the end. So for the number of constructions of the first circuit is $1$.

However, for the second circuit the order and orientation of the intersection blocks do matter for the isomorphism class of the $(i,j,k)$ double circuit we obtain (one could say that only the relative order and orientation of the intersections matter). Figure \ref{pic14} below illustrates how the change of orientation can affect the ismorphism class:
\begin{figure}[H] 
\begin{center} 
\includegraphics[scale=1]{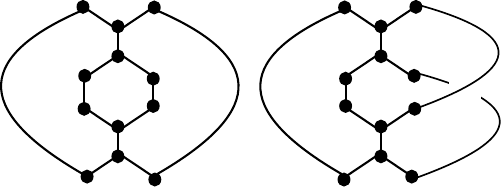} 
\caption{Two non-isomorphic $(4,2,8)$ double circuits.}
\label{pic14}
\end{center}
\end{figure}
The order of the line blocks in the second circuit again corresponds to a changing of lengths that has already been accounted for.

We count the orders as follows: first of all, we note that because the order of the blocks in the first circuit does not matter we can choose it. We choose an order that cyclically corresponds to $l_1,l_{j+1},l_2,l_{j+2},\ldots,l_j,l_{2j}$. To construct the second circuit we start with one of the loose edges of the intersection block corresponding to $l_1$ (it does not matter which edge we choose). We glue the first line block to it and for the other end of the line block we have a choice of $j-1$ intersection blocks and $2$ loose edges per intersection block to glue it to. After we have chosen the intersection block and the edge we glue it to, we glue the second line block to the other loose edge of this intersection block and repeat the process. Like this we pick up a factor of:
\begin{equation*}
2^{j-1}(j-1)!
\end{equation*}

Combining this with the formula for the number of lengths we see that the number of $(i,j,k)$ double circuits is at most:
\begin{equation*}
2^{j-1}(j-1)!\binom{i-j-1}{j-1}\binom{k-i+j-1}{j-1}^2
\end{equation*}
\end{proof}

Now we are ready to prove the following upper bound:

\begin{prp}\label{prp_ubPX1}
There exists a $D\in (0,\infty)$ such that:
$$
\Pro{}{X_{N,k}=0} \leq Dk^8 \left(\frac{3}{8}\right)^k
$$
for all $N\in\mathbb{N}$ with $2N\geq k$.
\end{prp}

\begin{proof} We will consider two cases for the upper bound on $\Pro{}{X_{N,k}=0}$, namely $N\leq \frac{1}{12k^2}\left(\frac{8}{3}\right)^k$ and \linebreak $N>\frac{1}{12k^2}\left(\frac{8}{3}\right)^k$. 

First suppose that $N\leq \frac{1}{12k^2}\left(\frac{8}{3}\right)^k$. In this case we will use Theorem \ref{thm_switch}. We first want to compute the expected value of $X_{N,k}$. To do this, we recapitulate part of the proof of Theorem \ref{thm_PoissonDist}. We have:
\begin{equation*}
\ExV{}{X_{N,k}} = a_{N,k}p_{N,k}
\end{equation*}
where $a_{N,k}$ counts the number of possible distinct labelings a $k$-circuit as a set of $k$ pairs of half edges can have and $p_{N,k}$ is the probability that an element of $\Omega_N$ contains a given set of $k$ pairs of half-edges. $p_{N,k}$ is given by:
\begin{equation*}
p_{N,k}=\frac{1}{(6N-1)(6N-3)\cdots (6N-2k+1)}
\end{equation*}
To count $a_{N,k}$ we reason as follows: we have $6^k(2N-1)(2N-2)\cdots (2N-k+1)$ ways to consistently assign half-edges to a $k$-circuit. However the dihedral group $\mathcal{D}_k$ of order $2k$ acts on the labelings, so we get:
\begin{equation*}
a_{N,k}=\frac{6^k}{2k}(2N-1)(2N-2)\cdots (2N-k+1)
\end{equation*}
So, we obtain:
\begin{equation*}
\ExV{}{X_{N,k}} = \frac{6^k}{2k}\frac{2N(2N-1)\cdots (2N-k+1)}{(6N-1)(6N-3)\cdots (6N-2k+1)}					
\end{equation*}
If $X_{N,k}=0$ then $\abs{X_{N,k}-\ExV{}{X_{N,k}}}=\ExV{}{X_{N,k}}$. Hence:
\begin{equation*}
\Pro{}{X_{N,k}=0}\leq\Pro{}{\abs{X_{N,k}-\ExV{}{X_{N,k}}}\geq \ExV{}{X_{N,k}}}
\end{equation*}
By lemma \ref{lem_kcirc} there are at most $2^{\floor{\frac{k}{2}}}$ $k$-circuits going through an edge. This means that a simple switching can change the number of $k$-circuits by at most $2\cdot 2^{\floor{\frac{k}{2}}}\leq 2^{\frac{k}{2}+1}$. So, using Theorem \ref{thm_switch} we get:
\begin{equation*}
\Pro{}{X_{N,k}=0} \leq 2\exp\left(-\frac{\ExV{}{X_{N,k}}^2}{24N\cdot 2^k}\right) 				
\end{equation*}
Because we are interested in an upper bound for this expression, we need to find a lower bound for $\ExV{}{X_{N,k}}$. We claim that $\ExV{}{X_{N,k}}$ is increasing in $N$, which means that we can get a lower bound by looking at $\ExV{}{X_{\ceil{\frac{k}{2}},k}}$. 

The fact that $\ExV{}{X_{N,k}}$ is increasing in $N$ follows from differentiating it with respect to $N$:
\begin{align*}
\frac{\partial}{\partial N}\left(\ExV{}{X_{N,k}}\right) & =  \ExV{}{X_{N,k}} \left(\sum_{i=0}^{k-1} \frac{2}{2N-i}-\frac{6}{6N-2i-1} \right) \\
	& = \ExV{}{X_{N,k}} \left(\frac{-2}{2N(6N-1)}+\frac{2}{(2N-2)(6N-5)}  + \sum_{i=3}^{k-1} \frac{2i-2}{(2N-i)(6N-2i-1)} \right) 
\end{align*}
We have $2N(6N-1) \geq (2N-2)(6N-5)$, so $\frac{-2}{2N(6N-1)} \geq \frac{-2}{(2N-2)(6N-5)}$. Hence:
\begin{equation*}	
\frac{\partial}{\partial N}\left(\ExV{}{X_{N,k}}\right)	 \geq  \ExV{}{X_{N,k}} \left(\sum_{i=3}^{k-1} \frac{2i-2}{(2N-i)(6N-2i-1)} \right) 
\end{equation*}
We have $\ExV{}{X_{N,k}}\geq 0$ and every term in the sum above is also non negative. So:
\begin{equation*}
\frac{\partial}{\partial N}\left(\ExV{}{X_{N,k}}\right) \geq 0
\end{equation*}
So, indeed:
\begin{align*}
\ExV{}{X_{N,k}} & \geq \ExV{}{X_{\ceil{\frac{k}{2}},k}} \\
		& = \frac{6^k}{2k}\frac{2\ceil{\frac{k}{2}}(2\ceil{\frac{k}{2}}-1)\cdots (2\ceil{\frac{k}{2}}-k+1)}{(6\ceil{\frac{k}{2}}-1)(6\ceil{\frac{k}{2}}-3)\cdots (6\ceil{\frac{k}{2}}-2k+1)}
\end{align*}
We will now look at the asymptotic behavior of this expression for $k\rightarrow \infty$. We get:
\begin{equation*}
\frac{6^k}{2k}\frac{2\ceil{\frac{k}{2}}(2\ceil{\frac{k}{2}}-1)\cdots (2\ceil{\frac{k}{2}}-k+1)}{(6\ceil{\frac{k}{2}}-1)(6\ceil{\frac{k}{2}}-3)\cdots (6\ceil{\frac{k}{2}}-2k+1)}  \sim \sqrt{\frac{2\pi}{k}} \left(\frac{4}{\sqrt{3}} \right)^k
\end{equation*}
where we have used Stirling's approximation. So for $k\rightarrow\infty$ we get:
\begin{align*}
\Pro{}{X_{N,k}=0} & \leq 2\exp\left(-\frac{\ExV{}{X_{N,k}}^2}{24N\cdot 2^k}\right) \\
  & \leq 2\exp\left(-\frac{12k^2\ExV{}{X_{\ceil{\frac{k}{2}},k}}^2}{24\cdot \left(\frac{8}{3}\right)^k 2^k}\right) \\  
  & \sim 2\exp\left(-\pi k\right) 
\end{align*}
So for $N\leq\frac{1}{12k^2}\left(\frac{8}{3}\right)^k$ there exists a $C\in (0,\infty)$ such that:
\begin{equation*}
\Pro{}{X_{N,k}=0} \leq C \exp\left(-\pi k\right) 
\end{equation*}
Because $\exp(-\pi)< \frac{3}{8}$ this implies that there exists a $D\in (0,\infty)$ such that:
\begin{equation}
\Pro{}{X_{N,k}=0} \leq Dk^8 \left(\frac{3}{8}\right)^k \label{eq_rem6}
\end{equation}
for all $\ceil{\frac{k}{2}}\leq N \leq \frac{1}{12k^2}\left(\frac{8}{3}\right)^k$.

Now suppose that $N>\frac{1}{12k^2}\left(\frac{8}{3}\right)^k$. The goal will be to use the following upper bound:
\begin{equation*}
\Pro{}{X_{N,k}=0} \leq 1-\frac{\ExV{}{X_{N,k}}^2}{\ExV{}{X_{N,k}^2}}
\end{equation*}
which can be derived from Cauchy's inequality (see for instance \cite{Bol2}). Also, we will again use some of the ideas from the proof of Theorem \ref{thm_PoissonDist}. We will need an upper bound on $\ExV{}{X_{N,k}}$, and because $N$ is bigger, we can also get a better lower bound on $\ExV{}{X_{N,k}}$. Furthermore we need to get an upper bound on $\ExV{}{X_{N,k}^2}$. 

We start with the bounds on $\ExV{}{X_{N,k}}$. Because $\ExV{}{X_{N,k}}$ has non-negative derivative with respect to $N$, we get:
\begin{equation*}
\ExV{}{X_{N,k}} \leq \frac{2^k}{2k}
\end{equation*}

For the lower bound on $\ExV{}{X_{N,k}}$ we have:
\begin{align*}
\ExV{}{X_{N,k}} & = \frac{2^k}{2k} \prod\limits_{i=0}^{k-1}\left(1-\frac{i-1}{6N-2i-1}\right) \\  
  & \geq \frac{2^k}{2k} \left(1-\frac{k-2}{\max\{\frac{1}{2k^2}\left(\frac{8}{3}\right)^k,3k\}-2k+1}\right)^k
\end{align*}
Because $k\geq 2$ we have $\max\{\frac{1}{2k^2}\left(\frac{8}{3}\right)^k,3k\}-2k+1\geq \frac{1}{4k^2}\left(\frac{8}{3}\right)^k$, so we get:
\begin{equation*}
\ExV{}{X_{N,k}} \geq \frac{2^k}{2k} \left(1-(4k^3-4k^2)\left(\frac{3}{8}\right)^k \right)^k 
\end{equation*}

The last and longest step is the upper bound on $\ExV{}{X_{N,k}^2}$. We will compute $\ExV{}{X_{N,k}^2}$ using the expected value of ${(X_{N,k})_2=X_{N,k}(X_{N,k}-1)}$. Note that $(X_{N,k})_2$ counts the number of ordered pairs of distinct $k$-circuits. We write:
\begin{equation*}
(X_{N,k})_2 = Y'_{N,k}+Y''_{N,k}
\end{equation*}
where $Y'_{N,k}$ counts the number of ordered pairs of vertex disjoint $k$-circuits and $Y''_{N,k}$ counts the number of ordered pairs of vertex non-disjoint $k$-circuits. 

To compute $\ExV{}{Y'_{N,k}}$ we use a similar argument as for $\ExV{}{X_{N,k}}$. Now we have $2k$ vertices and $2k$ pairs of half-edges to label and $\mathcal{D}_k\times\mathcal{D}_k$ acts on the labelings (note that we cannot exchange the two circuits, because we are dealing with ordered pairs of circuits). So we get:
\begin{equation*}
\ExV{}{Y'_{N,k}}=\frac{6^{2k}}{(2k)^2}\frac{2N(2N-1)\cdots (2N-2k+1)}{(6N-1)(6N-3)\cdots (6N-4k+1)}
\end{equation*}
Because the number of factors in the numerator in this expression is equal to the number of factors in its denominator, a similar argument as before shows that this expression again has non negative derivative with respect to $N$. So:
\begin{equation*}
\ExV{}{Y'_{N,k}}\leq \frac{2^{2k}}{(2k)^2} 
\end{equation*}

To compute $\ExV{}{Y''_{N,k}}$ we will use lemma \ref{lem_doublecirc}. So we split up $Y''_{N,k}$ and write:
\begin{equation*}
Y''_{N,k} = \sum_{i=2}^k\sum_{j=1}^{\floor{\frac{i}{2}}} Y''_{N,i,j,k}
\end{equation*}
where $Y''_{N,i,j,k}$ counts the number of ordered $(i,j,k)$ double circuits. We will start with giving an upper bound for $\ExV{}{Y''_{N,i,j,k}}$. Let $\mathcal{I}(i,j,k)$ be the set of isomorphism classes of $(i,j,k)$ double circuits. Furthermore, given $c\in\mathcal{I}(i,j,k)$ let $a_c$ be the number of non-isomorphic labelings of $c$ and let $p_c$ be the probability of finding a fixed labeled $(i,j,k)$ double circuit isomorphic to $c$ in a random cubic graph. Then we have:
\begin{align*}
\ExV{}{Y''_{N,i,j,k}} & = \sum_{c\in \mathcal{I}(i,j,k)} 2a_cp_c \\
   & \leq 2\aant{\mathcal{I}(i,j,k)}\max\verz{a_cp_c}{c\in\mathcal{I}(i,j,k)} \\
   & \leq 2\cdot 2^{j-1}(j-1)!\binom{i-j-1}{j-1}\binom{k-i+j-1}{j-1}^2\max\verz{a_cp_c}{c\in\mathcal{I}(i,j,k)}    
\end{align*}
Note the extra factor $2$ coming from the fact that we are counting oriented double circuits. Also observe that we have used lemma \ref{lem_doublecirc} in the last step. So now we need to find an upper bound on $\max\verz{a_cp_c}{c\in\mathcal{I}(i,j,k)}$. Note that an $(i,j,k)$ double circuit consits of $2k-i$ vertices and $2k-i+j$ edges (because in each connected component of the intersection the number of pairs of vertices that are identified is one more than the number of pairs of edges that are identified). There might also be some symmetry in the double circuit, but because we are looking for an upper bound we disregard this. So if $c\in\mathcal{I}(i,jk)$ we get:
\begin{equation*}
a_c \leq 6^{2k-i}2N(2N-1)\cdots (2N-2k+i+1)
\end{equation*}
and:
\begin{equation*}
p_c=\frac{1}{(6N-1)(6N-3)\cdots (6N-4k+2i-2j-1)}
\end{equation*}
So we get:
\begin{align*}
\ExV{}{Y''_{N,k}} & \leq \sum_{i=2}^k\sum_{j=1}^{\floor{\frac{i}{2}}}\left(2\cdot 2^{j-1}(j-1)!\binom{i-j-1}{j-1}\binom{k-i+j-1}{j-1}^2 \right. \\
   & \quad \left. \cdot \frac{6^{2k-i}2N(2N-1)\cdots (2N-2k+i+1)}{(6N-1)(6N-3)\cdots (6N-4k+2i-2j-1)}\right) \\
   & = \sum_{i=2}^k\sum_{j=1}^{\floor{\frac{i}{2}}} \frac{2^{2k-i+j}(j-1)!\binom{i-j-1}{j-1}\binom{k-i+j-1}{j-1}^2\prod\limits_{m=0}^{2k-i}\frac{6N-3m}{6M-2m-1}}{(6N-4k+2i-1)(6N-4k+2i-3)\cdots (6N-4k+2i-2j-1)}
\end{align*}   
Now we use the fact that:
\begin{equation*}
\prod\limits_{m=0}^{2k-i}\frac{6N-3m}{6M-2m-1}\leq 1
\end{equation*}
which can be proved with the same derivative argument that we used for $\ExV{}{X_{N,k}}$. So we get:
\begin{align*}   
\ExV{}{Y''_{N,k}} & \leq \sum_{i=2}^k\sum_{j=1}^{\floor{\frac{i}{2}}} \frac{2^{2k-i+j}(j-1)!\binom{i-j-1}{j-1}\binom{k-i+j-1}{j-1}^2}{(6N-4k-1)^j} \\
   & = 2^{2k}\sum_{i=2}^k\frac{1}{2^i}\sum_{j=1}^{\floor{\frac{i}{2}}} \frac{2^j (j-1)!(i-j-1)!((k-i+j-1)!)^2}{(j-1)!(i-2j)!((j-1)!)^2((k-i)!)^2(6N-4k-1)^j}       
\end{align*}
We have $\frac{(j-1)!(i-j-1)!}{(j-1)!(i-2j)!}\leq (i-j-1)^{j-1}\leq i^{j-1}$ and similarly $\frac{(k-i-1)!}{(k-i-j)!}\leq k^{j-1}$, so:
\begin{align*}   
\ExV{}{Y''_{N,k}} & \leq 2^{2k}\sum_{i=2}^k\frac{1}{2^i}\sum_{j=1}^{\floor{\frac{i}{2}}} \frac{2^j i^{j-1}(k-i+j-1)^{2j-2}}{((j-1)!)^2(6N-4k-1)^j} \\        
   & \leq 2^{2k}\sum_{i=2}^k\frac{1}{2^i}\sum_{j=1}^{\floor{\frac{i}{2}}} \frac{2^j k^{j}k^{2j}}{((j-1)!)^2(6N-4k-1)^j}        
\end{align*}
Now using the fact that $N\geq \frac{1}{12k^2}\left(\frac{8}{3}\right)^k$ we get:
\begin{align*}   
\ExV{}{Y''_{N,k}} & \leq 2^{2k}\sum_{i=2}^k\frac{1}{2^i}\sum_{j=1}^{\floor{\frac{i}{2}}} \frac{1}{((j-1)!)^2} \left(\frac{2 k^3}{\frac{1}{2k^2}\left(\frac{8}{3}\right)^k-4k-1}        \right)^j
\end{align*}
The $j=1$ term in the sum over $j$ is the largest, thus:
\begin{align*}   
\ExV{}{Y''_{N,k}} & \leq 2^{2k}\sum_{i=2}^k\frac{\floor{\frac{i}{2}}}{2^i} \frac{2 k^3}{\frac{1}{2k^2}\left(\frac{8}{3}\right)^k-4k-1}        
\end{align*}
We apply the same trick again to obtain:
\begin{align*}   
\ExV{}{Y''_{N,k}} & \leq  \frac{k^4 2^{2k}}{\frac{1}{2k^2}\left(\frac{8}{3}\right)^k-4k-1}        
\end{align*}
As such, there must be a $C'\in (0,\infty)$ satisfying:
\begin{align*}   
\ExV{}{Y''_{N,k}} & \leq  C'\frac{k^6 2^{2k}}{\left(\frac{8}{3}\right)^k}         \\
   & = C'k^6\left(\frac{3}{2}\right)^k
\end{align*}

Now we have all the bounds we need and we can put everything together. We have:
\begin{align*} 
\Pro{}{X_{N,k}=0}	& \leq 1-\frac{\ExV{}{X_{N,k}}^2}{\ExV{}{X_{N,k}^2}} \\
			& = 1-\frac{\ExV{}{X_{N,k}}^2}{\ExV{}{(Y'_{N,k})}+\ExV{}{(Y''_{N,k})}+\ExV{}{(X_{N,k})}} \\
			& \leq 1-\frac{\frac{2^{2k}}{(2k)^2} \left(1-\frac{k-1}{6\cdot 2^{k/2}-2k+1}\right)^{2k}}{\frac{2^{2k}}{(2k)^2} + C'k^6\left(\frac{3}{2}\right)^k+\frac{2^k}{2k}} \\
			& \sim 4C'k^8\left(\frac{3}{8}\right)^k
\end{align*}
for $k\rightarrow\infty$. This implies that there exists a $D\in (0,\infty)$ such that:
\begin{equation*}
\Pro{}{X_{N,k}=0} \leq Dk^8\left(\frac{3}{8}\right)^k
\end{equation*}
for all $N >\frac{1}{12k^2}\left(\frac{8}{3}\right)^k$. Putting this together with (\ref{eq_rem6}) proves the proposition. \end{proof}

We also need a bound on the probability that $X_{N,k}$ grows very big for fixed $k$. This is given by the following lemma:
\begin{lem}\label{lem_ubPX2}
Let $k\geq 2$. Then there exists a $C_k\in (0,\infty)$ such that:
$$
\Pro{}{X_{N,k}=i}\leq \frac{C_k}{i^2}
$$
for all $N\in\mathbb{N}$ and all $i\in\mathbb{N}$.
\end{lem}

\begin{proof} Chebyshef's inequality tells us that:
\begin{equation*}
 \Pro{}{\abs{X_{N,k}-\ExV{}{X_{N,k}}}\geq i} \leq \frac{\mathbb{V}\mathrm{ar}\left[X_{N,k} \right]}{i^2}
\end{equation*}
where:
\begin{equation*}
\mathbb{V}\mathrm{ar}\left[X_{N,k} \right] = \ExV{}{X_{N,k}^2}-\ExV{}{X_{N,k}}^2
\end{equation*}
In the proof of proposition \ref{prp_ubPX1} we have seen that $\ExV{}{X_{N,k}}\leq \frac{2^k}{2k}$ and ${\ExV{}{X_{N,k}^2}\leq D2^{2k}}$ for all $N$ and some $D\in (0,\infty)$, which proves the lemma. \end{proof}

\section{The systole in the hyperbolic model}

We have now developped enough graph theoretic tools to study the expected value of the systole. We will start with the hyperbolic model. We are interested in the limit of the expected value for the number of triangles going to infinity (both in the non-compact and compact case). This can be computed by expressing the expected value for a finite number of triangles as a sum and then taking the pointwise limits of the terms in the sum. That means that we have to find an appropriate way of expressing the expected value of the systole for a finite number of triangles, then find the pointwise limits and finally prove that the sum we get is actually equally to the limit. For this last part we will use the dominated convergence theorem.

From hereon we will not make a distinction between the systole and its length.

\subsection{An expression for $\ExV{}{\sys_N}$}\label{sec_exprExV}

We recall from section \ref{sec_GeomHyp} that if a closed curve on a random surface in the hyperbolic model can homotoped to a circuit on the embedded graph carrying the word $w\in\{L,R\}^*$ then the length of this curve is $2\cosh^{-1}(\tr{w}/2)$. Also recall that the systole has to be homotopic to a circuit. This means that if a hyperbolic random surface has systole $2\cosh^{-1}(k/2)$ for some $k\in\mathbb{N}$ then there is some word of trace $k$ that is realized as a non-contractible circuit on the surface and no word of of lower trace that is realized like this.

First of all we note that if a word appears as a non-contractible circuit then so does every cyclic permutation of the word and every cyclic permutation of the word read backwards where we replace every $L$ by an $R$ and vice versa (because reading a word backwards and interchanging the $L$'s and $R$'s corresponds to traveling through the circuit in opposite direction). So it is natural to define the following equivalence: 
\begin{dff}
Two words $w\in\{L,R\}^*$ and $w'\in\{L,R\}^*$ will be called \emph{equivalent} if one of following two conditions holds:
\begin{itemize}[leftmargin=0.2in]
\item $w'$ is a cyclic permutation of $w$
\item $w'$ is a cyclic permutation of $w^*$, where $w^*$ is the word obtained by reading $w$ backwards and replacing every $L$ with an $R$ and vice versa.
\end{itemize}
If $w\in\{L,R\}^*$, we will use $[w]$ to denote the set of words equivalent to $w$.
\end{dff}

Now we need something to measure the number of appearances of an equivalence class of words:
\begin{dff}
Let $N\in\mathbb{N}$ and $w\in\{L,R\}^*$. Define $Z_{N,[w]}:\Omega_N\rightarrow \mathbb{N}$ by:
$$
Z_{N,[w]}(\omega) = \aant{\verz{\gamma}{\gamma\text{ a circuit on }\Gamma(\omega)\text{, }\gamma\text{ carries }w}}
$$
and define $Z_{N,[w]}^{\circ}:\Omega_N\rightarrow \mathbb{N}$ by:
$$
Z_{N,[w]}^{\circ}(\omega) = \aant{\verz{\gamma}{\gamma\text{ a non-contractible circuit on }\Gamma(\omega)\text{, }\gamma\text{ carries }w}}
$$
\end{dff}

We also need to sort the words in $\{L,R\}^*$ by trace. Note that if two words are equivalent then their trace is equal. This means that the following definition makes sense:
\begin{dff}
Let $k\in\mathbb{N}$. Define:
$$
A_k = \verz{[w]}{w\in\{L,R\}^*,\tr{w}=k}
$$
\end{dff}

Now we can write down the expression we want for $\ExV{}{\sys_N}$. We have:
\begin{equation*}
\ExV{}{\sys_N} = \sum\limits_{k=3}^\infty \Pro{}{Z_{N,[w]}^{\circ}=0\;\forall [w]\in\bigcup\limits_{i=3}^{k-1}A_i\text{ and }\exists [w]\in A_k\text{ such that }Z_{N,[w]}^{\circ}>0} 2\cosh^{-1}\left(\frac{k}{2}\right)
\end{equation*}

\subsection{The pointwise limits of $\Pro{}{Z_{N,[w]}^{\circ}=0}$} We will now compute the pointwise limits of the terms in the sum above. To this end we have the following Theorem which is very similar to Theorem \ref{thm_PoissonDist}, both in statement and in proof:
\begin{thmB}
Let $W\subset\{L,R\}^*/\sim$ be a finite set of equivalence classes of words. Then we have: 
$$
Z_{N,[w]}\rightarrow Z_{[w]} \text{ in distribution as }N\rightarrow\infty
$$
for all $[w]\in W$, where:
\begin{itemize}[leftmargin=0.2in]
\item $Z_{[w]}:\mathbb{N}\rightarrow\mathbb{N}$ is a Poisson distributed random variable with mean $\lambda_{[w]}=\frac{\aant{[w]}}{2\abs{w}}$ for all $w\in W$.
\item The random variables $Z_{[w]}$ and $Z_{[w']}$ are independent for all $[w],[w']\in W$ with $[w]\neq [w']$.
\end{itemize}
\end{thmB}

\begin{proof} To prove this we will adapt the proof of Theorem \ref{thm_PoissonDist} as it can be found in \cite{Bol2}\linebreak (Theorem II.4 16). What we will do here is explain the ideas of this proof and how we will change them to obtain the result we need.

The basic tool in the proof of Theorem \ref{thm_PoissonDist} is the method of moments. More specifically: one looks at the limits of all the combined factorial moments. That is, we define:
\begin{equation*}
(X_{N,i})_m=X_{N,i}(X_{N,i}-1)\cdots (X_{N,i}-m+1)
\end{equation*}
for all $i,m,N\in\mathbb{N}$ and we prove that there exists a sequence $(\lambda_1,\lambda_2,\ldots)$ with $\lambda_i\in\mathbb{R}$ for $i=1,2,\ldots$ such that:
\begin{equation*}
\ExV{}{(X_{N,1})_{m_1}(X_{N,2})_{m_2}\cdots (X_{N,k})_{m_k}} \rightarrow \lambda_1^{m_1}\lambda_2^{m_2}\cdots \lambda_k^{m_k}
\end{equation*}
for $N\rightarrow\infty$ for all $(m_1,m_2,\ldots,m_k)\in\mathbb{N}^k$ for all $k\in\mathbb{N}$. This then implies the theorem. The structure of the proof of our proposition is exactly the same, the only difference is that our sequence $\left(\lambda_{[w]}\right)_{[w]\in \{L,R\}^*/\sim}$ is different. We will prove that:
\begin{equation*}
\ExV{}{\prod_{[w]\in W}(Z_{N,[w]})_{m_{[w]}}}\rightarrow \prod_{[w]\in W}\lambda_{[w]}^{m_{[w]}}
\end{equation*}
for all $(m_{[w]})_{[w]\in W}\in \mathbb{N}^{\aant{W}}$, where $\lambda_{[w]}=\frac{\aant{[w]}}{2\abs{w}}$ for all $w\in W$.

First we will look at $\ExV{}{Z_{N,[w]}}$ we will write:
\begin{equation*}
\ExV{}{Z_{N,[w]}} = a_{N,[w]}p_{N,[w]}
\end{equation*}
where $a_{N,[w]}$ counts the number of possible distinct labelings a $[w]$-circuit as a set of $\abs{w}$ pairs of half edges can have and $p_{N,[w]}$ is the probability that an element of $\Omega_N$ contains a given set of $\abs{w}$ pairs of half-edges. 

To count $a_{N,[w]}$ we will count the number of possible distinct labelings of a directed $[w]$-circuit with a start vertex. Because we are fixing a start vertex and direction what we are actually counting is $2\abs{w}a_{N,[w]}$. If we write $w=w_1w_2\cdots w_{\abs{w}}$ where $w_i\in\{L,R\}$ for $i=1,2,\ldots \abs{w}$ then a directed $w$-circuit with a start vertex corresponds to a list $((x_1,w_1x_1),(x_2,w_2x_2),\ldots (x_{\abs{w}},w_{\abs{w}}x_{\abs{w}}))$ where $x_i$ is a half-edge and $w_ix_i$ is the half-edge left from $x_i$ at the same vertex if $w_i=L$ and right from $x_i$ otherwise, for all $1\leq i\leq \abs{w}$. Because $x_1,x_2,\ldots,x_{\abs{w}}$ must all be half-edges from different vertices the number of such lists for the word $w$ is:
\begin{equation*}
3^{\abs{w}}2N(2N-1)\ldots (2N-\abs{w}+1)
\end{equation*}
and because we get these lists for all the representatives of $[w]$ we have:
\begin{equation*}
a_{N,[w]} = \frac{\aant{[w]}}{2\abs{w}}3^{\abs{w}}2N(2N-1)\ldots (2N-\abs{w}+1)
\end{equation*}
Like in the case of the $k$-circuits the probability $p_{N,[w]}$ depends only on the number of pairs of half-edges, so:
\begin{equation*}
p_{N,[w]}=\frac{1}{(6N-1)(6N-3)\cdots (6N-2\abs{w}+1)}
\end{equation*}
This means that:
\begin{equation*}
\ExV{}{Z_{N,[w]}} = \frac{\aant{[w]}}{2\abs{w}}3^{\abs{w}}\frac{2N(2N-1)\ldots (2N-\abs{w}+1)}{(6N-1)(6N-3)\cdots (6N-2\abs{w}+1)}
\end{equation*}
so:
\begin{equation*}
\lim\limits_{N\rightarrow\infty} \ExV{}{Z_{N,[w]}} =  \frac{\aant{[w]}}{2\abs{w}}
\end{equation*}

The next moment to consider is $(Z_{N,[w]})_2$, which counts the number of ordered pairs of $[w]$-circuits. Analogously to the proof of Theorem \ref{thm_PoissonDist} we will write:
\begin{equation*}
(Z_{N,[w]})_2 = Y'_{N,[w]} + Y''_{N,[w]}
\end{equation*}
where $Y'_{N,[w]}$ counts the number of ordered pairs of non-intersecting $[w]$-circuits and $Y''_{N,[w]}$ counts the number of ordered pairs of intersecting $[w]$-circuits. A similar argument as before tells us that:
\begin{equation*}
\lim_{N\rightarrow\infty}\ExV{}{Y'_{N,[w]}} = \left(\frac{\aant{[w]}}{2\abs{w}}\right)^2
\end{equation*}
Furthermore we have $Y''_{N,[w]}\leq Y''_{N,\abs{w}}$, where the latter counts the number of ordered pairs of intersecting $\abs{w}$-circuits. We already know from the proof of Theorem \ref{thm_PoissonDist} that $\ExV{}{Y''_{N,\abs{w}}}=\mathcal{O}(N^{-1})$ for $N\rightarrow\infty$, which implies that the same is true for $\ExV{}{Y''_{N,[w]}}$. So we get that:
\begin{equation*}
\lim_{N\rightarrow\infty}\ExV{}{(Z_{N,[w]})_2} = \left(\frac{\aant{[w]}}{2\abs{w}}\right)^2
\end{equation*}
As in the proof of Theorem \ref{thm_PoissonDist} a similar argument works for the higher and combined moments.\end{proof}

With this proposition we can prove the following:
\begin{prp}\label{prp_ptwiselim} Let $k\in\mathbb{N}$ with $k\geq 3$ then:
$$
\lim_{N\rightarrow\infty}\Pro{}{\substack{Z_{N,[w]}^{\circ}=0\;\forall [w]\in\bigcup\limits_{i=3}^{k-1}A_i\text{ and }\\ \exists [w]\in A_k\text{ such that }Z_{N,[w]}^{\circ}>0}} = \left(\prod\limits_{[w]\in\bigcup\limits_{i=3}^{k-1}A_i} \exp\left(-\frac{\aant{[w]}}{2\abs{w}}\right)\right) \left(1-\prod\limits_{[w]\in A_k}\exp\left(-\frac{\aant{[w]}}{2\abs{w}}\right)\right)
$$
\end{prp}

\begin{proof} First note that the probability we are interested in depends on $Z_{N,[w]}^{\circ}$ and not $Z_{N,[w]}$. So before we can use Theorem B we have to show that the probability that a word is carried by a contractible circuit tends to $0$. This follows from the fact that a circuit corresponding to $w$ (supposing $[w]\neq [L^k]$ for all $k\in\mathbb{N}$) can only be contractible if it is separating. However, a circuit carrying $[w]$ is of a fixed finite length (i.e. $\abs{w}$). By Theorem \ref{thm_sepcurves} the probability that such circuits are separating tends to $0$. So we get:
\begin{align*}
\lim_{N\rightarrow\infty}\Pro{}{\substack{Z_{N,[w]}^{\circ}=0\;\forall [w]\in\bigcup\limits_{i=3}^{k-1}A_i\text{ and }\\ \exists [w]\in A_k\text{ such that }Z_{N,[w]}^{\circ}>0}}
& =
\lim_{N\rightarrow\infty}\Pro{}{\substack{Z_{N,[w]}=0\;\forall [w]\in\bigcup\limits_{i=3}^{k-1}A_i\text{ and }\\ \exists [w]\in A_k\text{ such that }Z_{N,[w]}>0}} \\
& =
\lim_{N\rightarrow\infty}\Pro{}{\substack{Z_{N,[w]}=0\;\forall [w]\in\bigcup\limits_{i=3}^{k-1}A_i\text{ and }\\ \text{not }Z_{N,[w]}=0\; \forall [w]\in A_k}}
\end{align*}
Seeing how the statement on the right hand side is a statement about a finite number of equivalence classes of words, we can apply Theorem B. So, using both the formula for the Poisson distribution and the independence we get:
\begin{align*}
\lim_{N\rightarrow\infty}\Pro{}{\substack{Z_{N,[w]}^{\circ}=0\;\forall [w]\in\bigcup\limits_{i=3}^{k-1}A_i\text{ and }\\ \exists [w]\in A_k\text{ such that }Z_{N,[w]}^{\circ}>0}}
& = \left(\prod\limits_{[w]\in\bigcup\limits_{i=3}^{k-1}A_i} \Pro{}{Z_{[w]}=0}\right) \left(1- \prod\limits_{[w]\in A_k}\Pro{}{Z_{[w]}=0}\right) \\
& = \left(\prod\limits_{[w]\in\bigcup\limits_{i=3}^{k-1}A_i} \exp\left(-\frac{\aant{[w]}}{2\abs{w}}\right)\right) \left(1-\prod\limits_{[w]\in A_k}\exp\left(-\frac{\aant{[w]}}{2\abs{w}}\right)\right)
\end{align*}
which proves the proposition. \end{proof}

\subsection{Convergence} What remains is to prove that we can actually use the limits we computed in the previous section. This will be a three step process. First we look at the non-compact case and we prove that we can (in the appropriate sense) ignore random surfaces with a certain set of properties. After that we use the dominated convergence theorem for the sum that remains in the non-compact case. Finally we prove that the probability that a random surface has small cusps decreases fast enough, which will imply that the expression in the compact case is the same as in the non-compact case.

We start with describing the set of random surfaces that we want to exclude. Let $N\in\mathbb{N}$, recall that the genus of $S_O(\omega)$ for $\omega\in\Omega_N$ is denoted by $g_N(\omega)$. We define the following random variable:
\begin{dff}
Let $N\in\mathbb{N}$. Define $m_{\ell,N}:\Omega_N\rightarrow\mathbb{N}$ by:
$$
m_{\ell,N}(\omega) =
\left\{
\begin{array}{ll}
 \min\verz{\abs{\gamma}}{\gamma\text{ a circuit on } \Gamma(\omega)\text{, non-contractible on }S_C(\omega)} & \text{if }g_N(\omega) > 0 \\
 0 & \text{otherwise}
 \end{array}
\right.
$$
\end{dff}
 
The set of surfaces we want to ignore is the following set:
\begin{equation*}
B_N = \verz{\omega\in\Omega_N}{g_N(\omega) \leq \frac{N}{3}\text{ or }m_{\ell,N}(\omega)>C\log_2(N)\text{ or }\omega\in \bigcup_{2\leq k\leq C\log_2(N)}G_{N,k}}
\end{equation*}
where we have chosen some constant $C\in (0,1)$ that we will keep fixed until the end of this article. 

Before we can prove that we can ignore the surfaces in this set we need to state Gromov's systolic inequality for surfaces:

\begin{thm}\label{thm_gromov}\cite{Gro1} Let $S_g$ denote the homeomorphism class of closed orientable surfaces of genus $g$. Then:
$$
\sup\verz{\frac{\sys(S_g,ds^2)^2}{\mathrm{area}(S_g,ds^2)}}{ds^2\text{ a Riemannian metric on }S_g} \leq \frac{\log(g)}{\pi g}(1+o(1))
$$
when $g\rightarrow \infty$.
\end{thm}

As we said, we want to exclude the surfaces in $B_N$. We want to do this in the following way: we have:
\begin{equation*}
\ExV{}{\sys_N} = \frac{1}{\aant{\Omega_N}}\sum\limits_{\omega\in\Omega_N}\sys_N(\omega)
\end{equation*}
and in the sum above we want to forget about the $\omega\in B_N$. We can prove the following:

\begin{prp}\label{prp_badsurf} In the hyperbolic model we have:
$$
\lim_{N\rightarrow\infty}\frac{1}{\aant{\Omega_N}}\sum\limits_{\omega \in B_N} \sys_N(\omega) = 0
$$
\end{prp}

\begin{proof} Basically $B_N$ consists of $3$ subsets (with some overlap): surfaces of small genus, surfaces with a short separating curve and surfaces with large $m_\ell$. We will prove the seemingly stronger result that the restrictions of the sum to each of these subsets tend to $0$.

We start with surfaces with small genus. For these we will use Theorem \ref{thm_gromov}, Markov's inequality and Theorem \ref{thm_genus}. We have $g_N(\omega)\leq \frac{N+1}{2}$ for all $\omega\in\Omega_N$. This means that $\frac{N+1}{2}-g_N$ is a non-negative random variable and we can apply Markov's inequality to it. We have $g_N(\omega)\leq\frac{N}{3}$ if and only if $\frac{N+1}{2}-g_N(\omega)\geq \frac{N}{6}+1$. So we get:
\begin{align*}
\Pro{}{g_N\leq\frac{N}{3}} & = \Pro{}{\frac{N+1}{2}-g_N \geq \frac{N}{6}+1} \\ 
   & = \frac{\frac{N+1}{2}-\ExV{}{g_N}}{\frac{N}{6}+1}   
\end{align*}
Now we apply Theorem \ref{thm_genus}, which tells us that there exists a constant $C_1\in (0,\infty)$ such that:
\begin{align*}
\Pro{}{g_N\leq\frac{N}{3}} & \leq \frac{\frac{N+1}{2}-1-\frac{N}{2}+\left(C_1+\frac{3}{4}\log(N)\right)}{\frac{N}{6}+1}  \\
  & \leq \frac{K\log(N)}{N}
\end{align*}
for some $K\in (0,\infty)$. We want to apply Gromov's systolic inequality now. The problem is that we need a closed surface to apply this and our surface has cusps. However, we can do the following: at each cusp we cut off a horocycle neighborhood and replace it with a Euclidean hemisphere with an equator of the same length as the horocycle, as in figure \ref{pic15} below.

\begin{figure}[H] 
\begin{center} 
\includegraphics[scale=1]{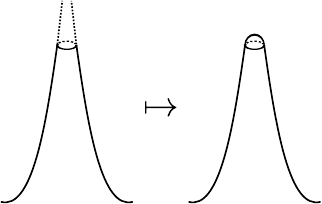} 
\caption{Cutting off the cusps}
\label{pic15}
\end{center}
\end{figure}

\noindent When we do this with all the cusps we get a compact surface with a Riemannian metric on it. If we shorten the length of the horocycle, the area of the neighborhood we cut off gets smaller and so does the area of the hemispheres we glue. Recall that the area of a random surface of $2N$ ideal hyperbolic triangles is $2\pi N$. So, given $\varepsilon > 0$, we can choose the horocycles such that the area of the resulting surface is at most $2\pi N+\varepsilon$. Furthermore, we want that the systole on the resulting surface is at least as long as the systole on the surface with cusps, so we need to be sure that the systole on the resulting surface does not pass through any of the added hemispheres. This again comes down to choosing the horocycles small enough. So if we apply Theorem \ref{thm_gromov} we get:
\begin{align*}
\frac{1}{\aant{\Omega_N}}\sum\limits_{\omega\in\Omega_N,\;g_N(\omega)\leq\frac{N}{3}} \sys_N (\omega) & \leq \frac{A}{\aant{\Omega_N}}\sum\limits_{\omega\in\Omega_N,\;g_N(\omega)\leq\frac{N}{3}}\sqrt{2\pi N+\varepsilon} \\
   & = \Pro{}{g_N\leq\frac{N}{3}}A\sqrt{2\pi N+\varepsilon} \\
   & \leq  AK\frac{\sqrt{2\pi N+\varepsilon}\log(N)}{N}
\end{align*}
for some $A\in (0,\infty)$. Note that we have only used the fact that the ratio $\frac{\sys_N (\omega)^2}{2\pi N+\varepsilon}$ is bounded and not how this bound behaves with respect to the genus of $\omega$. So we get:
\begin{equation}\label{eq_lim1}
\lim\limits_{N\rightarrow\infty}\frac{1}{\aant{\Omega_N}}\sum\limits_{\omega\in\Omega_N,\;g_N(\omega)\leq\frac{N}{3}} \sys_N (\omega) = 0
\end{equation}

The next part we treat is the set of random surfaces with short separating curves. To keep the notation simple we will write: $G_N=\bigcup\limits_{2\leq k\leq C\log_2(N)}G_{N,k}$. So $G_N$ is the set of random surfaces that contain a separating circuit with fewer than $C\log_2(N)$ edges. We have:
\begin{align*}
\frac{1}{\aant{\Omega_N}}\sum\limits_{\omega \in G_N} \sys_N(\omega) & =  \frac{1}{\aant{\Omega_N}}\sum\limits_{\omega \in G_N,\;g_N(\omega)\leq \frac{N}{3}} \sys_N(\omega) + \frac{1}{\aant{\Omega_N}}\sum\limits_{\omega \in G_N,\;g_N(\omega)>\frac{N}{3}} \sys_N(\omega) \\
& \leq \frac{1}{\aant{\Omega_N}}\sum\limits_{\omega \in \Omega_N,\;g_N(\omega)\leq \frac{N}{3}} \sys_N(\omega) + \frac{1}{\aant{\Omega_N}}\sum\limits_{\omega \in G_N,\;g_N(\omega)>\frac{N}{3}} \sys_N(\omega) 
\end{align*}
We already know that the first of these two terms tends to $0$. For the second term we will use Theorem \ref{thm_gromov} again, so we will again apply the trick of replacing the cusps with small hemispheres. This means that there exists a $K'\in(0,\infty)$ such that:
\begin{align*}
\frac{1}{\aant{\Omega_N}}\sum\limits_{\omega \in G_N,\;g_N(\omega)>\frac{N}{3}} \sys_N(\omega)  & \leq \frac{1}{\aant{\Omega_N}}\sum\limits_{\omega \in G_N,\;g_N(\omega)>\frac{N}{3}} \sup\verz{\frac{\sys(S_{g_N(\omega)},ds^2)}{\sqrt{\mathrm{area}(S_{g_N(\omega)},ds^2)}}}{ds^2}\sqrt{2\pi N+\epsilon} \\
   & \leq \frac{1}{\aant{\Omega_N}}\sum\limits_{\omega \in G_N,\;g_N(\omega)>\frac{N}{3}} \frac{K'\log(\frac{N}{3})}{\sqrt{\pi \frac{N}{3}}}\sqrt{2\pi N+\epsilon}  \\
   & = \Pro{}{\omega\in G_N}\frac{K'\log(\frac{N}{3})}{\sqrt{\pi \frac{N}{3}}}\sqrt{2\pi N+\epsilon} 
\end{align*}
From the proof of Theorem \ref{thm_sepcurves} we know that there exists an $R\in (0,\infty)$ such that $\Pro{}{\omega\in G_N}\leq\frac{R(C\log_2(N))^3}{N^{1-C}}$ for all $N\in\mathbb{N}$. So we get:
\begin{equation*}
\frac{1}{\aant{\Omega_N}}\sum\limits_{\omega \in G_N,\;g_N(\omega)>\frac{N}{3}} \sys_N(\omega) \leq  \frac{R(C\log_2(N))^3}{N^{1-C}} \frac{K'\log(\frac{N}{3})}{\sqrt{\pi \frac{N}{3}}}\sqrt{2\pi N+\epsilon}
\end{equation*}
so:
\begin{equation}\label{eq_lim2}
\lim\limits_{N\rightarrow\infty} \frac{1}{\aant{\Omega_N}}\sum\limits_{\omega \in G_N} \sys_N(\omega) = 0
\end{equation}

The last part is surfaces with large $m_\ell$. For these surfaces we can restrict to the surfaces with large genus and without short separating circuits. That is:
\begin{align*}
\lim\limits_{N\rightarrow\infty} \frac{1}{\aant{\Omega_N}}\sum\limits_{\omega \in \Omega_N,\;m_\ell(\omega)>C\log_2(N)} \sys_N(\omega) & = \lim\limits_{N\rightarrow\infty} \frac{1}{\aant{\Omega_N}}\sum\limits_{\substack{\omega \in \Omega_N-G_N,\\ m_\ell(\omega)>C\log_2(N),\;g_N(\omega > \frac{N}{3}}} \sys_N(\omega) \\
& \leq \lim\limits_{N\rightarrow\infty} \Pro{}{\substack{\omega \in \Omega_N-G_N,\\ m_\ell(\omega)>C\log_2(N),\;g_N(\omega) >\frac{N}{3}}} \frac{K'\log(\frac{N}{3})} {\sqrt{\pi \frac{N}{3}}}\sqrt{2\pi N+\epsilon}
\end{align*}
The reason we want to restrict to $\Omega_N-G_N$ is that it makes null homotopic circuits easier: if a circuit is null homotopic it is either separating or it cuts off a cusp, in which case it carries a word of type $L^k$ for some $k$. If we restrict to $\Omega_M-G_N$, the first case does not appear. This means that if the shortest non null homotopic circuit on a random surface has $k$ edges (i.e. $m_\ell=k$) then there are either no $k-1$ circuits or there are $k-1$ circuits that carry a word of the type $L^{k-1}$. So we get:
\begin{align*}
\Pro{}{\substack{\omega \in \Omega_N-G_N,\\ m_\ell(\omega)>C\log_2(N),\;g_N(\omega) >\frac{N}{3}}} & = \Pro{}{\substack{\omega \in \Omega_N-G_N,\;g_N(\omega) >\frac{N}{3}, X_{N,\floor{C\log_2(N)}}=0\\\text{or }X_{N,\floor{C\log_2(N)}}>0 \text{ and all } \\ \floor{C\log_2(N)}\text{-circuits carry words of the} \\\text{type } L^{\floor{C\log_2(N)}}}} \\ 
   & \leq \Pro{}{X_{N,\floor{C\log_2(N)}}=0} + \Pro{}{\substack{X_{N,\floor{C\log_2(N)}}>0 \text{ and all} \floor{C\log_2(N)}\text{-circuits} \\ \text{carry words of the type }  L^{\floor{C\log_2(N)}}}}
\end{align*}
For the first of the two we will use proposition \ref{prp_ubPX1}, which says that there exists a constant $D\in (0,\infty)$ such that $\Pro{}{X_{N,k}=0} \leq Dk^8 \left(\frac{3}{8}\right)^k$. This means that:
\begin{align*}
\Pro{}{X_{N,\floor{C\log_2(N)}}=0} & \leq D(\floor{C\log_2(N)})^8 \left(\frac{3}{8}\right)^{\floor{C\log_2(N)}} \\
   & \leq K''\log(N)^8 N^{C \log_2(\frac{3}{8})}
\end{align*}
Just to give an idea: $\log_2(\frac{3}{8})\approx -1.4$. For the second probability we have:
\begin{align*}
 \Pro{}{\substack{X_{N,\floor{C\log_2(N)}}>0 \text{ and all} \floor{C\log_2(N)}\text{-circuits} \\ \text{carry words of the type }  L^{\floor{C\log_2(N)}}}} & = \sum\limits_{i=1}^\infty  \Pro{}{\substack{X_{N,\floor{C\log_2(N)}}=i \text{ and all} \floor{C\log_2(N)}\text{-circuits} \\ \text{carry words of the type }  L^{\floor{C\log_2(N)}}}}
\end{align*}
The value of $X_{N,\floor{C\log_2(N)}}$ only depends on the graph and not on the orientation on the graph. Because every orientation has equal probability, we can work with the ratio of orientations on a $\floor{C\log_2(N)}$-circuit that correspond to a $L^{\floor{C\log_2(N)}}$ type word. If a graph has one $\floor{C\log_2(N)}$-circuit then this ratio is $\frac{2}{2^{\floor{C\log_2(N)}}}$ (there are $2$ words of type $L^{\floor{C\log_2(N)}}$: the word itself and $R^{\floor{C\log_2(N)}}$ and there are $2^{\floor{C\log_2(N)}}$ possible orientations on a $\floor{C\log_2(N)}$-circuit). If a graph has more than one $\floor{C\log_2(N)}$-circuit, the ratio is at most $\frac{2}{2^{\floor{C\log_2(N)}}}$, which we get from considering the ratio on just one circuit. So we get:
\begin{align*}
 \Pro{}{\substack{X_{N,\floor{C\log_2(N)}}>0 \text{ and all} \floor{C\log_2(N)}\text{-circuits} \\ \text{carry words of the type }  L^{\floor{C\log_2(N)}}}} & \leq \sum\limits_{i=1}^\infty \frac{2}{2^{\floor{C\log_2(N)}}} \Pro{}{X_{N,\floor{C\log_2(N)}}=i} \\
    & \leq \frac{4}{N^C}
\end{align*}
So we obtain:
\begin{equation*}
\frac{1}{\aant{\Omega_N}}\sum\limits_{\omega \in \Omega_N,\;m_\ell(\omega)>C\log_2(N)} \sys_N(\omega) \leq \left( K''\log(N)^8 N^{C \log_2(\frac{3}{8})} + \frac{4}{N^C}\right)\frac{K'\log(\frac{N}{3})} {\sqrt{\pi \frac{N}{3}}}\sqrt{2\pi N+\epsilon}
\end{equation*}
Hence:
\begin{equation}\label{eq_lim3}
\lim\limits_{N\rightarrow\infty} \frac{1}{\aant{\Omega_N}}\sum\limits_{\omega \in \Omega_N,\;m_\ell(\omega)>C\log_2(N)} \sys_N(\omega)  = 0
\end{equation}

When we put (\ref{eq_lim1}), (\ref{eq_lim2}) and (\ref{eq_lim3}) together we get the desired result.\end{proof}

We can now prove Theorem A in the non-compact case:

\begin{thmA1}
In the non-compact hyperbolic setting we have:
$$
\lim_{N\rightarrow\infty}\ExV{}{\sys_N} = \sum\limits_{k=3}^\infty 2\left(\prod\limits_{[w]\in\bigcup\limits_{i=3}^{k-1}A_i} \exp\left(-\frac{\aant{[w]}}{2\abs{w}}\right)\right) \left(1-\prod\limits_{[w]\in A_k}\exp\left(-\frac{\aant{[w]}}{2\abs{w}}\right)\right)\cosh^{-1}\left(\frac{k}{2}\right)
$$
\end{thmA1}

\begin{proof} Of course we will use proposition \ref{prp_badsurf}. This means that we can write:
\begin{equation*}
\lim_{N\rightarrow\infty}\ExV{}{\sys_N} = \lim\limits_{N\rightarrow\infty}  \sum\limits_{k=3}^\infty \Pro{}{\substack{Z_{N,[w]}^{\circ}(\omega)=0\;\forall [w]\in\bigcup\limits_{i=3}^{k-1}A_i\text{ and}\\ \exists [w]\in A_k\text{ such that }Z_{N,[w]}^{\circ}(\omega)>0\\ \omega\in \Omega_N-B_N}} 2\cosh^{-1}\left(\frac{k}{2}\right)
\end{equation*}
supposing the left hand side exists. We know the pointwise limits of the probabilities on the right hand side, these are the same as the ones for $\omega\in\Omega_N$, because the probability that $\omega\in B_N$ tends to $0$. We will apply the dominated convergence theorem to prove the fact that we can use those pointwise limits. This means that we need a uniform upper bound on the probabilities in the sum on the right hand side of the equation.

The point is that given the trace of the word in $L$ and $R$ corresponding to the systole, we get a lower bound on the number of edges on the circuit corresponding to the systole. Because if we want the trace of a word in $L$ and $R$ to increase we need to increase the number of letters in this word. More concretely, the distance between the midpoints on the triangle $T$ is $\log\left(\frac{3+\sqrt{5}}{2}\right)$, this means that $k\log\left(\frac{3+\sqrt{5}}{2}\right)$ is an upper bound for the hyperbolic length of a circuit with $k$ edges. It is noteworthy that the length on the surface of a curve that corresponds to the word $(LR)^k$ is also $\log\left(\frac{3+\sqrt{5}}{2}\right)2k$. This implies that, among all the words of $2k$ letters, $(LR)^k$ is `the word of greatest geodesic length'. 

This means that if the systole is $2\cosh^{-1}\left(\frac{k}{2}\right)$ then $m_\ell\geq \frac{2\cosh^{-1}\left(\frac{k}{2}\right)}{\log\left(\frac{3+\sqrt{5}}{2}\right)}$. So we get:
\begin{equation*}
\Pro{}{\substack{Z_{N,[w]}^{\circ}(\omega)=0\;\forall [w]\in\bigcup\limits_{i=3}^{k-1}A_i\text{ and}\\ \exists [w]\in A_k\text{ such that }Z_{N,[w]}^{\circ}(\omega)>0\\ \omega\in \Omega_N-B_N}}  \leq \Pro{}{\omega\in \Omega_N-B_N,\; m_\ell(\omega) \geq \frac{2\cosh^{-1}\left(\frac{k}{2}\right)}{\log\left(\frac{3+\sqrt{5}}{2}\right)}}
\end{equation*}
Now we use the fact that we can ignore surfaces with short separating curves and big $m_\ell$. That is, if $\omega\notin B_N$ and $m_\ell(\omega) \geq \frac{2\cosh^{-1}\left(\frac{k}{2}\right)}{\log\left(\frac{3+\sqrt{5}}{2}\right)}$ then either $\Gamma(\omega)$ has no circuits of $\floor{\frac{2\cosh^{-1}\left(\frac{k}{2}\right)}{\log\left(\frac{3+\sqrt{5}}{2}\right)}-1}$ edges, or circuits of this length all carry a word consisting of only $L$'s (or only $R$'s, depending on the direction in which we read the word). So we get:
\begin{align*}
\Pro{}{\substack{Z_{N,[w]}^{\circ}(\omega)=0\;\forall [w]\in\bigcup\limits_{i=3}^{k-1}A_i\text{ and}\\ \exists [w]\in A_k\text{ such that }Z_{N,[w]}^{\circ}(\omega)>0\\ \omega\in \Omega_N-B_N}} &  \leq \Pro{}{X_{N,\floor{\frac{2\cosh^{-1}\left(\frac{k}{2}\right)}{\log\left(\frac{3+\sqrt{5}}{2}\right)}-1}}=0} + 2\left(\frac{1}{2}\right)^{\floor{\frac{2\cosh^{-1}\left(\frac{k}{2}\right)}{\log\left(\frac{3+\sqrt{5}}{2}\right)}-1}} \\
   &  \leq D\left(\floor{\frac{2\cosh^{-1}\left(\frac{k}{2}\right)}{\log\left(\frac{3+\sqrt{5}}{2}\right)}-1}\right)^8\left(\frac{3}{8}\right)^{\floor{\frac{2\cosh^{-1}\left(\frac{k}{2}\right)}{\log\left(\frac{3+\sqrt{5}}{2}\right)}-1}} \\
   & \quad + 2\left(\frac{1}{2}\right)^{\floor{\frac{2\cosh^{-1}\left(\frac{k}{2}\right)}{\log\left(\frac{3+\sqrt{5}}{2}\right)}-1}}
\end{align*}
for some $D\in (0,\infty)$ independent of $N$ and $k$. We have:
\begin{equation*}
\cosh^{-1}\left(\frac{k}{2}\right)=\log\left(\frac{k}{2}+\sqrt{\frac{k^2}{4}-1}\right)\geq \log\left(\frac{k}{2}\right)
\end{equation*} 
So: 
\begin{equation*}
\floor{\frac{2\cosh^{-1}\left(\frac{k}{2}\right)}{\log\left(\frac{3+\sqrt{5}}{2}\right)}-1}\geq \frac{2\log\left(\frac{k}{2}\right)}{\log\left(\frac{3+\sqrt{5}}{2}\right)}-2=\frac{2\log\left(\frac{3}{8}\right)}{\log\left(\frac{3+\sqrt{5}}{2}\right)}\log_{\frac{3}{8}}\left(\frac{k}{2}\right)-2
\end{equation*}
and likewise: 
\begin{equation*}
\floor{\frac{2\cosh^{-1}\left(\frac{k}{2}\right)}{\log\left(\frac{3+\sqrt{5}}{2}\right)}-1}\geq \frac{2\log\left(\frac{1}{2}\right)}{\log\left(\frac{3+\sqrt{5}}{2}\right)}\log_{\frac{1}{2}}\left(\frac{k}{2}\right)-2
\end{equation*}
Hence:
\begin{equation*}
\Pro{}{\substack{Z_{N,[w]}^{\circ}(\omega)=0\;\forall [w]\in\bigcup\limits_{i=3}^{k-1}A_i\text{ and}\\ \exists [w]\in A_k\text{ such that }Z_{N,[w]}^{\circ}(\omega)>0\\ \omega\in \Omega_N-B_N}} \leq D\left(\floor{\frac{2\cosh^{-1}\left(\frac{k}{2}\right)}{\log\left(\frac{3+\sqrt{5}}{2}\right)}-1}\right)^8 \left(\frac{8}{3}\right)^2 k^{\frac{2\log\left(\frac{3}{8}\right)}{\log\left(\frac{3+\sqrt{5}}{2}\right)}} + 8 k^{\frac{2\log\left(\frac{1}{2}\right)}{\log\left(\frac{3+\sqrt{5}}{2}\right)}}
\end{equation*}
So there exist constants $D'$ and $D''$ (independent of $N$ and $k$) such that:
\begin{equation*}
\Pro{}{\substack{Z_{N,[w]}^{\circ}(\omega)=0\;\forall [w]\in\bigcup\limits_{i=3}^{k-1}A_i\text{ and}\\ \exists [w]\in A_k\text{ such that }Z_{N,[w]}^{\circ}(\omega)>0\\ \omega\in \Omega_N-B_N}}2\cosh^{-1}\left(\frac{k}{2}\right) \leq D' (\log(k))^9 k^{\frac{2\log\left(\frac{3}{8}\right)}{\log\left(\frac{3+\sqrt{5}}{2}\right)}} + D''\log(k) k^{\frac{2\log\left(\frac{1}{2}\right)}{\log\left(\frac{3+\sqrt{5}}{2}\right)}}
\end{equation*}
We have $\frac{2\log\left(\frac{3}{8}\right)}{\log\left(\frac{3+\sqrt{5}}{2}\right)} \approx -2.0$ and $\frac{2\log\left(\frac{1}{2}\right)}{\log\left(\frac{3+\sqrt{5}}{2}\right)}\approx -1.4$. So the right hand side above is a summable function. This means that we can apply the dominated convergence theorem. In combination with proposition \ref{prp_ptwiselim} this gives the desired result. \end{proof}

It will turn out that the limit of the expected value of the systole in the compact case is the same. This will follow from the fact that asymptotically the non-compact surfaces have large cusps with high probability, which implies that the metrics on the compact and non-compact surfaces are comparable. Given $L\in(0,\infty)$, we write:
\begin{equation*}
E_{L,N} = \verz{\omega\in\Omega_N}{S_O(\omega)\text{ has cusp length }<L}
\end{equation*}

In \cite{BM} (Theorem 2.1), Brooks and Makover proved that random surfaces have large cusps with probability tending to $1$ for $N\rightarrow\infty$. However, we also need to know how fast this probability tends to $1$, so we sharpen their result as follows:

\begin{prp}\label{prp_cusplength} Let $L\in (0,\infty)$. We have:
$$
\Pro{}{E_{L,N}} = \mathcal{O}(N^{-1})
$$
for $N\rightarrow\infty$.
\end{prp}

\begin{proof} The idea of the proof is similar to that in \cite{BM}: if $S_O(\omega)$ has cusp length $< L$, that means that we cannot find non-intersecting horocyles of length $L$ around its cusps. So, there must be two circuits in $\Gamma(\omega)$ that are close. So, we must have subgraph of the form shown in figure \ref{pic32} below: two circuits of lengths $0<\ell_1< L$ and $0<\ell_2< L$ joined by a path of $0\leq d\leq d_{\max}(L)$ edges. Where $d_{\max}(L)$ is determined by the fact that if the distance $d$ between the two circuits becomes too big, it will be possible to choose horocycles of length $L$ around the corresponding cusps. Furthermore, the case $d=0$ will be interpreted as two intersecting circuits.

\begin{figure}[H] 
\begin{center} 
\includegraphics[scale=1.3]{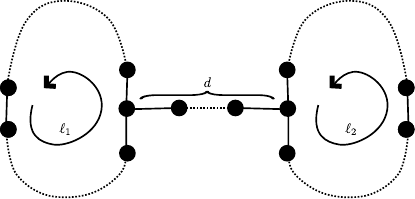} 
\caption{Two circuits joined by a path.}
\label{pic32}
\end{center}
\end{figure}

The set $M_L$ of such graphs is finite. Given $\Gamma\in M_L$, let $X_{\Gamma,N}:\Omega_N\rightarrow\mathbb{N}$ denote the random variable that counts the number of appearances of $\Gamma$ in $\Gamma(\omega)$. By Markov's inequality we have:
\begin{align*}
\Pro{}{X_{\Gamma,N}> 0} & = \Pro{}{X_{\Gamma,N}\geq \frac{1}{2}} \\
   & \leq 2\ExV{}{X_{\Gamma,N}}
\end{align*}
So:
\begin{align*}
\Pro{}{E_{L,N}} & \leq \sum\limits_{\Gamma\in M_L}\Pro{}{X_{\Gamma,N}\geq 0} \\
   & \leq 2\sum\limits_{\Gamma\in M(L)} \ExV{}{X_{\Gamma,N}}
\end{align*}
Every graph $\Gamma\in M_L$ has at least one more edge than it has vertices. This implies that $\ExV{}{X_{\Gamma,N}}=\mathcal{O}(N^{-1})$  (as in the proof of Theorem \ref{thm_PoissonDist}) for all $\Gamma\in M_L$. Because $M_L$ is finite and does not depend on $N$ we get that: 
\begin{equation*}
\Pro{}{E_{L,N}} = \mathcal{O}(N^{-1})
\end{equation*}
\end{proof}

With this and lemma \ref{lem_brooks} we can prove:

\begin{thmA2}
In the compact hyperbolic setting we have:
$$
\lim_{N\rightarrow\infty}\ExV{}{\sys_N} = \sum\limits_{k=3}^\infty 2\left(\prod\limits_{[w]\in\bigcup\limits_{i=3}^{k-1}A_i} \exp\left(-\frac{\aant{[w]}}{2\abs{w}}\right)\right) \left(1-\prod\limits_{[w]\in A_k}\exp\left(-\frac{\aant{[w]}}{2\abs{w}}\right)\right)\cosh^{-1}\left(\frac{k}{2}\right)
$$
\end{thmA2}

\begin{proof} We want to compare the compact with the non-compact setting. To make a distinction between the two, we will write $\sys_{N,C}$ for the systole in the compact setting and $\sys_{N,O}$ for the systole in the non-compact setting. Given $L\in (0,\infty)$, we will split off the set of surfaces in $\Omega_N$ that have cusp length $<L$ in the non-compact setting. 

We have:
\begin{equation*}
\ExV{}{\sys_{N,C}} = \frac{1}{\aant{\Omega_N}}\sum\limits_{\omega \in \Omega_N\backslash E_{L,N}}\sys_{N,C}(\omega)+ \frac{1}{\aant{\Omega_N}}\sum\limits_{\omega \in E_{L,N}}\sys_{N,C}(\omega)
\end{equation*}
Using Theorem \ref{thm_gromov} and proposition \ref{prp_cusplength} we obtain:
\begin{align*}
\frac{1}{\aant{\Omega_N}}\sum\limits_{\omega \in E_{L,N}}\sys_{N,C}(\omega) & \leq \Pro{}{E_{L,N}}\sqrt{2\pi N} \\
   & = \mathcal{O}(N^{-\frac{1}{2}})
\end{align*}
So:
\begin{equation*}
\lim_{N\rightarrow\infty}\ExV{}{\sys_{N,C}} = \lim_{N\rightarrow\infty}\frac{1}{\aant{\Omega_N}}\sum\limits_{\omega \in \Omega_N\backslash E_{L,N}}\sys_{N,C}(\omega)
\end{equation*}
Using lemma \ref{lem_brooks}, we get:
\begin{equation*}
\lim_{N\rightarrow\infty}\ExV{}{\sys_{N,O}}\leq \lim_{N\rightarrow\infty}\ExV{}{\sys_{N,C}} \leq (1+\delta(L))\lim_{N\rightarrow\infty}\ExV{}{\sys_{N,O}}
\end{equation*}
Because $\delta(L)\rightarrow 0$ for $L\rightarrow\infty$ and we can choose $L$ as big as we like, the two limits are actually equal.
\end{proof}

\subsection{A numerical value}\label{sec_numval}

Theorem A of the previous section gives us a formula for the limit of the expected value of the systole in the hyperbolic model. The problem is that the formula is rather abstract and it is not clear how to determine the sets $A_k$ for all $k=3,4,\ldots$. To get to a numerical value for the limit, we can however compute the first couple of terms (because it is not difficult to determine the sets $A_k$ for $k$ up to any finite value) and then give an upper bound for the remainder of the sum. To simplify notation we write:
\begin{equation*}
p_k=\left(\prod\limits_{[w]\in\bigcup\limits_{i=3}^{k-1}A_i} \exp\left(-\frac{\aant{[w]}}{2\abs{w}}\right)\right) \left(1-\prod\limits_{[w]\in A_k}\exp\left(-\frac{\aant{[w]}}{2\abs{w}}\right)\right)
\end{equation*}
for $k=3,4,\ldots$. So:
\begin{equation*}
\lim\limits_{N\rightarrow\infty}\ExV{}{\sys_N} = \sum\limits_{k=3}^\infty p_k 2\cosh^{-1}\left(\frac{k}{2}\right)
\end{equation*}

We have the following lemma:
\begin{lem}\label{lem_ubPsys}
Let $k,n\in \mathbb{N}$ such that $4\leq n\leq k$ then:
$$
p_k \leq \frac{p_n}{e^{k-n}\left(1-\prod\limits_{[w]\in A_n}\exp\left(-\frac{\aant{[w]}}{2\abs{w}}\right)\right)}
$$
\end{lem}
\begin{proof} We have:
\begin{align*}
p_{k} & \leq  \left(\prod\limits_{[w]\in\bigcup\limits_{i=3}^{k-1}A_i} \exp\left(-\frac{\aant{[w]}}{2\abs{w}}\right)\right) \\
   & =  \left(\prod\limits_{[w]\in\bigcup\limits_{i=3}^{n-1}A_i} \exp\left(-\frac{\aant{[w]}}{2\abs{w}}\right)\right)  \left(\prod\limits_{[w]\in\bigcup\limits_{i=n}^{k-1}A_i} \exp\left(-\frac{\aant{[w]}}{2\abs{w}}\right)\right) \\
   & = \frac{p_n}{1-\prod\limits_{[w]\in A_n}\exp\left(-\frac{\aant{[w]}}{2\abs{w}}\right)}\left(\prod\limits_{[w]\in\bigcup\limits_{i=n}^{k-1}A_i} \exp\left(-\frac{\aant{[w]}}{2\abs{w}}\right)\right)
\end{align*}
We know that $[L^{i-2}R]\in A_i$. It is not difficult to see that $\aant{[L^{i-2}R]} = 2(i-1)$ when $i>3$. So:
\begin{align*}
\prod\limits_{[w]\in A_i} \exp\left(-\frac{\aant{[w]}}{2\abs{w}}\right) \leq \exp (-1)
\end{align*}
Because $\prod\limits_{[w]\in\bigcup\limits_{i=n}^{k-1}A_i} \exp\left(-\frac{\aant{[w]}}{2\abs{w}}\right)=\prod\limits_{i=n}^{k-1}\prod\limits_{[w]\in A_i}  \exp\left(-\frac{\aant{[w]}}{2\abs{w}}\right)$ we get:
\begin{equation*}
p_k \leq \frac{p_n}{e^{k-n}\left(1-\prod\limits_{[w]\in A_n}\exp\left(-\frac{\aant{[w]}}{2\abs{w}}\right)\right)}
\end{equation*}
which is what we wanted to prove. \end{proof}

We will write:
\begin{equation*}
S_n = \sum\limits_{k=3}^n p_k 2\cosh^{-1}\left(\frac{k}{2}\right)
\end{equation*}
for $n=3,4,\ldots$. So $S_n$ is the approximation of the limit of the expected value of the systole by the fist $n-2$ terms of the sum.

We have the following proposition:
\begin{prp}
For all $n=3,4,\ldots$ we have:
$$
S_n \leq \lim_{N\rightarrow\infty}\ExV{}{\sys_N} \leq S_n+2\frac{p_n}{1-\prod\limits_{[w]\in A_n}\exp\left(-\frac{\aant{[w]}}{2\abs{w}}\right)}\frac{\left(\frac{\log(n+1)^{\frac{1}{n+1}}}{e}\right)^{n+1}}{1-\frac{\log(n+1)^{\frac{1}{n+1}}}{e}}
$$
\end{prp}

\begin{proof} The inequality on the left hand side is trivial, so we will forcus on the inequality on the right hand side. We have:
\begin{align*}
\lim_{N\rightarrow\infty}\ExV{}{\sys_N} - S_n & = \sum\limits_{k=n+1}^\infty 2p_k\cosh^{-1}\left(\frac{k}{2}\right) \\
  & \leq \sum\limits_{k=n+1}^\infty 2p_k\log(k)
\end{align*}
Now we use lemma \ref{lem_ubPsys} and we get:
\begin{align*}
\lim_{N\rightarrow\infty}\ExV{}{\sys_N} - S_n & \leq  \frac{2e^n p_n}{1-\prod\limits_{[w]\in A_n}\exp\left(-\frac{\aant{[w]}}{2\abs{w}}\right)} \sum\limits_{k=n+1}^\infty \frac{\log(k)}{e^{k}} \\
   & \leq \frac{2e^n p_n}{1-\prod\limits_{[w]\in A_n}\exp\left(-\frac{\aant{[w]}}{2\abs{w}}\right)} \sum\limits_{k=n+1}^\infty \left(\frac{\log(n+1)^{\frac{1}{n+1}}}{e}\right)^k \\
   & =  \frac{2e^n p_n}{1-\prod\limits_{[w]\in A_n}\exp\left(-\frac{\aant{[w]}}{2\abs{w}}\right)} \frac{\left(\frac{\log(n+1)^{\frac{1}{n+1}}}{e}\right)^{n+1}}{1-\frac{\log(n+1)^{\frac{1}{n+1}}}{e}}
\end{align*}
which proves the proposition. \end{proof}

So now approximating $\lim\limits_{N\rightarrow\infty}\ExV{}{\sys_N}$ is just a matter of filling in the proposition above. For instance with $n=7$ we obtain:
\begin{equation}\label{eq_numHyp}
2.48432 \leq \lim_{N\rightarrow\infty}\ExV{}{\sys_N} \leq 2.48434
\end{equation}

\section{The systole in the Riemannian model}

In the Riemannian case we are only able to give an upper bound for the limit of the expected value of the systole. Essentially, the idea is to compute the limit of the expected value of $m_\ell$ which (up to a factor) gives an upper bound for the limit of the expected value of the systole. In section \ref{sec_Pdist} we will compute the limits of the probabilities $\Pro{}{m_{\ell,N}=k}$ for all $k\in\mathbb{N}$ and in section \ref{sec_RiemSys} we will show how we can use these limits to give an upper bound for the limit of the expected value of the systole.

\subsection{The shortest non-trivial curve on the graph}\label{sec_Pdist}

The goal of this section will to compute the following probability:
\begin{equation*}
\Pro{}{m_\ell =k}=\lim\limits_{N\rightarrow\infty}\Pro{}{m_{\ell,N}=k}
\end{equation*}

The idea behind the computation is again to split the probability space $\Omega_N$ up into two subsets. In this case means we will split off the surfaces with short non-trivial curves that are separating and surfaces with pairs of intersecting short non-trivial curves. So we define the following set:
\begin{equation*}
H_{N,k}=\verz{\omega\in\Omega_N}{\omega\text{ contains two intersecting circuits both with }\leq k\text{ edges}}
\end{equation*}
We have the following lemma about this set:
\begin{lem}\label{lem_Intersect}
$$
\lim\limits_{N\rightarrow\infty}\Pro{}{H_{N,k}} = 0
$$
\end{lem}

\begin{proof} This proof is basically a part of the proof in \cite{Bol1} of Theorem \ref{thm_PoissonDist}.

Let $Y_{N,k}:\Omega_N\rightarrow\mathbb{N}$ be the random variable that counts the number of distinct pairs of intersecting circuits of length at most $k$. So:
\begin{equation*}
H_{N,k} = \verz{\omega\in\Omega_N}{Y_{N,k}(\omega)\geq 1}
\end{equation*}
So Markov's inequality implies that:
\begin{equation*}
\Pro{}{H_{N,k}} \leq \ExV{}{Y_{N,k}}
\end{equation*}
In the proof of Theorem \ref{thm_PoissonDist} in \cite{Bol1} (and in the special case of two circuits of the same length in the proof of proposition \ref{prp_ubPX1}) it is proved that $\ExV{}{Y_{N,k}}$ is $\mathcal{O}(N^{-1})$ for $N\rightarrow\infty$. \end{proof}

\begin{prp}\label{prp_mell} For all $k\in\mathbb{N}$ we have:
$$
\Pro{}{m_\ell=k} = e^{-\sum\limits_{j=1}^{k-1}\frac{2^{j-1}-1}{j}} -e^{-\sum\limits_{j=1}^k\frac{2^{j-1}-1}{j}}
$$
\end{prp}

\begin{proof} First of all we have:
\begin{align*}
\Pro{}{m_\ell=k} =  \lim\limits_{N\rightarrow\infty} \Pro{}{m_{\ell,N}=k,\; \omega\in \Omega_N-H_{N,k}} +\Pro{}{m_{\ell,N}=k,\; \omega\in H_{N,k}}
\end{align*}
Because $\Pro{}{m_{\ell,N}=k,\; \omega\in H_{N,k}} \leq \Pro{}{H_{N,k}}$, lemma \ref{lem_Intersect} tells us that:
\begin{equation*}
\Pro{}{m_\ell=k} = \lim\limits_{N\rightarrow\infty}\Pro{}{m_{\ell,N}=k,\; \omega\in \Omega_N-H_{N,k}}
\end{equation*}
With a similar argument and Theorem D from section \ref{sect_sepcurv} we can also exclude separating circuits. Recall that $G_{N,i}$ denotes the set of partitions $\omega\in\Omega_N$ such that $\Gamma(\omega)$ has a separating circuit of $i$ edges. We have:
\begin{align*}
\Pro{}{m_\ell=k} & =  \lim\limits_{N\rightarrow\infty} \left(\Pro{}{m_{\ell,N}=k,\; \omega\in \Omega_N-H_{N,k}-\bigcup\limits_{i=2}^kG_{N,i}}\right. \notag\\ 
                 & \quad   + \left. \Pro{}{m_{\ell,N}=k,\; \omega\in \bigcup\limits_{i=2}^kG_{N,i}-H_{N,k}}\right) \\
                 & =  \lim\limits_{N\rightarrow\infty} \Pro{}{m_{\ell,N}=k,\; \omega\in \Omega_N-H_{N,k}-\bigcup\limits_{i=2}^kG_{N,i}} 
\end{align*}
So, in our computation we only need to consider non-separating curves that do not intersect each other. Recall that the random variable $X_{N,j}$ counts the number of circuits of length $j$ on elements of $\Omega_N$. We split the probability above up into a sum over the possible values of $X_{N,j}$:
\begin{equation*}
\Pro{}{m_\ell=k} = \lim\limits_{N\rightarrow\infty} \sum\limits_{i_1,\ldots,i_{k-1}=0}^\infty\sum\limits_{i_k=1}^\infty \Pro{}{\substack{m_{\ell,N}=k,\;\omega\in \Omega_N-H_{N,k}-\bigcup\limits_{i=2}^kG_{N,i}\\
X_{N,j}(\omega)=i_j \text{ for } 1\leq j\leq k.
}}
\end{equation*}

Because a surface in the Riemannian setting still induces an orientation on the corresponding graph, we can still assign words in $L$ and $R$ to circuits on the graph. These words no longer have a geometric meaning, but they still tell us whether or not a circuit turns around a corner on the surface. In fact, a non-separating curve is (non-)trivial if and only if the word on the corresponding curve on the graph is (un)equal to $L^j$ or $R^j$, where $j$ is the length of this curve. So if $\omega\in \Omega_N-H_{N,k}-\bigcup\limits_{i=2}^kG_{N,i}$ then the condition that $m_{\ell,N}(\omega)=k$ is equivalent to: all circuits $\gamma$ on $\Gamma(\omega)$ of less than $k$ edges carry a word equivalent to $L^j$ where $j$ is the number of edges of $\gamma$ and there is at least one circuit of $k$ edges on $\Gamma(\omega)$ that carries a word that is unequivalent to $L^k$.
 
Furthermore, we observe that if a graph has no intersecting curves of length less than $k$, the words of these curves are independent: any combination of words on the curves is possible and equally probable. This means that we can just count the fraction of surfaces with the `right words' on short curves. If $X_{N,j}=i_j$ for $j=1,\ldots k$, this fraction is: $\left(1-\frac{2^{i_k}}{2^{i_kk}}\right)\left(\prod\limits_{j=1}^{k-1}\frac{2^{i_j}}{2^{i_jj}}\right)$. So:
\begin{align*}
\Pro{}{m_\ell=k} & = \lim\limits_{N\rightarrow\infty} \sum\limits_{\substack{i_1,\ldots,\\ i_{k-1}=0,\\ i_k=1}}^\infty \Pro{}{\substack{\omega\in \Omega_N-H_N^k-\bigcup\limits_{i=2}^kG_N^i\\ X_N^j(\omega)=i_j}}\left(1-\frac{2^{i_k}}{2^{i_kk}}\right)\left(\prod\limits_{j=1}^{k-1}\frac{2^{i_j}}{2^{i_jj}}\right) 
\end{align*}
We will now use the fact (i.e. Theorem \ref{thm_PoissonDist}) that the random variables $X_{N,j}$ converge to Poisson distributions in the $\sup$-norm on $\mathbb{N}$. Even though this theorem is about the convergence on the entire probability space $\Omega_N$, the fact that $\Pro{}{H_N^k\cup\bigcup\limits_{i=2}^kG_N^i}\rightarrow 0$ for $N\rightarrow\infty$ tells us that the limit of $ \Pro{}{\substack{\omega\in \Omega_N-H_N^k-\bigcup\limits_{i=2}^kG_N^i\\ X_N^j(\omega)=i_j}}$ is the same as that of $\Pro{}{\substack{\omega\in \Omega_N, X_N^j(\omega)=i_j}}$ for $N\rightarrow\infty$.

We also need to prove that we can actually use these limits. For this we have lemma \ref{lem_ubPX2} in combination with the dominated convergence theorem. Given $i_1\in\mathbb{N}$ we have:
\begin{align*}
\sum\limits_{\substack{i_2,\ldots,\\ i_{k-1}=0,\\ i_k=1}}^\infty \Pro{}{\substack{\omega\in \Omega_N-H_N^k-\bigcup\limits_{i=2}^kG_N^i\\ X_N^j(\omega)=i_j}}\left(1-\frac{2^{i_k}}{2^{i_kk}}\right)\left(\prod\limits_{j=1}^{k-1}\frac{2^{i_j}}{2^{i_jj}}\right) & \leq \Pro{}{X_N,1=i_1} \\
   & \leq \frac{C_1}{i_1^2}
\end{align*}
for some $C_1\in (0,\infty)$ independent of $i_1$. This is a summable function, so we have:
\begin{align*}
\Pro{}{m_\ell=k} & = \sum\limits_{i_1=0}^\infty\lim\limits_{N\rightarrow\infty} \sum\limits_{\substack{i_2,\ldots,\\ i_{k-1}=0,\\ i_k=1}}^\infty \Pro{}{\substack{\omega\in \Omega_N-H_N^k-\bigcup\limits_{i=2}^kG_N^i\\ X_N^j(\omega)=i_j}}\left(1-\frac{2^{i_k}}{2^{i_kk}}\right)\left(\prod\limits_{j=1}^{k-1}\frac{2^{i_j}}{2^{i_jj}}\right) 
\end{align*}
We can apply this trick $k$ times and we get:
\begin{align*}
\Pro{}{m_\ell=k} & = \sum\limits_{\substack{i_1,\ldots,\\ i_{k-1}=0,\\ i_k=1}}^\infty \lim\limits_{N\rightarrow\infty} \Pro{}{\substack{\omega\in \Omega_N-H_N^k-\bigcup\limits_{i=2}^kG_N^i\\ X_N^j(\omega)=i_j}}\left(1-\frac{2^{i_k}}{2^{i_kk}}\right)\left(\prod\limits_{j=1}^{k-1}\frac{2^{i_j}}{2^{i_jj}}\right) 
\end{align*}
This means that:
\begin{align*}
\Pro{}{m_\ell=k} = \sum\limits_{i_1,\ldots,i_{k-1}=0}^\infty\sum\limits_{i_k=1}^\infty \frac{\left(\frac{2^j}{2j}\right)^{i_j}e^{-\frac{2^j}{2j}}}{i_j!}\left(1-\frac{2^{i_k}}{2^{i_kk}}\right)\left(\prod\limits_{j=1}^{k-1}\frac{2^{i_j}}{2^{i_jj}}\right)
\end{align*}
For $1\leq j<k$ we compute:
\begin{align*}
\sum\limits_{i_j=0}^\infty \frac{\left(\frac{2^j}{2j}\right)^{i_j}e^{-\frac{2^j}{2j}}}{i_j!} \frac{2^{i_j}}{2^{i_jj}} & = e^{-\frac{2^j}{2j}}\sum\limits_{i_j=0}^\infty \frac{1}{i_j!} \frac{1}{j^{i_j}} \\
	& = e^{-\frac{2^{j-1}-1}{j}}
\end{align*}
and for $j=k$ we have:
\begin{align*}
\sum\limits_{i_k=1}^\infty \frac{\left(\frac{2^k}{2k}\right)^{i_k}e^{-\frac{2^k}{2k}}}{i_k!} \left(1-\frac{2^{i_k}}{2^{i_kk}}\right) & = e^{-\frac{2^k}{2k}}\sum\limits_{i_k=1}^\infty \frac{1}{i_k!} \left(\left(\frac{2^k}{2k}\right)^{i_k}-\frac{1}{k^{i_k}}\right) \\
	& = e^{-\frac{2^k}{2k}}\left(e^{\frac{2^k}{2k}}-1-e^{\frac{1}{k}}+1\right) \\
	& = 1-e^{-\frac{2^{k-1}-1}{k}}	
\end{align*}
So we get:
\begin{align*}
\Pro{}{m_\ell=k} & = \left(1-e^{-\frac{2^{k-1}-1}{k}}\right)\prod\limits_{j=1}^{k-1}e^{-\frac{2^{j-1}-1}{j}} \\
                              & =  e^{-\sum\limits_{j=1}^{k-1}\frac{2^{j-1}-1}{j}} -e^{-\sum\limits_{j=1}^k\frac{2^{j-1}-1}{j}}
\end{align*}
\end{proof}

\subsection{The systole}\label{sec_RiemSys}

To prove an upper bound on the limit of the expected value of the systole in the Riemannian model we proceed in the same way as in the hyperbolic model: we start by showing that we can ignore a certain set of surfaces in our computation and after that we will use dominated convergence to prove a formula for what remains. 
\begin{prp}\label{prp_badsurfEucl} In the Riemannian model we have:
$$
\lim_{N\rightarrow\infty}\frac{1}{\aant{\Omega_N}}\sum\limits_{\omega \in B_N} \sys_N(\omega) = 0
$$
\end{prp}
\begin{proof} The proof of this proposition is identical to that of proposition \ref{prp_badsurf}, except that we use different constants here: the area of a random surface with $2N$ triangles is $2N\cdot\mathrm{area}(\Delta,d)$ instead of $2\pi N$.
\end{proof}

Using this proposition we can prove the following theorem:

\begin{thmC}
In the Riemannian we have:
$$
m_1(d)\leq \liminf_{N\rightarrow\infty}\ExV{}{\sys_N}
$$
and
$$
\limsup_{N\rightarrow\infty}\ExV{}{\sys_N} \leq m_2(d)\sum_{k=2}^\infty k\left(e^{-\sum\limits_{j=1}^{k-1}\frac{2^{j-1}-1}{j}} -e^{-\sum\limits_{j=1}^k\frac{2^{j-1}-1}{j}}\right)
$$
\end{thmC}

\begin{proof} For the lower bound we reason as follows: if the random surface has genus greater than $0$ then it has a homotopically non-trivial closed curve and hence a non zero systole. The systole cannot be contained in a triangle and neither can it turn around a shared corner of some number of triangles. This means that it has to contain a segment that crosses two triangles through the opposite sides as in the picture below:
\begin{figure}[H] 
\begin{center} 
\includegraphics[scale=0.5]{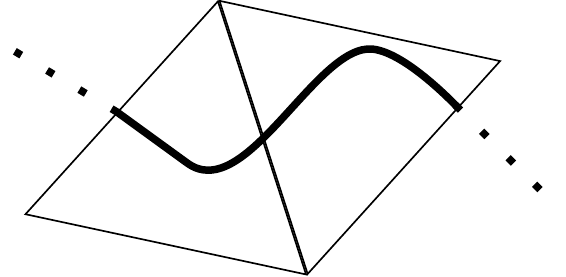} 
\caption{A segment that crosses two triangles through the opposite sides}
\label{pic17}
\end{center}
\end{figure}
\noindent This implies that:
\begin{equation*}
m_1(d)(1-\Pro{}{g_N=0}) \leq \ExV{}{\sys_N}
\end{equation*}
So:
\begin{align*}
\liminf\limits_{N\rightarrow\infty} \ExV{}{\sys_N} & \geq \liminf\limits_{N\rightarrow\infty} m_1(d)(1-\Pro{}{g_N=0})  \\
   & = m_1(d)
\end{align*}

For the upper bound proposition \ref{prp_badsurfEucl} tells us that: 
\begin{equation*}
\limsup_{N\rightarrow\infty}\ExV{}{\sys_N} = \limsup\limits_{N\rightarrow\infty} \frac{1}{\aant{\Omega_N}}\sum_{\omega\in\Omega_N-B_N}\sys_N(\omega)
\end{equation*}
We want to use our results on $m_\ell$, so we split the sum on the right hand side up over $m_\ell$ and we get:
\begin{equation*}
\limsup_{N\rightarrow\infty}\ExV{}{\sys_N} = \limsup\limits_{N\rightarrow\infty} \frac{1}{\aant{\Omega_N}} \sum\limits_{k=2}^\infty\sum_{\substack{\omega\in\Omega_N-B_N,\\ m_{\ell,N}(\omega)=k}} \sys_N(\omega)
\end{equation*}
Given $\omega\in\Omega_N$ with $m_{\ell,N}(\omega)=k$ we know that $\sys_N(\omega) \leq m_2(d)k$, so:
\begin{align*}
\limsup_{N\rightarrow\infty}\ExV{}{\sys_N} & \leq \limsup\limits_{N\rightarrow\infty} \frac{1}{\aant{\Omega_N}} \sum\limits_{k=2}^\infty\sum_{\substack{\omega\in\Omega_N-B_N,\\ m_{\ell,N}(\omega)=k}} m_2(d) k \\
   & = \limsup\limits_{N\rightarrow\infty}m_2(d)\sum\limits_{k=2}^\infty \Pro{}{\substack{\omega\in\Omega_N-B_N,\\ m_{\ell,N}(\omega)=k}}k
\end{align*}
The limit on the right hand side we can compute. We will show that $\Pro{}{\substack{\omega\in\Omega_N-B_N,\\ m_{\ell,N}(\omega)=k}}k$ is universally bounded by a summable function of $k$. This implies we can apply the dominated convergence theorem in combination with the pointwise limits that we already know from proposition \ref{prp_mell}.

To get an upper bound on $\Pro{}{\substack{\omega\in\Omega_N-B_N,\\ m_{\ell,N}(\omega)=k}}k$ we reason as follows: if $\omega\notin B_N$ and $m_{\ell,N}(\omega)=k$ then there are either no circuits of $k-1$ edges on $\omega$ or there are some of these circuits of $k-1$ that all carry a word of the type $L^{k-1}$ (again the third option would be that there is some circuit of $k-1$ edges that cuts off a disk and hence is separating, but because $\omega\notin B_N$ this is impossible). So we get:
\begin{align*}
\Pro{}{\substack{\omega\in\Omega_N-B_N,\\ m_{\ell,N}(\omega)=k}}k & \leq \Pro{}{X_{N,k-1}=0}k+\Pro{}{\substack{X_{N,k-1}>0\text{ and all }k-1\text{-circuits} \\ \text{carry a word of the type }L^{k-1}}}k \\
   & \leq D(k-1)^8 \left(\frac{3}{8}\right)^{k-1}k + \frac{2}{2^{k-1}}k
\end{align*}
for some $D\in(0,\infty)$, where we have used proposition \ref{prp_ubPX1} for the first term. This is a summable function that is independent of $N$. So the dominated convergence theorem implies that we can fill in the pointwise limits of the terms, which completes the proof. 
\end{proof}

The expression on the right hand side of Theorem C is something we can compute up to a finite number of digits. We have:
\begin{equation}\label{eq_numRiem}
\limsup_{N\rightarrow\infty}\ExV{}{\sys} \leq 2.87038 \cdot m_2(d)
\end{equation}

We already noted that in the equilateral Euclidean case we have $m_1(d)=1$ $m_2(d)=\frac{1}{2}$, so we get:
\begin{equation}\label{eq_numEucl1}
1\leq \liminf_{N\rightarrow\infty}\ExV{}{\sys_N} 
\end{equation}
and:
\begin{equation}\label{eq_numEucl2}
\limsup_{N\rightarrow\infty}\ExV{}{\sys_N} \leq 1.43519
\end{equation}
It is not difficult to see that this last inequality is not optimal. We can however construct a sequence of metrics that does come close to the upper bound in Theorem C, which we will do in the next section.

\subsection{Sharpness of the upper bound}\label{sec_Sharpness}

The goal of this section is to show that Theorem C is sharp. We have the following proposition:

\begin{prp}
For every $\varepsilon > 0$ and every $M\in (0,\infty)$ there exists a Riemannian metric \linebreak $d:\Delta\times\Delta\rightarrow [0,\infty)$ such that:
$$
m_2(d)\sum_{k=2}^\infty k\left(e^{-\sum\limits_{j=1}^{k-1}\frac{2^{j-1}-1}{j}} -e^{-\sum\limits_{j=1}^k\frac{2^{j-1}-1}{j}}\right)-\varepsilon\leq \liminf_{N\rightarrow\infty}\ExV{}{\sys_N}
$$
and:
$$
m_2(d) = M
$$
\end{prp}

\begin{proof} The idea of the proof is simply to construct the metric $d$. Recall that an element $x\in\Delta$ can be expressed as $x=t_1e_1+t_2e_2+t_3e_3$ with $t_1,t_2,t_3\in[0,1]$ and $t_1+t_2+t_3=1$. We can write $t_i=\inp{x}{e_i}$ for $i=1,2,3$, where $\inp{\cdot}{\cdot}$ denotes the inner product in $\mathbb{R}^3$.

We define the following two subsets of $\Delta$:
\begin{equation*}
\mathcal{P}_1=\verz{x\in\Delta}{\exists i\in\{1,2,3\}\text{ such that }\inp{x}{e_i}=\frac{1}{2}}
\end{equation*}
and for a given $\delta>0$:
\begin{equation*}
\mathcal{P}_2=\verz{x\in\Delta}{\exists i\in\{1,2,3\}\text{ such that }\inp{x}{e_i}=\frac{1\pm\delta}{2}}
\end{equation*}
as depicted in figure \ref{pic19} below:
\begin{figure}[H] 
\begin{center} 
\includegraphics[scale=1.2]{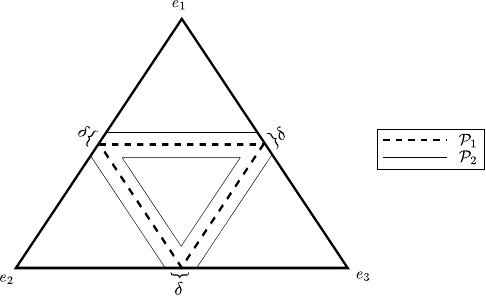} 
\caption{$\Delta$, $\mathcal{P}_1$ and $\mathcal{P}_2$}
\label{pic19}
\end{center}
\end{figure}
Furthermore, given $D\in (1,\infty)$, we construct a function $\rho_{D,\delta}:\Delta\rightarrow [1,D]$ that satisfies the following properties:
\begin{itemize}
\item[I.] $(\rho_{D,\delta})\big|_{\mathcal{P}_1} = 1$.
\item[II-a.] $\rho_{D,\delta}(x)=1$ when $\inp{x}{e_i}\leq \frac{1}{2}-\delta$ for all $i\in\{1,2,3\}$.
\item[II-b.] $\rho_{D,\delta}(x)=1$ when there exists an $i\in\{1,2,3\}$ such that $\inp{x}{e_i}\geq \frac{1}{2}+\delta$.
\item[III-a.] $\int_\gamma\rho_{D,\delta}\geq D$ when $\gamma:[0,1]\rightarrow \Delta$ is a curve such that there exists an $i\in\{1,2,3\}$ and $s_1,s_2\in[0,1]$ such that $\inp{\gamma(s_1)}{e_i} \leq \frac{1}{2}-\delta$ and $\inp{\gamma(s_2)}{e_i} \geq \frac{1}{2}$.
\item[III-b.] $\int_\gamma\rho_{D,\delta}\geq D$ when $\gamma:[0,1]\rightarrow \Delta$ is a curve such that there exists an $i\in\{1,2,3\}$ and $s_1,s_2\in[0,1]$ such that $\inp{\gamma(s_1)}{e_i} \leq \frac{1}{2}$ and $\inp{\gamma(s_2)}{e_i} \geq \frac{1}{2}+\delta$.
\item[IV.] $\rho_{D,\delta}\in C^\infty(\Delta)$ and $\frac{\partial^k}{\partial n^k}\big|_x \rho_{D,\delta} = 0$ for all $k\geq 1$, all $x\in\partial\Delta$ and all $n$ normal to $\partial\Delta$ at $x$.
\end{itemize}
Property II-a says that $\rho_{D,\delta}$ has to be equal to $1$ far enough from $\mathcal{P}_2$ inside the middle triangle in $\Delta$ and property II-b says the same about the triangles in the corners. Property III-a means that when a curve goes between the middle triangle and $\mathcal{P}_1$ then the area lying under the function $\rho_{D,\delta}$ on this curve is at least $D$. Property III-b states the equivalent for the triangles in the corner.

A candidate for such a function is the function $\tilde{\rho}_{D,\delta}:\Delta\rightarrow [1,\infty)$ given by:
\begin{equation*}
\tilde{\rho}_{D,\delta} (x) = \min\left\{D\left(1-\frac{4}{\delta}d_{\text{Eucl}}\left(x,\mathcal{P}_2\right)\right) ,1\right\}
\end{equation*}
where $d_{\text{Eucl}}$ denotes the standard Euclidean metric on $\Delta$. It is easy to see that $\tilde{\rho}_{D,\delta}$ satisfies all the properties except property IV. So if we let $\rho_{D,\delta}$ be a smoothing of $\tilde{\rho}_{D,\delta}$ we get the type of function we are looking for. Figure \ref{pic21} below shows a cross section of $\rho_{D,\delta}$ around $(t_1,t_2,t_3)=(\frac{1}{2},\frac{1}{4},\frac{1}{4})$:
\begin{figure}[H] 
\begin{center} 
\includegraphics[scale=1]{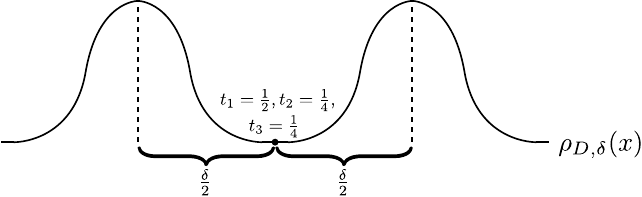} 
\caption{A cross section of $\rho_{D,\delta}$}
\label{pic21}
\end{center}
\end{figure}

$\rho_{D,\delta}$ gives us a Riemannian metric $g_{D,\delta}:T\Delta\times T\Delta\rightarrow\mathbb{R}$ by:
\begin{equation*}
g_{D,\delta} = \rho_{D,\delta} g_{\text{Eucl}}
\end{equation*}
Where $g_{\text{Eucl}}$ denotes the standard Riemannian metric on $\Delta$ that induces $d_{\text{Eucl}}$. $g_{D,\delta}$ induces a metric $d_{D,\delta}:\Delta\times\Delta\rightarrow (0,\infty)$, given by:
\begin{equation*}
d_{D,\delta}(x,y) = \inf\verz{\int_\gamma \rho_{D,\delta}}{\gamma:[0,1]\rightarrow\Delta\text{ continuous, }\gamma(0)=x,\;\gamma(1)=y}
\end{equation*}

Note that:
\begin{equation*}
m_2(d_{D,\delta})=\frac{1}{\sqrt{2}}
\end{equation*}

The idea behind the metric $d_{D,\delta}$ is that when $D$ grows it is very `expensive' to go from one of the smaller triangles to another. So the shortest paths between the different sides of the triangle avoid the region around $\mathcal{P}_2$.

Now suppose that $\gamma$ is a closed curve on a random surface, not homotopic to a corner. Then somewhere $\gamma$ must cross a quadrilatereal (i.e. two triangles glued together) through both the opposite  sides (the curve cannot turn in the same direction on every triangle, because then it would be homotopic to a corner) as in figure \ref{pic22} below:
\begin{figure}[H] 
\begin{center} 
\includegraphics[scale=1]{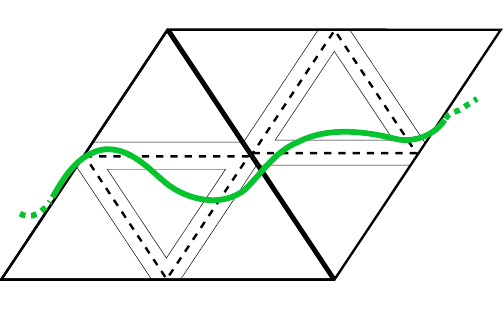} 
\caption{Two consecutive triangles that are crossed by a path.}
\label{pic22}
\end{center}
\end{figure}

It is clear that if $D$ is chosen large enough then the `cheapest' way to do this is without crossing the thin lines in the picture and staying on the dotted line. 

It is not difficult to see that the systole of a random surface equipped with the metric coming from $d_{D,\delta}$ has to be homotopic to a circuit. So if $\omega\in\Omega_N$ and $m_\ell(\omega)\leq D/m_2(d_{D,\delta})$ then the systole stays between the thin lines (so at Euclidean distance at most $\frac{\delta}{2}$ from the dotted lines). 

Because $\rho_{D,\delta}(x)\geq 1$ for all $x\in\Delta$ we have:
\begin{equation*}
\ell_{\text{Eucl}}(\gamma) \leq \ell_{D,\delta}(\gamma) 
\end{equation*}
for every curve $\gamma$ in every random surface, where $\ell_{D,\delta}$ denotes the length with respect to the metric $d_{D,\delta}$ and $\ell_{\text{Eucl}}$ denotes the length with respect to the Euclidean metric. Also note that for a curve traveling over the dotted lines we have equality.

Suppose we are given $\omega\in\Omega_N$ with $m_\ell(\omega)\leq D/m_2(d_{D,\delta})$ and we look at the systole. On every triangle the systole passes, it must travel at least the (Euclidean) length of the middle subtriangle between the dotted lines (before and after this middle subtriangle there is some possibility to `cut the corners'). This length is $\frac{1}{\sqrt{2}}-\sqrt{2}\delta=m_2(d_{D,\delta})-\sqrt{2}\delta$. So we get:
\begin{equation*}
\ExV{}{\sys_N} \geq \sum\limits_{2\leq k\leq D/m_2(d_{D,\delta})} \Pro{}{m_{\ell,N}=k}(m_2(d_{D,\delta})-\sqrt{2}\delta)k
\end{equation*}
Because the expression on the right is a finite sum we get:
\begin{align*}
\liminf\limits_{N\rightarrow\infty}\ExV{}{\sys_N} & \geq \liminf\limits_{N\rightarrow\infty}\sum\limits_{2\leq k\leq D/m_2(d_{D,\delta})} \Pro{}{m_{\ell,N}=k}(m_2(d_{D,\delta})-\sqrt{2}\delta)k \\
   & \geq (m_2(d_{D,\delta})-\sqrt{2}\delta) \sum\limits_{2\leq k\leq D/m_2(d_{D,\delta})}k\left(e^{-\sum\limits_{j=1}^{k-1}\frac{2^{j-1}-1}{j}} -e^{-\sum\limits_{j=1}^k\frac{2^{j-1}-1}{j}}\right) 
\end{align*}
We know that the sum on the right converges for $D\rightarrow\infty$, so by increasing $D$ it gets arbitrarily close to its limit. Furthermore we have chosen $\delta$ and $D$ such that they do not depend on each other, which means that we can make $\delta$ arbitrarily small. So this proves the proposition for $m_2(d)=\frac{1}{\sqrt{2}}$. To get the result for any other prescribed $m_2$ we can put a constant factor in front of $\rho_{D,\delta}$.
\end{proof}


\end{document}